\documentclass[a4paper, 11pt, headings = small, abstract]{scrartcl}
\usepackage[english]{babel}

\usepackage{amscd,amsmath,amsthm,array,hhline, enumerate, booktabs,fancybox,calc,textcomp,xcolor,graphicx,bbm,xspace,nicefrac,stmaryrd,url,arcs,listings,comment}

\usepackage[theoremfont, defaultsups=true]{newpxtext}
\usepackage[vvarbb, upint, bigdelims]{newpxmath}
\usepackage[scr=boondoxupr, scrscaled = 1.05, cal = euler, calscaled =1.05]{mathalpha}
\usepackage[scaled=0.95]{inconsolata}
\usepackage{accents}
\usepackage{dsfont}

\DeclareMathAlphabet{\mathsf}{OT1}{\sfdefault}{m}{n}
\SetMathAlphabet{\mathsf}{bold}{OT1}{\sfdefault}{b}{n}
\usepackage[utf8]{inputenc}
\usepackage{enumitem}
\usepackage{etoolbox}
\usepackage[normalem]{ulem}
\usepackage{mathtools}
\usepackage{geometry}
\usepackage{float}

\usepackage{bm}
\usepackage{tikz}
\usepackage[outdir =./]{epstopdf}
\usepackage[citestyle = numeric-comp,bibstyle = numeric, giveninits=true, natbib = true, backend = bibtex,maxnames=2,maxbibnames=99, url=false]{biblatex}
\numberwithin{equation}{section}

\usepackage{chngcntr}
\counterwithin{figure}{section}

\usepackage{xcolor}
\usepackage[colorlinks=true, allcolors = myteal]{hyperref}

\definecolor{WIMgreen}{RGB}{60 134 132}
\definecolor{UMblue}{RGB}{4 47 86}
\definecolor{myteal}{RGB}{0 123 137}
\definecolor{material_green}{RGB}{27 43 52}
\definecolor{dracula_pink}{RGB}{180 93 149}
\definecolor{dracula_blue}{RGB}{40 42 54}
\definecolor{dracula_turq}{RGB}{92 143 159}
\definecolor{dracula_orange}{RGB}{255 184 108}
\definecolor{material_petrol}{RGB}{2 119 189}
\definecolor{Purple}{RGB}{103 58 183}

\theoremstyle{plain}
\newtheorem{theorem}{Theorem}[section]
\newtheorem*{theorem*}{Theorem}
\newtheorem{proposition}[theorem]{Proposition}
\newtheorem{lemma}[theorem]{Lemma}
\newtheorem{corollary}[theorem]{Corollary}

\theoremstyle{definition}
\newtheorem{definition}[theorem]{Definition}

\theoremstyle{remark}
\newtheorem{remark}[theorem]{Remark}
\newtheorem{example}[theorem]{Example}
\setkomafont{sectioning}{\normalcolor\bfseries}
\setkomafont{descriptionlabel}{\normalcolor\bfseries}
\setkomafont{section}{\Large}
\setkomafont{subsection}{\large}
\setkomafont{subsubsection}{\large}
\setkomafont{paragraph}{\large}
\setkomafont{subparagraph}{\large}
\setkomafont{author}{\large}
\setkomafont{date}{\normalsize}

\def\C{\mathbb C}
\def\CC{\mathbb C}

\def\E{\mathbb{E}}

\def\N{\mathbb{N}}

\def\N{\mathbb{N}}

\def\R{\mathbb{R}}
\definecolor{darkred}{rgb}{0,0.6,0}

\newcommand{\PP}{\mathbb{P}}

\renewcommand{\Re}{\operatorname{Re}}

\renewcommand{\hat}{\widehat}

\renewcommand{\tilde}{\widetilde}%

\newcommand{\overbar}[1]{\mkern 1.5mu\overline{\mkern-1.5mu#1\mkern-1.5mu}\mkern 1.5mu}
\newcommand{\underbars}[1]{\mkern 1.5mu\underline{\mkern-1.5mu#1\mkern-1.5mu}\mkern 1.5mu}
\newcommand\iu{\mathrm{i}}

\restylefloat{table}

\newcommand*\diff{\mathop{}\!\mathrm{d}}

\newcommand{\one}{\mathbf{1}}

\newcommand{\vertiii}[1]{{\left\vert\kern-0.25ex\left\vert\kern-0.25ex\left\vert #1
\right\vert\kern-0.25ex\right\vert\kern-0.25ex\right\vert}}

\let\originalleft\left
\let\originalright\right
\renewcommand{\left}{\mathopen{}\mathclose\bgroup\originalleft}
\renewcommand{\right}{\aftergroup\egroup\originalright}

\newcommand{\bbGamma}{{\mathpalette\makebbGamma\relax}}
\newcommand{\makebbGamma}[2]{%
  \raisebox{\depth}{\scalebox{1}[-1]{$\mathsurround=0pt#1\mathds{L}$}}%
}

\newcommand\scalemath[2]{\scalebox{#1}{\mbox{\ensuremath{\displaystyle #2}}}}

\makeatletter
\newcommand{\specificthanks}[1]{\@fnsymbol{#1}}
\makeatother

\newcommand\blfootnote[1]{%
  \begingroup
  \renewcommand\thefootnote{}\footnote{#1}%
  \addtocounter{footnote}{-1}%
  \endgroup
}

\allowdisplaybreaks

\makeatother
\title{\fontsize{16}{19} \selectfont Markov additive friendships}
\author{Leif Döring\thanks{University of Mannheim, Institute of Mathematics, B6
26, 68159 Mannheim, Germany.\\
Email: \href{mailto:doering@uni-mannheim.de}{doering@uni-mannheim.de}}
\and Lukas Trottner\thanks{Aarhus University, Department of Mathematics, Ny Mungegade 118, 8000 Aarhus C, Denmark. \\ 
Email: \href{mailto:trottner@math.au.dk}{trottner@math.au.dk}} \and Alexander R. Watson\thanks{University
College London, UK. Email:
\href{mailto:alexander.watson@ucl.ac.uk}{alexander.watson@ucl.ac.uk}}}
\date{\today}

\bibliography{lit_overshoot}

\begin{document}
\maketitle

\begin{abstract}
  The Wiener--Hopf factorisation of a Lévy or Markov additive process describes
  the way that it attains new maxima and minima in terms of a pair of so-called
  ladder height processes. Vigon's theory of friendship for Lévy processes addresses the
  inverse problem: when does a process exist which has certain prescribed
  ladder height processes? We give a complete answer to this problem for Markov additive
  processes, provide simpler sufficient conditions for constructing processes
  using friendship, and address in part the question of the uniqueness of the
  Wiener--Hopf factorisation for Markov additive processes.%
  \blfootnote{\textit{2020 MSC}: 60G51, 60J25, 47A68.}
\end{abstract}

\section{Introduction}

Lévy processes and Markov additive processes are a staple of applied
probability and have found a home in areas as diverse as queueing theory,
stochastic finance and fragmentation modelling. For many applications, it is
beneficial to know how the processes make new maxima (or minima), and a key
tool for Lévy processes is the theory of friendship, which makes it possible to
build models which cross levels in a prescribed way. In this article, we
approach the theory of friendship for Markov additive processes, beginning with
a review of the situation for Lévy processes, before exploring how this changes
with the introduction of a Markov component.

The Wiener--Hopf factorisation is one of the central results in fluctuation
theory for Lévy processes. Its spatial version tells us that, for a Lévy
process $\xi$ with characteristic exponent 
$\psi(\theta) \coloneqq \log \E[\mathrm{e}^{\mathrm{i}\theta \xi_1}]$, we have the identity
\begin{equation} \label{eq: wh levy}
  \psi(\theta) = -c\psi^-(-\theta)\psi^+(\theta), \quad \theta \in \R,
\end{equation}
where the functions $\psi^\pm$ are the characteristic exponents of the
ascending and descending ladder height processes $H^\pm$.  The processes $H^+$
and $H^-$ are (killed) subordinators whose ranges are the set of new suprema
and infima of $\xi$, respectively, and $c>0$ is a constant whose value
influences the time scale of $H^\pm$.  We refer to \cite[Chapter
6]{kyprianou2014} for details. $H^+$ and $H^-$ are of central importance for
both theoretical and practical considerations since they are the building block
for first passage identities of $\xi$ \cite{DoneyKyprianou2006}.  However, for
any given Lévy process $\xi$, the factorisation \eqref{eq: wh levy}, and the
identities derived from it, are in most cases not explicit.  Vigon flipped the
perspective on the Wiener--Hopf factorisation in his pioneering thesis
\cite{vigondiss}. Instead of considering some fixed Lévy process $\xi$ and
looking for its Wiener--Hopf factors, Vigon starts with two subordinators $H^+$
and $H^-$ and finds necessary and sufficient criteria on their characteristics
such that the right hand side of \eqref{eq: wh levy} determines the
Lévy--Khintchine exponent of some Lévy process $\xi$. Vigon calls such
subordinators \textit{friends} (\textit{amis}) and refers to the resulting Lévy
process $\xi$ as \textit{le fruit de l'amitié}, which we translate (more
conservatively) as the \textit{bonding process}. 

We recall his results explicitly at this point. Let $H^+$ and $H^-$ be two Lévy
subordinators with drifts and Lévy measures $(d^+,\Pi^+)$ and $(d^-,\Pi^-)$,
respectively. Let moreover  $\dagger^\pm = -\psi^\pm(0)$ be their respective
killing rates and denote by $\overbar{\Pi}{}^\pm(x) = \Pi^\pm(x,\infty)$ the
tails of $\Pi^\pm$, where $x > 0$. We call $H^+$ and $H^-$ \textit{compatible}
if $d^{\mp} > 0$  implies that $\Pi^\pm$ has a càdlàg Lebesgue density
$\uppartial \Pi^\pm$  on $(0,\infty)$, which can be expressed as the right tail
of a signed measure $\eta^\pm$, i.e., $\uppartial \Pi^\pm(x) =
\eta^\pm(x,\infty)$ for $x > 0$. For the sake of notational consistency, if
$d^\mp = 0$  we let  $\uppartial \Pi^\pm$ be some version of the Lebesgue
density of the absolutely continuous part of $\Pi^+_i$. Then, we have the
following result.

\begin{theorem}[Vigon's theorem of friends] \label{theo: friends levy}
  $H^+$ and $H^-$ are friends if, and only if, they are compatible and the function 
  \begin{align*}
    \Upsilon(x) &= \one_{(0,\infty)}(x) \Big(\int_{x+}^\infty \big(\overbar{\Pi}^-(y-x) - \psi^-(0)\big) \, \Pi^+(\diff{y}) + d^- \uppartial \Pi^+(x) \Big) \\
    &\quad + \one_{(-\infty,0)}(x) \Big(\int_{(-x)+}^\infty \big(\overbar{\Pi}^+(y+x) -\psi^+(0)\big) \, \Pi^-(\diff{y}) + d^+ \uppartial \Pi^-(-x) \Big), \quad x \in \R,
  \end{align*}
  is a.e.\ equal to a function decreasing on $(0,\infty)$ and increasing on $(-\infty,0)$. Moreover, when $H^+$ and $H^-$ are friends, we have for a.e.\ $x \in \R$ the identity
  \begin{equation} \label{eq: ami levy}
    \one_{(0,\infty)}(x)\Pi(x,\infty) + \one_{(-\infty,0)}(x)\Pi(-\infty,x) = \Upsilon(x),
  \end{equation}
  for the Lévy measure $\Pi$ of the bonding Lévy $\xi$.
\end{theorem}
\begin{remark} 
  In Vigon's original formulation, \eqref{eq: ami levy} holds everywhere rather
  than almost everywhere, the reason being that it is claimed that the
  convolution $x \mapsto \overbar{\tilde{\Pi}}{}^- \ast \Pi^+(x)$ is a càdlàg
  function.  However, this fails, for example, when $H^\pm$ are pure jump
  processes with jumps of size $\{1,2\}$ and 
  $\Pi^\pm(\{1\}) > \Pi^\pm(\{2\}) > 0$,
  even though these processes can be shown to be friends (they are
  \textit{philanthopes discrètes} \cite{vigondiss}).  In this example, the
  statement can be easily  repaired by considering the closed tails
  $\Pi^\pm([x,\infty))$ instead of the open tails 
  $\overbar{\Pi}^\pm(x) = \Pi^\pm((x,\infty))$,
  but it is not clear that such an approach can work in
  general, for instance, when one of the Lévy measures is singular continuous.
  We arrived at the above statement of the result after consultation with
  Vigon, but we remark that the original formulation remains valid when
  $\Pi^\pm$ are absolutely continuous. 
\end{remark}

Equation \eqref{eq: ami levy} is called \textit{équation amicale} by Vigon and
characterises the Lévy measure of $\xi$ in terms of the characteristics of its
ascending and descending ladder height processes.  Aside from \cite{vigondiss},
this equation is proved in \cite[Proposition~3.3]{vigon2002} and
\cite[Theorem~16]{doney2007}. A textbook treatment of the theorem of friends
appears in \cite[Theorem~6.22]{kyprianou2014}.

The monotonicity conditions on $\Upsilon$ do not appear easy to handle at first
sight.  However, they simplify dramatically for a certain class of processes.
Vigon calls a subordinator with decreasing Lévy density a
\textit{philanthopist}, and develops the following neat result:
\begin{theorem}[Vigon's theorem of philanthropy] \label{theo: vigon philan}
  Two philanthropists are always friends.
\end{theorem}
For a direct proof of this result under under additional second moment
assumptions we refer to \cite[Theorem 4.4]{kyprianou22}.  Let us also remark
that the class of Laplace exponents of philanthropists (or equivalently,
Bernstein functions with decreasing Lévy density) is known as the \textit{Jurek
class} of Bernstein functions \cite{schilling2012}. 

The class of philanthropists is quite large and contains many tractable
examples.  As a consequence, the theorem of philanthropy offers a new tool for
the construction of Lévy processes, quite different from the classical
techniques of specifying the Lévy triplet, transition semigroup or
characteristic exponent.  However, the form of the équation amicale means that
the Lévy measure of the bonding process may not be simple to express, and there
is a tradeoff between explicitness of the bonding process and explicitness of
the corresponding friends, regardless of which side of the Wiener--Hopf
factorisation we start on.  One of the major achievements of this approach is
the class of \textit{hypergeometric Lévy processes}, built from friendships of
$\beta$-subordinators \cite{kyprianou2010,KuznetsovPardo2013,kyprianou2014} and
motivated by the relation with killed and conditioned stable processes (see
\cite{Caballero2011} and \cite[Theorem~1]{KuznetsovPardo2013}.)

Coming from this well-established theory for Lévy processes, our goal in this
paper is to extend Vigon's theory of friends and philanthropy to the class of
Markov additive processes (MAPs).  A MAP $(\xi,J)$ with state space
$\R \times \{1,\ldots,n\}$
can be thought of as a regime-switching Lévy process: depending
on the state (or phase) of a Markov chain $J$, the process $\xi$ follows a
different Lévy process; see Section~\ref{sec: map intro} for a full
description.  Many concepts in Lévy process theory have direct analogues in the
theory of MAPs.  For any MAP $(\xi,J)$ there exists a matrix form of the
characteristic exponent, which we call the MAP exponent,
$\bm{\Psi}\colon \R \to \mathbb{C}^{n\times n}$, such that 
\[\big(\E^{0,i}[\exp(\mathrm{i} \theta \xi_t);\,  J_t = j] \big)_{i,j \in [n]} = \mathrm{e}^{t\bm{\Psi}(\theta)}, \quad \theta \in \R, t \geq 0,\]
where $\PP^{x,i}$ indicates that the process $(\xi,J)$ is a.s.\ started in 
$(x,i) \in \R \times [n]$. 
The Lévy measure $\Pi$ of a Lévy process $\xi$ has a natural analogue 
in the \textit{Lévy measure matrix} $\bm{\Pi}$ of a MAP $(\xi,J)$, which 
describes the jump structure of $\xi$ in the  Markovian environment 
governed by $J$. 
Much as in the Lévy case, the new suprema of $\xi$ can be related to a
MAP subordinator $(H^+,J^+)$, 
referred to as the \textit{ascending ladder height MAP};
likewise, the new infima can be related to the
\textit{descending ladder height MAP} $(H^-,J^-)$.
Again, we make these statements precise in Section \ref{sec: map intro}. 

Our question arises from considering the Wiener--Hopf factorisation
of the MAP exponent $\bm{\Psi}$ into
the MAP exponents $\bm{\Psi}^+,\bm{\Psi}^-$ of $(H^+,J^+)$ and $(H^-,J^-)$,
respectively, which was first proved, under certain constraints, 
in \cite[Theorem 26]{dereich2017} and \cite[equation~(19)]{ivanovs17}.
This states that, if $\xi$ is either killed at the same (possibly zero)
rate in all phases or is killed with positive rate in every phase,
$\xi$ is nonlattice and $J$ is irreducible with stationary distribution $\bm{\pi}$, 
then, for an appropriate time scaling of $H^+$ and $H^-$, we have
the matrix identity
\begin{equation} \label{eq: wh map}
  \bm{\Psi}(\theta) = - \bm{\Delta}_{\bm{\pi}}^{-1} \bm{\Psi}^-(-\theta)^\top \bm{\Delta_{\bm{\pi}}} \bm{\Psi}^+(\theta), \quad \theta \in \R,
\end{equation}
where $\bm{\Delta}_{\bm{\pi}}$ is the diagonal matrix with entries
given by the vector $\bm{\pi}$.
Notice here that when the modulating space is one-dimensional, i.e., $J_t =1$ for all $t \geq 0$, the above equality reduces to \eqref{eq: wh levy}. We will show in Theorem \ref{c:whf-norm} that \eqref{eq: wh map} holds for any irreducible MAP, irrespectively of the lifetimes of the Lévy components.

Even when compared to the already quite scarce number of explicit Wiener--Hopf
Lévy factorisations, the situation for MAPs is even more tenuous. Apart from
the \textit{deep factorisation} of the stable process in \cite{kyprianou2016},
to the best of our knowledge there is no known example of an explicit MAP
Wiener--Hopf factorisation. As will become apparent from our analysis this is
not a mere artifact of the youth of the MAP Wiener--Hopf
factorisation, but also a consequence of the increased complexity
due to phase transitions.

Despite these difficulties, the theorem below, which is our main result,
provides a complete picture of friendship of MAPs. This is established
in Section~\ref{sec: map friends} in the form of Theorems~\ref{theo: eq amicales}
and~\ref{theo: friends}. 
\begin{theorem} \label{theo: main}
  Two MAP subordinators $(H^+,J^+)$  and $(H^-,J^-)$ are $\bm{\pi}$-friends 
  (in the sense of Definition~\ref{def: friends})
  if,
  and only if, they are $\bm{\pi}$-compatible (in the sense of Definition~\ref{def: comp})
  and  the matrix-valued function 
  \begin{align*}
    \bm{\Upsilon}(x) &= \Big(\int_{x+}^\infty \bm{\Delta}_{\bm{\pi}}^{-1}\Big(\overbar{\bm{\Pi}}^-(y-x) - \bm{\Psi}^-(0)\Big)^\top \bm{\Delta}_{\bm{\pi}} \, \bm{\Pi}^+(\diff{y})  + \bm{\Delta}^-_{\bm{d}} \uppartial \bm{\Pi}^+(x)\Big) \one_{(0,\infty)}(x)\\
    &\quad + \Big(\int_{(-x)+}^\infty \bm{\Delta}_{\bm{\pi}}^{-1} \big(\bm{\Pi}^-(\diff{y}) \big)^\top \bm{\Delta}_{\bm{\pi}} \, \big(\overbar{\bm{\Pi}}{}^+(y+x) - \bm{\Psi}^+(0)\big)  +  \bm{\Delta}_{\bm{\pi}}^{-1} \big(\bm{\Delta}_{\bm{d}}^+\uppartial\bm{\Pi}^-(-x)\big)^\top \bm{\Delta}_{\bm{\pi}}\Big) \one_{(-\infty,0)}(x),
  \end{align*}
  is a.e.\ equal to a function 
  decreasing on $(0, \infty)$ and increasing on $(-\infty,0)$. 
  Moreover, when $(H^+,J^+)$ is a $\bm{\pi}$-friend of $(H^-,J^-)$, then for a.e.\ $x \in \R,$
  \[\one_{(0,\infty)}(x)\bm{\Pi}(x,\infty) + \one_{(-\infty,0)}(x) \bm{\Pi}(-\infty,x) = \bm{\Upsilon}(x),\]
  for the Lévy measure matrix of the bonding MAP $(\xi,J)$.
\end{theorem}
The notions of $\bm{\pi}$-friendship and $\bm{\pi}$-compatibility, which we
have not yet defined, are made precise in Section~\ref{sec: map friends}.
$\bm{\pi}$-friendship is the obvious counterpart to friendship of Lévy
processes, meaning that the matrix Wiener--Hopf factorisation 
\eqref{eq: wh map} holds.  Meanwhile, $\bm{\pi}$-compatibility is partly the analogue of
Vigon's compatibility of Lévy processes, but also places more stringent
requirements on the Lévy measure matrices at and near zero and conditions on
the rates of the Markov chains.

The second part of our work seeks an extension of Vigon's theory of philanthropy;
that is, sufficient conditions that allow one to more easily establish that
two MAP subordinators are friends. The additional challenges of $\bm{\pi}$-friendship
make it difficult to find `unilateral' conditions, which can be verified
separately for each of two subordinators and lead to $\bm{\pi}$-friendship
between them. Instead, we introduce the notion of $\bm{\pi}$-fellowship,
and Theorem~\ref{theo: philan} states that two mutual $\bm{\pi}$-fellows
are indeed $\bm{\pi}$-friends. Although this theory is harder to apply
than in the Lévy case, we use it to give general
sufficient conditions for mutual $\bm{\pi}$-fellowship
(Theorems~\ref{theo: friend drift} and~\ref{theo: quasi help})
and use
these to give explicit examples of
spectrally positive MAPs with completely monotone ascending ladder
jump measures, as well as a class of MAPs with double
exponential jump structure and known Wiener--Hopf factorisation. 

Finally, since our approach is to analyse the MAP Wiener--Hopf factorisation 
equation \eqref{eq: wh map}, we need to know about the uniqueness of this
in order to draw probabilistic conclusions, and in Section~\ref{sec: unique}
we prove this under a wide range of assumptions, the main result
being Theorem~\ref{theo:unique}.

The remainder of the paper is structured as follows.
Section \ref{sec: map intro} is devoted to introducing notation and the most
important facts on MAPs that we need in the rest of the paper. Section
\ref{sec: map friends} deals with friendship of MAPs and general consequences
thereof, while in Section \ref{sec: map philan} we explore
$\bm{\pi}$-fellowship and develop constructive criteria for
$\bm{\pi}$-compatibility in order to generate examples of
$\bm{\pi}$-friendship, and thereby of MAPs with explicit Wiener--Hopf factorisation.
The final Section \ref{sec: unique} discusses the uniqueness of the matrix
Wiener--Hopf factorisation, and verifies that this applies to the examples
found in previous sections.

\paragraph{Notation}
We collect some notational conventions that are used throughout the paper.  Small bold letters $\bm{a}$ represent column vectors in $\R^n$ and capitalized letters $\bm{A}$ refer to matrices in $\R^{n\times n}$. The vectors containing only zeros, respectively ones, are denoted by $\bm{0}$ and $\bm{1}$, respectively. For measures $\mu^\pm$ concentrated on $\R_{\pm} \setminus\{0\}$, where $\R_+ \coloneqq [0,\infty), \R_- \coloneqq (-\infty,0]$, we let $\overbar{\mu}^+(x) \coloneqq \mu^+((x,\infty)) \equiv \mu(x,\infty)$ be the right tail of $\mu^+$ for $x > 0$ and $\overbar{\mu}^-(x) \coloneqq \mu^-((-\infty,x)) \equiv \mu^-(-\infty,x)$ be the left tail of $\mu^-$ for $x < 0$. For a (signed) finite measure $\mu$ on $\R$, its distribution function is represented by $\underbars{\mu}(x) \coloneqq \mu((-\infty,x])$, $x\in \R$. Moreover, for a given (signed) measure $\mu$ on $\R$, $\tilde{\mu}(\diff{x}) \coloneqq \mu(-\diff{x})$ denotes the reflected measure and for a function $f\colon \R \to \R$ we let $\tilde{f} \coloneqq f(-\cdot)$. Then, given a finite measure $\mu$, we interpret $\tilde{\underbars{\mu}}$ in the sense of reflecting the distribution function of $\mu$, i.e., $\tilde{\underbars{\mu}}(x) = \underbars{\mu}(-x)$.  For $n\in \N$ we let $[n] \coloneqq \{1,\dotsc,n\}$.

\section{Fundamentals on Markov additive processes}\label{sec: map intro}
A \emph{Markov additive process (MAP)} with $n$ phases
is a strong Markov process 
$(\xi,J) = (\xi_t,J_t)_{t\ge 0}$ with state space $\R \times [n]$
and lifetime $\zeta \in (0,\infty]$,
with probability measures $(\PP^{x,i})_{x\in \R,i\in [n]}$ and
associated expectations $(\E^{x,i})_{x\in \R, i\in [n]}$,
adapted to a filtration $\mathcal{F} = (\mathcal{F}_t)_{t\ge 0}$ satisfying the usual hypotheses,
with the property that
\[
  \E^{x,i}
  \bigl[ f(\xi_{t+s} - \xi_t, J_{t+s}) \one_{\{\zeta>t\}} \mathbin{\big|} \mathcal{F}_t \bigr]
  = \E^{0,J_t} \bigl[ f(\xi_s, J_{s})\bigr] \one_{\{\zeta>t\}},
  \quad s,t\ge 0,
\]
for all bounded measurable $f$.
The process $\xi$ is called the \emph{ordinator} of the MAP, and the process
$J$ the \emph{modulator}. Any such MAP can be decomposed into a Markov process
on $[n]$ with generator matrix $\bm{Q} = (q_{i,j})_{i,j\in [n]}$
and killing rate $\dag_i$ in state $i$,
a collection of Lévy processes $(\xi^{(i)})_{i\in [n]}$, and a
collection of probability distributions $(F_{i,j})_{i,j\in [n]}$
with the convention $F_{i,i} = \delta_0$.
Roughly speaking, the MAP evolves as follows.
$J$ evolves as a Markov process with generator matrix $\bm{Q}$.
When $J_t = i$, $\xi_t$ evolves as $\xi^{(i)}$, and the process
is killed at rate $\dag_i$.
When $J$ jumps to state $j$, which occurs with rate $q_{i,j}$,
$\xi$ experiences a jump
whose distribution is $F_{i,j}$, and the evolution of $\xi$
then begins to follow $\xi^{(j)}$.

The distribution of a MAP can be described more formally using the \emph{MAP
exponent} of $(\xi,J)$, which is a matrix-valued function $\bm{\Psi}$
such that
\[
  (\mathrm{e}^{t\bm{\Psi}(\theta)})_{i,j}
  = \E^{0,i}\bigl[ \mathrm{e}^{\iu\theta \xi_t} ; J_t = j\bigr].
\]
The preceding description of the MAP then gives rise to the structure
\[
  \bm{\Psi}(\theta)
  = \bm{\Delta}_{\bm{\psi}(\theta)} + \bm{Q} \odot \bm{G}(\theta)
  - \bm{\Delta}_{\bm{\dag}},
\]
where $\bm{\Delta}_{\bm{v}}$ is the diagonal matrix whose $(i,i)$th
entry is $v_i$; $\psi_i$ is the characteristic exponent of the
Lévy process $\xi^{(i)}$, satisfying 
$\mathrm{e}^{t\psi_i(\theta)} = \E^0[\mathrm{e}^{\iu\theta \xi^{(i)}_t}]$;
$G_{i,j}(\theta) = \int \mathrm{e}^{\iu\theta x} \, F_{i,j}(\diff{x})$;
and $\odot$ is the Hadamard (elementwise) product.

This decomposition reveals that we can view $(\xi,J)$ either as
a killed Markov process $J$ on top of which we run unkilled Lévy
processes, or as an unkilled Markov process $J$ on which we run
killed Lévy processes, killing the whole MAP when
one of the component processes dies. We will typically
take the former perspective.

We remark that, given only $\bm{\Psi}$, one can readily obtain
$\bm{\dag} = - \bm{\Psi}(0) \one$ and
$\bm{Q} = \bm{\Psi}(0) + \bm{\Delta}_{\bm{\dag}}$.

If $\bm{\pi} \in \R^n$ represents a probability distribution
on $[n]$ with full support,
in that $\pi(i) > 0$ for all $i$ and $\sum_{i=1}^n \pi(i) = 1$,
then we say that $\bm{\pi}$ is \emph{invariant} for $J$ if
\[
  \bm{\pi}^\top \bm{Q} = \bm{0}^\top;
\]
that is, $\bm{\pi}$ is an invariant distribution for an unkilled version
of $J$. Under this condition, we can speak about $\bm{\pi}$-duality
  for the MAP
\cite[section A.2]{dereich2017}.
The function
\begin{equation*}
  \hat{\bm{\Psi}}(\theta) = \bm{\Delta}_{\bm{\pi}}^{-1} \bm{\Psi}(-\theta)^\top
  \bm{\Delta}_{\bm{\pi}},
  \quad \theta \in \R,
\end{equation*}
is the MAP exponent of some MAP $(\hat{\xi},\hat{J})$ which can be obtained
by time-reversal:
\[
  \E^{0,\bm{\pi}}
  \bigl[ f( \xi_{(t-s)-} - \xi_t, J_{(t-s)-}; s \le t) \one_{\{\zeta>t\}} \bigr]
  = \E^{0,\bm{\pi}} 
  \bigl[ f( \hat{\xi}_s, \hat{J}_s; s \le t) \one_{\{\hat{\zeta}>t\}} \bigr]
\]
holds for any functional $f$ and $t\ge 0$;
here, $\PP^{0,\bm{\pi}} = \sum_{i=1}^n \pi(i) \PP^{0,i}$.
Moreover,
$\bm{\pi}$ is invariant for $\hat{J}$, and $\hat{J}$ is also
killed with rates $\bm{\dag}$.
We say that $(\hat{\xi},\hat{J})$ is the \emph{$\bm{\pi}$-dual} of $(\xi,J)$.

A key aspect of the theory of MAPs is the \emph{Wiener--Hopf factorisation}, 
which expresses the characteristics of a MAP in terms of its ladder processes.
We begin with the 
\emph{local time at the supremum}, which we define by stitching
together the local times of the constituent Lévy processes.
During a given phase $i$, if $\xi^{(i)}$ is such that $0$ is regular for $(0,\infty)$,
it is a continuous increasing functional $L$
which increases precisely at the time at which $\xi$ is at its running supremum.
On the other hand, if $\xi^{(i)}$ is such that $0$ is irregular for $(0,\infty)$,
it is an increasing jump process $L$ whose jumps occur at the isolated times
at which $\xi$ makes new suprema (in the strict sense), and whose jump sizes
are independent with a standard exponential distribution.
At a phase switch occuring at time $T$, if $\xi_{T-}< \xi_T = \sup_{t\le T} \xi_t$
and $\xi^{(J_T)}$ is such that $0$ is irregular for $(0,\infty)$,
  it is necessary to introduce an additional jump of $L$
with standard exponential distribution at time $T$.

The \emph{inverse local time at the supremum} is defined by
$L^{-1}_t = \sup\{ s \ge 0: L_s > t \}$, and the phase
at this time is $J^+_t = J_{L^{-1}_t}$.
If we further let $\bm{L}^{-1}_t$
represent the vector whose $i$-th entry is $\int_0^t \one_{\{J^+_s = i\}}\, \diff{s}$,
and define $H^+_t = \xi_{L^{-1}_t}$, then both
$(\bm{L}^{-1}, H^+, J^+)$ and $(H^+, J^+)$ are MAPs,
the former having an $(n+1)$-dimensional ordinator.
These two MAPs are called the \emph{ascending ladder process} and the
\emph{ascending ladder height process}, respectively.

The descending ladder processes may be defined similarly by considering
the local time at the infimum, and we write $(H^-,J^-)$ for the descending
ladder height process.
Some care is required here
when some components of the MAP are compound
Poisson processes. We adopt the convention that the ascending ladder process
records strict new suprema, and the descending ladder process weak new infima;
see \cite[section 6.2]{kyprianou2014} for a discussion of this distinction in the setting
of Lévy processes.
Likewise, the case where a phase switch occurs at the time of a new
infimum must be handled correctly. For a further discussion of this,
we refer to \cite[section~5.1]{ivanovs17}.

Denote the MAP exponent of $(H^+,J^+)$ by $\bm{\Psi}^+$ and that of
$(H^-,J^-)$ by $\bm{\Psi}^-$.
The local times at both the maximum and the minimum are only uniquely
defined up to a multiplicative constant (which may differ in each phase).
Changing normalisation amounts to multiplying $\bm{\Psi}^+$
or $\bm{\Psi}^-$ by a diagonal matrix with strictly positive diagonal
entries.
The Wiener--Hopf factorisation can be expressed
in the following result, which is really a corollary of 
the more general Theorem~\ref{t:whf-gen} proven 
in Appendix~\ref{s:whf}.
\begin{theorem}\label{c:whf-norm}
  For suitable normalisation of the local times at the maximum and minimum,
  \[
    -\bm{\Psi}(\theta) 
    = \bm{\Delta}_{\bm{\pi}}^{-1} 
    \bm{\Psi}^-(-\theta)^\top 
    \bm{\Delta_{\bm{\pi}}} 
    \bm{\Psi}^+(\theta), 
    \quad \theta \in \R.
  \]
\end{theorem}

\medskip 

Our analytical approach to study the MAP Wiener--Hopf factorisation fundamentally relies on the theory developed by Vigon in \cite{vigondiss,vigon2002}, which allows to translate the Wiener--Hopf factorisation of a characteristic Lévy exponent into a convolution identity on the space of tempered distributions. Denote by $\mathcal{S}(\R)$  the usual Schwartz space of rapidly decreasing functions $\varphi\colon\R \to \mathbb{C}$. Tempered distributions are the elements of its dual space $\mathcal{S}^\prime(\R)$. We define the Fourier transform of a function $\varphi \in \mathcal{S}(\R)$ by $\mathscr{F}\varphi(x) \coloneqq \int_{\R} \varphi(y) \mathrm{e}^{\mathrm{i}xy} \diff{y}$, $x \in \R$. The Fourier transform is an isomporphism on $\mathcal{S}(\R)$, which is extended on the space of tempered distributions via the duality relation $\langle \mathscr{F}T,\varphi\rangle = \langle T, \mathscr{F}\varphi\rangle$, $\varphi \in \mathcal{S}(\R)$, for a given tempered distribution $T \in \mathcal{S}^\prime(\R)$. We also recall that for $T \in \mathcal{S}^\prime(\R)$, its distributional derivative is specified by $\langle T^\prime, \varphi \rangle = - \langle T, \varphi^\prime \rangle$ for $\varphi \in \mathcal{S}(\R)$.

To make the connection to the MAP Wiener--Hopf factorisation, let $\psi$ be the characteristic exponent of a Lévy process with characteristic triplet $(a,\sigma,\Pi)$ and killing rate $\dagger$ and $\kappa$ be the characteristic exponent of a finite variation Lévy process with drift $d$, Lévy measure $\Lambda$ and killing rate $\overbar{\dagger}$, i.e., for $x \in \R$,
\[\psi(x) = -\dagger + \mathrm{i}ax - \frac{\sigma^2}{2}x^2 + \int_{\R} (\mathrm{e}^{\mathrm{i}xy} - 1 - \one_{[-1,1]}(y)\mathrm{i}xy) \,\Pi(\diff{y}), \quad \kappa(x) = -\overbar{\dagger} + \mathrm{i}d x + \int_{\R} (\mathrm{e}^{\mathrm{i}xy} -1) \, \Lambda(\diff{y}).\] 
By the integrability properties of $\Pi$ and $\Lambda$, for any $\varphi \in \mathcal{S}(\R)$, the compensated integrals
\[\langle \bbGamma^2 \Pi,\varphi \rangle \coloneqq \int_{\R} (\varphi(x) - 1 - \one_{[-1,1]}(x)\varphi^\prime(x)) \, \Pi(\diff{x}), \quad \langle \bbGamma \Lambda,\varphi \rangle \coloneqq \int_{\R}(\varphi(x) - 1) \, \Lambda(\diff{x}),\] 
are well defined. In fact, these relations define tempered distributions $\bbGamma^2 \Pi \in \mathcal{S}^\prime(\R)$ and $\bbGamma \Lambda \in \mathcal{S}^\prime(\R)$, which now allow us to express $\psi,\kappa$ (interpreted as tempered distributions induced by the slowly growing functions $\psi,\kappa$ via $\langle \psi, \varphi \rangle \coloneqq \int \psi \varphi$ and $\langle \kappa,\varphi \rangle \coloneqq \int \kappa \varphi$) as Fourier transforms of tempered distributions:
\begin{equation}\label{eq:levy_dist}
\mathscr{F}\Big\{-\dagger \delta - a \delta^\prime +\frac{\sigma^2}{2} \delta^{\prime\prime} + \bbGamma^2 \Pi\Big\} = \psi, \quad \mathscr{F}\Big\{-\overbar{\dagger} \delta - d \delta^\prime + \bbGamma \Lambda\Big\} = \kappa.
\end{equation}
The elements of the matrix product in the rhs of \eqref{eq: wh map} are given as  linear combinations of products $fg$ of (i) a subordinator exponent $f$ with a negative subordinator exponent $g$, or (ii) a (negative or positive) subordinator exponent $f$ and the Fourier transform of a finite measure $g = \mathscr{F}\mu$, or (iii) Fourier transforms of finite measures. Since all of these factors are locally bounded and have polynomial growth, their products again induce a tempered distribution, which by \eqref{eq:levy_dist} is given by the Fourier transform of two tempered distributions.  Moreover, for any characteristic Lévy exponent $\psi$, by splitting $\bbGamma^2 \Pi$ into its restrictions on $[-1,1]$  and $[-1,1]^{\mathrm{c}}$, it is clear from \eqref{eq:levy_dist} that the tempered distribution $\mathscr{F}^{-1}\psi$ can be split into the sum of a tempered distribution with compact support $[-1,1]$ and a tempered distribution induced by the finite measure $\Pi\vert_{[-1,1]^{\mathrm{c}}}$, see also \cite[Proprieté 3.9]{vigon2002}. 

Consequently, standard convolution theorems for tempered distributions show that for factors $f,g$ as above with $f = \mathscr{F}T, g = \mathscr{F}S$ and $S,T$ as in \eqref{eq:levy_dist} or induced by a finite measure it holds that $\mathscr{F}^{-1}(fg) = S \ast T$, with the usual definition of convolutions of appropriate tempered distributions. Thus, by taking inverse Fourier transforms, the MAP Wiener--Hopf factorisation can be interpreted component wise in the sense of equalities of tempered distributions associated either to characteristic Lévy exponents or finite measures with a linear combination of convolved tempered distributions having factors of these types.

\section{Friendship of MAPs} \label{sec: map friends}

The matrix Wiener--Hopf factorisation \eqref{eq: wh map} motivates the following definition.
\begin{definition}\label{def: friends}
  Let $\bm{\pi} \in (0,1]^n$ be a stochastic vector and $(H^+,J^+)$ and $(H^-,J^-)$ be Markov additive subordinators. We say that $(H^+,J^+)$ is 
  a \emph{$\bm{\pi}$-friend}
  of $(H^-, J^-)$ if there exists a Markov additive process $(\xi,J)$, such that 
  \begin{equation}\label{eq: wienerhopf}
    \bm{\Psi}(\theta) = - \bm{\Delta}^{-1}_{\bm{\pi}}\bm{\Psi}^-(-\theta)^\top \bm{\Delta}_{\bm{\pi}} \bm{\Psi}^+(\theta),\quad \theta \in \R,
  \end{equation}
  and $\bm{\pi}^\top\bm{\Psi}(0) \le \mathbf{0}^\top$.
  Here, $\bm{\Delta}_{\bm{\pi}} = \mathrm{diag}(\bm{\pi})$ and $\bm{\Psi}, \bm{\Psi}^+$ and $\bm{\Psi}^-$ are the MAP exponents of $(\xi,J)$, $(H^+,J^+)$ and $(H^-,J^-)$, respectively. In this case, we call $(\xi,J)$ 
  the \emph{bonding MAP} of the $\bm{\pi}$-friends $(H^+,J^+)$ with $(H^-,J^-)$.
\end{definition}

While the requirement \eqref{eq: wienerhopf} is motivated by the matrix Wiener--Hopf factorisation, the additional condition $\bm{\pi}^\top \bm{\Psi}(0) \leq \bm{0}^\top$ makes sure that $\bm{\pi}$ is a valid candidate for an invariant distribution of the modulator $J$. 
This also entails that, as suggested by the name, $\bm{\pi}$-friendship is a symmetric relation:
\begin{proposition}
  \label{prop:dual-friends}
  $(H^+,J^+)$ is a $\bm{\pi}$-friend of $(H^-,J^-)$ if, and only if,
  $(H^-,J^-)$ is a $\bm{\pi}$-friend of $(H^+,J^+)$.
\end{proposition}
\begin{proof}
  Consider the matrix function
  \begin{equation}\label{eq:duality}
    \hat{\bm{\Psi}}(\theta) = \bm{\Delta}_{\bm{\pi}}^{-1} \bm{\Psi}(-\theta)^\top
    \bm{\Delta}_{\bm{\pi}}
    = - \bm{\Delta}_{\bm{\pi}}^{-1} \bm{\Psi}^+(-\theta)^\top \bm{\Delta}_{\bm{\pi}}
    \bm{\Psi}^-(\theta).
  \end{equation}
  This function meets the conditions of Lemma~\ref{lemma: map exp}. In particular,
  the inequality $\bm{\pi}^\top\bm{\Psi}(0) \le \bm{0}^\top$ ensures that condition \ref{cond3} of
  that lemma is
  satisfied for $\hat{\bm{\Psi}}$; and in turn, the condition $\bm{\pi}^\top\hat{\bm{\Psi}}(0)
  \le \bm{0}^\top$ is implied by \ref{cond3} for $\bm{\Psi}$.
\end{proof}

Proposition~\ref{prop:dual-friends} says that the bonding MAP of
$(H^-,J^-)$ with $(H^+,J^+)$ is equal to the $\bm{\pi}$-dual of the
bonding MAP of $(H^+,J^+)$ with $(H^-,J^-)$.
As a consequence of this result, we will often refer to two MAPs as being
\textit{$\bm{\pi}$-friends}, rather than imposing an order.

Our goal in this section is two-fold:
\begin{enumerate}[label=(\arabic*), ref=(\arabic*)]
  \item given two $\bm{\pi}$-friends, express the Lévy measure matrix of the bonding MAP in terms of $\bm{\pi}$ and the Lévy measure matrices of the friends;
  \item find necessary and sufficient criteria for two MAP subordinators to be $\bm{\pi}$-friends.
\end{enumerate}

\subsection{The Lévy measure matrix of the bonding MAP}

Let us now investigate the relationship between the Lévy measure matrices of friends and their bonding MAP,
complementing the \'equations amicales invers\'es for MAPs derived in
\cite{doering21}
and extending the équations amicales for L\'evy processes in \cite{vigon2002}.

The following results show that the Lévy characteristics and transitional jumps of friends must satisfy certain compatibility conditions. The proofs of the propositions will be given along the proof of Theorem \ref{theo: eq amicales}. Proposition \ref{prop: dens}, describing the existence of densities of the jump measures associated to $\bm{\pi}$-friends subject to existence of Lévy drifts of the respective friend is the natural generalization of Proposition 3.1 in \cite{vigon2002}. Proposition \ref{prop: comp cond} and the following corollary demonstrate that transitional jumps and Lévy jumps in a friendship must be finely synchronized. This aspect has no counterpart for friendships of Lévy processes and is among the reasons why constructing explicit examples of MAP friendships is far from being a trivial extension of the strategy known for Lévy friendships.

The measures
\[ 
  \chi^+_i(\diff{x}) 
  = d^+_i \delta_0(\diff{x}) + \one_{(0,\infty)}(x)\overbar{\Pi}{}^+_i(x)\diff{x},
\]
and
\[
  \tilde{\chi}^-_i(\diff{x}) 
  = d^-_i \delta_0(\diff{x}) + \one_{(-\infty,0)}(x)\overbar{\tilde{\Pi}}{}^-_i(x)\diff{x},
\]
defined for $x \in \R$ and $i \in [n]$, will play an important role throughout
the section.
Probabilistically,
$\chi^+_i$ can be interpreted as
the invariant measure of the overshoot process associated to (an unkilled version of) $H^{+,(i)}$,
and $\tilde{\chi}{}^-_i(-\diff{x})$ as the same quantity for $H^{-,(i)}$;
see \cite[Theorem 3.10]{doering21}.

\begin{proposition}\label{prop: dens}
  Suppose that $(H^+,J^+)$ is a $\bm{\pi}$-friend of $(H^-,J^-)$. Then, for any $i \in [n]$, $\Pi^\pm_i$  has a density $\uppartial \Pi^\pm_i$  on $(0,\infty)$ if $d^\mp_i > 0$, which has a càdlàg  version. Moreover, for any $i,j \in [n]$ with $i \neq j$, $F^{\pm}_{i,j}$  restricted to $(0,\infty)$ has a density $f^\pm_{i,j}$  if $d^\mp_i > 0$, which can also be chosen càdlàg. 
\end{proposition}

\begin{proposition}\label{prop: comp cond}
  Suppose that $(H^+,J^+)$ is a $\bm{\pi}$-friend of $(H^-, J^-)$. 
  Then, for any $i,j \in [n]$ with $i \neq j$, it holds that
  \[q^+_{i,j} F^+_{i,j} \ast \tilde{\chi}{}^-_i(\diff{x}) - \frac{\pi(j)}{\pi(i)} q^-_{j,i} \tilde{F}{}^-_{j,i} \ast \chi^+_j(\diff{x}) = \underbars{\nu}{}_{i,j}(x)\diff{x}, \quad x \in \R,\]
  for some finite signed measure $\nu_{i,j}$ with $\nu_{i,j}(\R) = 0$.
  Moreover,
  \begin{align*}
    q_{i,j} F_{i,j}(\{0\}) &= (\dagger^-_i- q^-_{i,i})q^+_{i,j} F^+_{i,j}(\{0\}) + (\dagger^+_j-q^+_{j,j}) \frac{\pi(j)}{\pi(i)}q^-_{j,i} F^-_{j,i}(\{0\}) - \nu_{i,j}(\{0\})\\
    &\quad- \sum_{k \neq i,j} \frac{\pi(k)}{\pi(i)} q^-_{k,i} q^+_{k,j} \tilde{F}{}^-_{k,i} \ast F^+_{k,j}(\{0\}).
  \end{align*}
\end{proposition}

\begin{corollary}\label{coro: balance atoms}
  Suppose that $(H^+,J^+)$ is a $\bm{\pi}$-friend of $(H^-, J^-)$ and that the measures $F^\pm_{i,j}$ are continuous on $(0,\infty)$ for distinct $i,j \in [n]$. Then, 
  \[\bm{\Delta}^-_{\bm{d}}\bm{\Pi}^+(\{0\}) = \bm{\Delta}_{\bm{\pi}}^{-1} \big(\bm{\Delta}_{\bm{d}}^+\bm{\Pi}^-(\{0\})\big)^\top \bm{\Delta}_{\bm{\pi}}.\]
\end{corollary}
\begin{proof}
  By Proposition \ref{prop: comp cond}, the measure $q^+_{i,j} F^+_{i,j} \ast \tilde{\chi}{}^-_i - \frac{\pi(j)}{\pi(i)} q^-_{j,i} \tilde{F}{}^-_{j,i} \ast \chi^+_j$ is absolutely continuous. Using that the measures $F^\pm_{i,j}$ are continuous on $(0,\infty)$, we obtain for  $i \neq j$ that
  \[0 = q^+_{i,j} F^+_{i,j} \ast \tilde{\chi}{}^-_i(\{0\}) - \frac{\pi(j)}{\pi(i)} q^-_{j,i} \tilde{F}{}^-_{j,i} \ast \chi^+_j(\{0\}) = d^-_i q^+_{i,j}F^+_{i,j}(\{0\}) - d^+_j\frac{\pi(j)}{\pi(i)}q^-_{j,i}F^-_{j,i}(\{0\}),\]
  which implies 
  \[d^-_i q^+_{i,j}F^+_{i,j}(\{0\}) = d^+_j\frac{\pi(j)}{\pi(i)}q^-_{j,i}F^-_{j,i}(\{0\}),\]
  proving the claim.
\end{proof}

The following statement refines the characterisation of the Lévy measures $\Pi^+_i$ and $\Pi^-_i$ of $\bm{\pi}$-friends in terms of càdlàg densities whenever $d^-_i > 0$ and $d^+_i > 0$, resp., as stated above. The proof is almost identical to that of Théorème 6.4.1 in \cite{vigondiss} if we take into account the équations amicales of MAPs given in Theorem \ref{theo: eq amicales} below and is therefore omitted.
\begin{proposition}\label{prop: levy dens}
  Suppose that $(H^+,J^+)$ is a $\bm{\pi}$-friend of $(H^-,J^-)$. Then, for any $i \in [n]$, if $d^\mp_i > 0$ the measure $\Pi^\pm_i$  restricted to $(0,\infty)$ has a càdlàg  density $\uppartial \Pi^\pm_i = \overbar{\eta}{}^\pm_i$, where $\eta^\pm_i$  is a signed measure on $(0,\infty)$.
\end{proposition}

The necessary conditions for friendship given in Proposition \ref{prop: dens}, Proposition \ref{prop: comp cond} and Proposition \ref{prop: levy dens} together with the characterisation of MAP exponents in Lemma \ref{lemma: map exp} motivate the following definition. 

\begin{definition} \label{def: comp}
  A MAP subordinator $(H^+,J^+)$ is called $\bm{\pi}$-compatible with a MAP subordinator $(H^-,J^-)$ if the following conditions are satisfied:
  \begin{enumerate}[label = (\roman*), ref= (\roman*)]
    \item for any distinct $i, j \in [n]$, if $d^\mp_i > 0$, the measures $\Pi^\pm_i$ and $F^\pm_{i,j}$  restricted to $(0,\infty)$ have càdlàg  densities $\uppartial \Pi^\pm_i$ and $f^\pm_{i,j}$, respectively, where $\uppartial \Pi^\pm_i$ is given as the right tail of signed measures $\eta^\pm_i$ on $(0,\infty)$;\label{comp cond1}
    \item for any distinct $i,j \in [n]$, there exists a finite signed measure $\nu_{i,j}$ such that $\nu_{i,j}(\R) = 0$ and
      \begin{equation} \label{eq: friend dens}
        q^+_{i,j} F^+_{i,j} \ast \tilde{\chi}{}^-_i(\diff{x}) - \frac{\pi(j)}{\pi(i)} q^-_{j,i} \tilde{F}{}^-_{j,i} \ast \chi^+_j(\diff{x}) = \underbars{\nu}{}_{i,j}(x)\diff{x}, \quad x \in \R,
      \end{equation}
      and moreover
      \begin{equation}
        \begin{split} \label{eq: friends atom0}
          0 &\leq (\dagger^-_i- q^-_{i,i})q^+_{i,j} F^+_{i,j}(\{0\}) + (\dagger^+_j-q^+_{j,j}) \frac{\pi(j)}{\pi(i)}q^-_{j,i} F^-_{j,i}(\{0\}) - \nu_{i,j}(\{0\})\\
          &\quad- \sum_{k \neq i,j} \frac{\pi(k)}{\pi(i)} q^-_{k,i} q^+_{k,j} \tilde{F}{}^-_{k,i} \ast F^+_{k,j}(\{0\});
        \end{split}
      \end{equation} \label{friends cond4}
    \item the vector $\bm{\Delta}_{\bm{\pi}}^{-1}\bm{\Psi}^-(0)^\top\bm{\Delta}_{\bm{\pi}} \bm{\Psi}^+(0)\one$ is nonnegative.
      \label{friends cond3}
    \item
      the vector $\bm{\pi}^\top\bm{\Delta}_{\bm{\pi}}^{-1}\bm{\Psi}^-(0)^\top\bm{\Delta}_{\bm{\pi}} \bm{\Psi}^+(0)$ is nonnegative.
      \label{friends cond-pi}
  \end{enumerate}
\end{definition}
\begin{remark}
  \begin{enumerate}[label = (\roman*), ref= (\roman*)]
    \item
      When $(H^+,J^+)$ is unkilled, \ref{friends cond3} is satisfied and the vector
      in question is zero; and 
      when $(H^-,J^-)$ is unkilled, the same applies to \ref{friends cond-pi}.
    \item
      If condition \ref{comp cond1} of $\bm{\pi}$-compatibility holds,
      then condition \ref{friends cond4} is equivalent to:
      \begin{enumerate}[label = (\alph*), ref= (\alph*)]
        \item 
          $
          q_{i,j}^+ d_i^- F_{i,j}^+(\{0\}) 
          - \frac{\pi(j)}{\pi(i)} q_{j,i}^- d_i^+ F_{j,i}^-(\{0\}) = 0
          $,
        \item
          the expression
          \begin{multline*}
            q_{i,j}^+ \biggl(
              \int_{\R} \mathbf{1}_{\{y\ge 0, y>x\}} \overbar{\Pi}^-_i(y-x) F_{i,j}^+(\diff{y})
              + d_i^- f_{i,j}^+(x)
            \biggr)
            \\
            {} - \frac{\pi(j)}{\pi(i)}
            q_{j,i}^- \biggl(
              \int_{\R} \mathbf{1}_{\{y \geq 0, -x<y\}} \overbar{\Pi}^+_j(x+y) \, F_{j,i}^-(\diff{y})
              + d_j^+ f_{j,i}^-(-x)
            \biggr)
          \end{multline*}
          is a.e.\ equal to a right-continuous, bounded variation function of $x$,
          denoted $n_{i,j}(x)$, whose limit
          as $x\to\pm\infty$ is $0$, and
        \item
          the inequality
          \begin{multline*}
            (\dagger^-_i- q^-_{i,i})q^+_{i,j} F^+_{i,j}(\{0\}) + (\dagger^+_j-q^+_{j,j}) \frac{\pi(j)}{\pi(i)}q^-_{j,i} F^-_{j,i}(\{0\}) \\
            {} - \sum_{k \neq i,j} \frac{\pi(k)}{\pi(i)} q^-_{k,i} q^+_{k,j} \tilde{F}{}^-_{k,i} \ast F^+_{k,j}(\{0\})
            - \lim_{x \downarrow 0} n_{i,j}(x) + \lim_{x \uparrow 0} n_{i,j}(x)
            \ge 0
          \end{multline*}
          holds.
      \end{enumerate}
  \end{enumerate}
\end{remark}

\begin{lemma} 
  \label{lem:compat-symm}
  $(H^+,J^+)$ is $\bm{\pi}$-compatible with $(H^-,J^-)$ if, and only if,
  $(H^-,J^-)$ is $\bm{\pi}$-compatible with $(H^+,J^+)$.
\end{lemma} 
\begin{proof} 
  We prove one direction of the equivalence, the other following immediately
  by swapping $(H^+,J^+)$ and $(H^-,J^-)$.
  Clearly, condition \ref{comp cond1} for $\bm{\pi}$-compatibility between $(H^-,J^-)$ and $(H^+,J^+)$ is satisfied since $(H^+,J^+)$ is assumed to be $\bm{\pi}$-compatible with $(H^-,J^-)$. Next, we have 
  \begin{align*} 
    -\frac{\pi(i)}{\pi(j)} \Big( q^-_{i,j} F^-_{i,j} \ast \tilde{\chi}{}^+_i(\diff{x}) - \frac{\pi(j)}{\pi(i)} q^+_{j,i} \tilde{F}{}^+_{j,i} \ast \chi^-_j(\diff{x})\Big) &= q^+_{j,i} \tilde{F}{}^+_{j,i} \ast \chi^-_j(\diff{x}) -\frac{\pi(i)}{\pi(j)} q^-_{i,j} F^-_{i,j} \ast \tilde{\chi}{}^+_i(\diff{x}) \\
    &= q^+_{j,i} F^+_{j,i} \ast \tilde{\chi}{}^-_j(-\diff{x}) -\frac{\pi(i)}{\pi(j)} q^-_{i,j} \tilde{F}{}^-_{i,j} \ast \chi^+_i(-\diff{x})\\
    &= \underbars{\nu}_{j,i}(-x) \diff{x}\\
    &= -\overbar{\nu}_{j,i}(-x) \diff{x}\\ 
    &= -\tilde{\nu}_{j,i}((-\infty,x]) \diff{x},
  \end{align*}
  where for the penultimate line we used $\nu_{j,i}(\R) = 0$.
  Thus, \eqref{eq: friend dens} holds with with the role of $(H^\pm,J^\pm)$ reversed and $\nu_{i,j}$ replaced by the finite signed measure $\rho_{i,j}(\diff{x}) = \frac{\pi(j)}{\pi(i)}\nu_{j,i}(-\diff{x})$, for which $\rho_{i,j}(\R) = 0$.  Moreover, 
  \begin{align*} 
    &\frac{\pi(i)}{\pi(j)} \Big((\dagger^+_i- q^+_{i,i})q^-_{i,j} F^-_{i,j}(\{0\}) + (\dagger^-_j-q^-_{j,j}) \frac{\pi(j)}{\pi(i)}q^+_{j,i} F^+_{j,i}(\{0\}) - \rho_{i,j}(\{0\})\\
    &\qquad  - \sum_{k \neq i,j} \frac{\pi(k)}{\pi(i)} q^+_{k,i} q^-_{k,j} \tilde{F}{}^+_{k,i} \ast F^-_{k,j}(\{0\})\Big)\\ 
    &\, = (\dagger^-_j-q^-_{j,j}) q^+_{j,i} F^+_{j,i}(\{0\}) + \frac{\pi(i)}{\pi(j)} (\dagger^+_i- q^+_{i,i})q^-_{i,j} F^-_{i,j}(\{0\}) - \nu_{j,i}(\{0\}) - \sum_{k \neq i,j} \frac{\pi(k)}{\pi(j)} q^+_{k,i} q^-_{k,j} \tilde{F}{}^+_{k,i} \ast F^-_{k,j}(\{0\})\\
    &\,\geq 0,
  \end{align*}
  by the assumption that $(H^+,J^+)$ is $\bm{\pi}$-compatible with $(H^-,J^-)$. Consequently, \eqref{eq: friends atom0} is satisfied as well with the role of $(H^\pm,J^\pm)$ reversed and $\rho_{i,j}$ in place of $\nu_{i,j}$. 
  We next need to verify that (a) $\bm{\Delta}_{\bm\pi}^{-1} \bm{\Psi}^+(0)^\top \bm{\Delta}_{\bm{\pi}} \bm{\Psi}^-(0)\one$ using hypothesis (b) $\bm{\pi}^\top \bm{\Delta}_{\bm{\pi}}^{-1} \bm{\Psi}^-(0)^\top \bm{\Delta}_{\bm{\pi}} \bm{\Psi}^+(0) \geq \bm{0}^\top$. Statement (b) is equivalent to $\bm{0} \leq \one^\top \bm{\Psi}^-(0)^\top \bm{\Delta}_{\bm{\pi}} \bm{\Psi}^+(0)$ and statement (a) is equivalent to 
  \[\bm{0}^\top \leq \bm{\Delta}_{\bm{\pi}} \big(\bm{\Delta}_{\bm\pi}^{-1} \bm{\Psi}^+(0)^\top \bm{\Delta}_{\bm{\pi}} \bm{\Psi}^-(0)\one\big) = \big(\one^\top \bm{\Psi}^-(0)^\top \bm{\Delta}_{\bm{\pi}} \bm{\Psi}^+(0) \big)^\top.
  \]
  It follows that (b) implies (a).
  The proof that $\bm{\pi}^\top\bm{\Delta}_{\bm{\pi}}^{-1} \Psi^+(0)^\top \bm{\Delta}_{\bm{\pi}} \bm{\Psi}^-(0)$
  is nonnegative is analogous.
\end{proof}

Given two $\bm{\pi}$-compatible MAP subordinators $(H^+,J^+)$ and $(H^-,J^-)$, let 
\[\uppartial \bm{\Pi}^+(x) \coloneqq (\uppartial \Pi^+_i(x)\one_{\{i=j\}} + q^+_{i,j} f^+_{i,j}(x)\one_{\{i \neq j\}})_{i,j=1,\ldots,n}, \quad x > 0,\]
where $\uppartial \Pi^+_i$ is the absolutely continuous part of the measure
$\Pi^+_i$ and $f^+_{i,j}$ the absolutely continuous part of $F^+_{i,j}$, and in
case $d^-_i > 0$ we use the càdlàg versions guaranteed by the definition
of $\bm{\pi}$-compatibility.
We call $\uppartial \bm{\Pi}^+$ the \textit{Lévy density matrix}
of $(H^+,J^+)$, and define $\uppartial \bm{\Pi}^-$ for $(H^-,J^-)$
analogously.

\begin{theorem}[Équations amicales for MAPs]\label{theo: eq amicales}
  Suppose that $(H^+,J^+)$ is a $\bm{\pi}$-friend of $(H^-,J^-)$. Then, for a.e.\ $x > 0$
  \begin{equation} \label{eq: vigon1}
    \bm{\Pi}(x,\infty) = \int_{x+}^\infty \bm{\Delta}_{\bm{\pi}}^{-1}\Big(\overbar{\bm{\Pi}}^-(y-x) - \bm{\Psi}^-(0)\Big)^\top \bm{\Delta}_{\bm{\pi}} \, \bm{\Pi}^+(\diff{y})  + \bm{\Delta}^-_{\bm{d}} \uppartial \bm{\Pi}^+(x),
  \end{equation}
  and for a.e.\ $x < 0$
  \begin{equation} \label{eq: vigon2}
    \bm{\Pi}(-\infty,x) = \int_{(-x)+}^\infty \bm{\Delta}_{\bm{\pi}}^{-1} \big(\bm{\Pi}^-(\diff{y}) \big)^\top \bm{\Delta}_{\bm{\pi}} \, \big(\overbar{\bm{\Pi}}{}^+(y+x) - \bm{\Psi}^+(0)\big)  +  \bm{\Delta}_{\bm{\pi}}^{-1} \big(\bm{\Delta}_{\bm{d}}^+\uppartial\bm{\Pi}^-(-x)\big)^\top \bm{\Delta}_{\bm{\pi}}.
  \end{equation}
\end{theorem}
\begin{proof}[Proof of Propositions \ref{prop: dens} and \ref{prop: comp cond} and Theorem \ref{theo: eq amicales}]
  Define the measures $\tilde{\Pi}{}^-_i(\diff{x}) = \Pi^-_i(-\diff{x})$ and $\tilde{F}{}^-_{i,j}(\diff{x}) = F^-_{i,j}(-\diff{x})$ and recall \eqref{eq:levy_dist}. Taking inverse Fourier transforms on both sides of \eqref{eq: wienerhopf} (intepreted in the sense of distributions) we obtain for $i = j$
  \begin{equation}\label{eq: vigon diag}
    \begin{split}
      (q_{i,i}-\dagger_i)\delta - a_i \delta^\prime + \frac{1}{2}\sigma_i^2 \delta^{\prime \prime} + \bbGamma^2 \Pi_i  &= - \Big(\big((q^-_{i,i} - \dagger^-_i)\delta + d^-_i\delta^\prime + \bbGamma \tilde{\Pi}{}^-_i\big) \ast \big((q^+_{i,i}- \dagger^+_i)\delta - d^+_i \delta^\prime + \bbGamma \Pi^+_i \big) \\
      &\qquad + \sum_{k \neq i} \frac{\pi(k)}{\pi(i)} q^-_{k,i}q^+_{k,i} \tilde{F}{}^-_{k,i} \ast F^+_{k,i}\Big).
    \end{split}
  \end{equation}
  Above, all convolutions are well defined since $\bbGamma \tilde{\Pi}{}^-_i$ and $\bbGamma \Pi^+_i$ can both be decomposed into the sum of a distribution with compact support and a distribution induced by a finite measure, respectively (see also Proprieté 3.9 in \cite{vigon2002}). Let $\varrho(x) = \one_{(0,\infty)}(x) - \one_{(-\infty,0)}(x)$, $\overbar{\Pi}_i(x) = \one_{(0,\infty)}(x)\Pi_i(x,\infty) + \one_{(-\infty,0)}(x)\Pi_i(-\infty,x)$ and $\overbar{\tilde{\Pi}}{}^-_i(x) = \one_{(-\infty,0)}(x) \tilde{\Pi}{}^-_i(-\infty,x)$. Then, $(\bbGamma \varrho \overbar{\Pi}_i)^\prime = - \bbGamma^2 \Pi_i + c_i \delta^\prime$ for some constant $c_i \in \R$ and $(\overbar{\tilde{\Pi}}{}^-_i)^\prime = \bbGamma \tilde{\Pi}{}^-_i$ by Lemma 3.12 in \cite{vigon2002}. Moreover, it is easily shown that for $\underbars{\tilde{F}}{}^-_{i,j}(x) = \one_{\R_-}(x)\tilde{F}{}^-_{i,j}([x,0])$ we have $(\underbars{\tilde{F}}{}^-_{i,j})^\prime = -\tilde{F}{}^-_{i,j}$. Thus, taking primitives on both sides of \eqref{eq: vigon diag} we obtain 
  \begin{align*} 
    &(q_{i,i}-\dagger_i) \one_{\R_{-}} + a_i \delta - \tfrac{1}{2}\sigma^2_i\delta^\prime + \bbGamma \varrho \overbar{\Pi}_i \\
    &\quad =  \big((\dagger^-_i - q^-_{i,i})\one_{\R_{-}} + d^-_i\delta + \overbar{\tilde{\Pi}}{}^-_i\big) \ast \big((q^+_{i,i}- \dagger^+_i)\delta - d^+_i \delta^\prime + \bbGamma \Pi^+_i  \big)\\
    &\qquad - \sum_{k \neq i} \frac{\pi(k)}{\pi(i)} q^-_{k,i}q^+_{k,i} \underbars{\tilde{F}}{}^-_{k,i} \ast F^+_{k,i} + c_i \delta + c_{\text{int}},
  \end{align*}
  for some integration constant $c_{\text{int}} \in \R$. 
  By resticting to $(0,\infty)$ this implies that we have the following equality of distributions in $\mathcal{D}^\prime_{(0,\infty)}$:
  \begin{equation} \label{eq: vigon diag2}
    \begin{split} 
      &\overbar{\Pi}_i\vert_{(0,\infty)} \\
      &\quad= \big((\dagger^-_i - q^-_{i,i})\one_{\R_{-}} + d^-_i\delta + \overbar{\tilde{\Pi}}{}^-_i\big) \ast \Pi^+_i\vert_{(0,\infty)} - \sum_{k \neq i} \frac{\pi(k)}{\pi(i)} q^-_{k,i}q^+_{k,i} \underbars{\tilde{F}}{}^-_{k,i} \ast F^+_{k,i}\vert_{(0,\infty)} + c_{\text{int}}\\ 
      &\quad= (\dagger^-_i - q^-_{i,i})\overbar{\Pi}{}^+_i\vert_{(0,\infty)} + d^-_i \Pi^+_i \vert_{(0,\infty)} + \overbar{\tilde{\Pi}}{}^-_i \ast \Pi^+_i\vert_{(0,\infty)} - \sum_{k \neq i} \frac{\pi(k)}{\pi(i)} q^-_{k,i}q^+_{k,i} \underbars{\tilde{F}}{}^-_{k,i} \ast F^+_{k,i}\vert_{(0,\infty)} + c_{\text{int}}.
    \end{split}
  \end{equation}
  Here we used Proprieté 3.8 in \cite{vigon2002}, telling us that for tempered distributions $S,T$, where $T$ is supported on $\R_-$ and the convolution $T \ast S$ is well defined as a tempered distribution, it holds $(T \ast S)\vert_{(0,\infty)} = (T \ast S\vert_{(0,\infty)})\vert_{(0,\infty)}$. 
  Since all other terms are distributions induced by some function on $(0,\infty)$, it follows that if $d^-_i > 0$, $\Pi^+_i \vert_{(0,\infty)}$ is also induced by a function, i.e.\ $\Pi^+_i$ possesses a Lebesgue density $\uppartial \Pi^+_i$ on $(0,\infty)$. 
  Let us first show that $c_{\mathrm{int}} = 0$. Let $\varphi \in \mathcal{D}_{(0,\infty)}$ be non-negative with $\mathrm{supp}(\varphi) \subset (0,1)$, $\lVert \varphi \rVert_\infty \leq 1$, $\int \varphi = 1/2$ and $\varphi_z = \varphi(\cdot - z)$ for $z  > 0$. Then, with multiple uses of Fubini,
    \begin{equation}\label{eq:calc_const}
      \begin{split} 
        \int_{\R} \varphi_z(x) \overbar{\tilde{\Pi}}{}^-_i \ast \Pi^+_i(x) \diff{x} & \leq  \int_z^{z+1} \overbar{\tilde{\Pi}}{}^-_i \ast \Pi^+_i(x) \diff{x}\\
        &= \int^0_{-\infty} \int_0^\infty \one_{(z,z+1)}(x+y) \, \Pi^+_i(\diff{x}) \overbar{\tilde{\Pi}}{}^-_i(y) \diff{y} \\
        &= \int_0^\infty \Pi^+_i((z+y,z+y+1)) \overbar{\Pi}{}^-_i(y) \diff{y}\\
        &\leq \Pi^+_i((z,z+2)) \int_0^1 \overbar{\Pi}{}^-_i(y) \diff{y} + \int_1^\infty \Pi^+_i((z+y,z+y+1)) \overbar{\Pi}{}^-_i(y) \diff{y}\\
        &\leq \Pi^+_i((z,z+2)) \int_0^1 \overbar{\Pi}{}^-_i(y) \diff{y} + \overbar{\Pi}{}^-_i(1) \int_{1+z}^\infty \int_{(u-(z+2)) \vee 1}^{u-(z+1)} \diff{y} \, \Pi^+_i(\diff{u})\\
        &\leq \Pi^+_i((z,z+2)) \int_0^1 \overbar{\Pi}{}^-_i(y) \diff{y}  \\
        &\qquad + \overbar{\Pi}{}^-_i(1)\big(\Pi^+_{i}(z+2,\infty) + \int_{z+1}^{z+2} (u-z-1) \Pi^+_i(\diff{u}) \big)\\
        &\leq \Pi^+_i((z,z+2)) \int_0^1 \overbar{\Pi}{}^-_i(y) \diff{y} + \overbar{\Pi}{}^-_i(1) \Pi^+_i(z+1,\infty) \underset{z \to \infty}{\longrightarrow} 0.
      \end{split}
    \end{equation}
    Writing 
    \[\mu = \overbar{\Pi}_i\vert_{(0,\infty)} - (\dagger^-_i - q^-_{i,i})\overbar{\Pi}{}^+_i\vert_{(0,\infty)} - d^-_i \Pi^+_i \vert_{(0,\infty)} - \overbar{\tilde{\Pi}}{}^-_i \ast \Pi^+_i\vert_{(0,\infty)} + \sum_{k \neq i} \frac{\pi(k)}{\pi(i)} q^-_{k,i}q^+_{k,i} \underbars{\tilde{F}}{}^-_{k,i} \ast F^+_{k,i}\vert_{(0,\infty)},\]
    it therefore follows from \eqref{eq: vigon diag2} that 
    \[\frac{c_{\mathrm{int}}}{2} = c_{\mathrm{int}} \int \varphi_z(x) \diff{x} = \langle \mu, \varphi_z \rangle \underset{z \to \infty}{\longrightarrow} 0,\]
    hence $c_{\mathrm{int}} = 0.$ Next, let us show that if $d^-_i> 0$, $\uppartial \Pi^+_i$ has a càdlàg version. Reordering \eqref{eq: vigon diag2} using $c_{\mathrm{int}} = 0$ we obtain 
    \begin{equation}\label{eq: vigon diag4}
      \begin{split}
        \overbar{\Pi}_i\vert_{(0,\infty)} &- (\dagger^-_i - q^-_{i,i}) \overbar{\Pi}{}^+_i\vert_{(0,\infty)}+ \sum_{k \neq i} \frac{\pi(k)}{\pi(i)} q^-_{k,i}q^+_{k,i} \underbars{\tilde{F}}{}^-_{k,i} \ast F^+_{k,i}\vert_{(0,\infty)}\\
        &= d^-_i \Pi^+_i\vert_{(0,\infty)}  + \overbar{\tilde{\Pi}}{}^-_i \ast \Pi^+_i\vert_{(0,\infty)}.
      \end{split}
    \end{equation}
    Since the left hand side is a distribution induced by a function that is bounded away from $0$ it follows that the Lebesgue density of $\Pi^+_i$ has a version $g^+_i$ that is bounded away from zero as well, i.e., for any $x>0$ it holds that 
    \begin{equation}\label{eq:dens bounded}
      \sup_{z \geq x} g^+_i(z) < \infty.
    \end{equation}
    Integration by parts then shows
    \[\overbar{\tilde{\Pi}}{}^-_i \ast \Pi^+_i(x) = \int_{(0,\infty)} \int_{x}^{x+y} g^+_i(z) \diff{z} \, \Pi^-_i(\diff{y}),\]
    such that dominated convergence in conjunction with \eqref{eq:dens bounded} and the integrability properties of the Lévy measure $\Pi^-_i$ readily imply that $x \mapsto \overbar{\tilde{\Pi}}{}^-_i \ast \Pi^+_i(x)$ is continuous on $(0,\infty)$. Hence, the function 
    \begin{align*}
      \uppartial \Pi^+_i(x) &\coloneqq \frac{1}{d^-_i}\Big(\overbar{\Pi}_i\vert_{(0,\infty)}(x)- (\dagger^-_i - q^-_{i,i}) \overbar{\Pi}{}^+_i\vert_{(0,\infty)}(x)\\
      &\qquad + \sum_{k \neq i} \frac{\pi(k)}{\pi(i)} q^-_{k,i}q^+_{k,i} \int_{(x,\infty)} F^-_{k,i}([0,y-x)) \, F^+_{k,i}(\diff{y}) - \overbar{\tilde{\Pi}}{}^-_i \ast g^+_i(x) \Big), \quad x > 0,
    \end{align*}
  is càdlàg and by \eqref{eq: vigon diag4} is the desired càdlàg version of the Lebesgue density of $\Pi^+_i$ on $(0,\infty)$.

  It now follows from above that the equality \eqref{eq: vigon diag2} of distributions in $\mathcal{D}^\prime_{(0,\infty)}$ translates to the equality of functions 
  \begin{equation} \label{eq: vigon diag3}
    \begin{split}
      \Pi_i(x,\infty) &= (\dagger^-_i - q^-_{i,i})\overbar{\Pi}{}^+_{i}(x) + d^-_i \uppartial\Pi^+_i(x) + \int_{x+}^{\infty} \overbar{\Pi}{}^-_i(y-x)\, \Pi_i^+(\diff{y})\\
      &\quad -\sum_{k\neq i} \frac{\pi(k)}{\pi(i)}q^-_{k,i}q^+_{k,i}\int_{x+}^\infty F^-_{k,i}([0,y-x]) \, F^+_{k,i}(\diff{y}),
    \end{split}
  \end{equation}
  which holds for a.e.\ $x > 0$.

  Next, for $i \neq j$, it follows again by taking inverse Fourier transforms on the $(i,j)$th entry of \eqref{eq: wienerhopf} that 
  \begin{equation} \label{eq: vigon off1}
    \begin{split}
      q_{i,j} F_{i,j} &= - \Big\{q^+_{i,j} \big((q^-_{i,i} - \dagger^-_i)\delta + d^-_i\delta^\prime + \bbGamma \tilde{\Pi}{}^-_i\big) \ast F^+_{i,j} + \frac{\pi(j)}{\pi(i)} q^-_{j,i}  \tilde{F}{}^-_{j,i} \ast \big((q^+_{j,j}- \dagger^+_j)\delta - d^+_j \delta^\prime + \bbGamma \Pi^+_j \big) \\
      &\qquad + \sum_{k \neq i,j} \frac{\pi(k)}{\pi(i)} q^-_{k,i}q^+_{k,j} \tilde{F}{}^-_{k,i} \ast F^+_{k,j}\Big\}.
    \end{split}
  \end{equation}
  From this it follows that 
  \begin{align*}
    q^+_{i,j} F^+_{i,j} \ast \big(d^-_i \delta^\prime + \bbGamma \tilde{\Pi}{}^-_i\big) + \frac{\pi(j)}{\pi(i)} q^-_{j,i} \tilde{F}{}^-_{j,i} \ast \big(-d^+_j\delta^\prime + \bbGamma  \Pi^+_j\big) &= q^+_{i,j} F^+_{i,j} \ast \big(\tilde{\chi}{}^-_i\big)^\prime - \frac{\pi(j)}{\pi(i)} q^-_{j,i} \tilde{F}{}^-_{j,i} \ast \big(\chi^+_j\big)^\prime\\
    &= \Big(q^+_{i,j} F^+_{i,j} \ast \tilde{\chi}{}^-_i- \frac{\pi(j)}{\pi(i)} q^-_{j,i} \tilde{F}{}^-_{j,i} \ast \chi^+_j\Big)^\prime,
  \end{align*}
  must be induced by a finite signed measure. By Lemma \ref{lem: dist der}, this implies the existence of a finite signed measure $\nu_{i,j}$ such that 
  \begin{equation}\label{eq: comp friends1}
    q^+_{i,j} F^+_{i,j} \ast \tilde{\chi}{}^-_i - \frac{\pi(j)}{\pi(i)} q^-_{j,i} \tilde{F}{}^-_{j,i} \ast \chi^+_j = \underbars{\nu}{}_{i,j} +c
  \end{equation}
  for some $c \in \R$. Letting $\varphi_z$ as before and arguing as in \eqref{eq:calc_const}, we obtain 
  \begin{equation} \label{eq:calc_const2}
    \lim_{z \to \infty} \int \varphi_z(x) \overbar{\tilde{\Pi}}{}^-_i \ast F^+_{i,j}(x) \diff{x} = 0, \quad \lim_{z \to -\infty} \int \varphi_{z}(-x) q^-_{j,i}\tilde{F}{}^-_{j,i} \ast \overbar{\Pi}{}^+_j(x) \diff{x} = 0.
  \end{equation}
  Since for $x > z > 0$ we have  $\tilde{F}{}^-_{j,i} \ast \overbar{\Pi}^+_j(x) \leq \overbar{\Pi}^+_j(z)$ and for $x < z < 0$, $\overbar{\tilde{\Pi}}{}^-_i \ast F^+_{i,j}(x) \leq \overbar{\Pi}^-_i(-z)$, we  obtain 
  \begin{equation} \label{eq:calc_const3}
    \lim_{z \to -\infty} \int \varphi_z(-x) \overbar{\tilde{\Pi}}{}^-_i \ast F^+_{i,j}(x) \diff{x} = 0, \quad \lim_{z \to \infty} \int \varphi_{z}(x) q^-_{j,i}\tilde{F}{}^-_{j,i} \ast \overbar{\Pi}{}^+_j(x) \diff{x} = 0.
  \end{equation}
  Consequently, using also $\lim_{x \to -\infty} \underbars{\nu}_{i,j}(x) = 0$, it follows
  \[\frac{c}{2} = \int \varphi_z(-x) \diff{x} = \big\langle q^+_{i,j} F^+_{i,j} \ast \tilde{\chi}{}^-_i - \frac{\pi(j)}{\pi(i)} q^-_{j,i} \tilde{F}{}^-_{j,i} \ast \chi^+_j - \underbars{\nu}{}_{i,j}, \varphi_z(-\cdot) \big\rangle \underset{z \to -\infty}{\longrightarrow} 0,\] 
  whence, $c = 0$. Since $\nu_{i,j}$ is finite we may now write 
  \[q^+_{i,j} F^+_{i,j} \ast \tilde{\chi}{}^-_i - \frac{\pi(j)}{\pi(i)} q^-_{j,i} \tilde{F}{}^-_{j,i} \ast \chi^+_j = \nu_{i,j}(\R) -\overbar{\nu}{}_{i,j}. \]
  Hence, from $\overbar{\nu}_{i,j}(x) \underset{x \to \infty}{\longrightarrow} 0$ and \eqref{eq:calc_const2}, \eqref{eq:calc_const3} it follows 
  \[\frac{\nu_{i,j}(\R)}{2} = \nu_{i,j}(\R) \int \varphi_z(x) \diff{x} = \big\langle q^+_{i,j} F^+_{i,j} \ast \tilde{\chi}{}^-_i - \frac{\pi(j)}{\pi(i)} q^-_{j,i} \tilde{F}{}^-_{j,i} \ast \chi^+_j + \overbar{\nu}{}_{i,j}, \varphi_z \big\rangle \underset{z \to \infty}{\longrightarrow} 0,\]
  whence, $\nu_{i,j}(\R) = 0$.
  We now also obtain from \eqref{eq: vigon off1},
  \begin{align*}
    q_{i,j} F_{i,j}(\{0\}) &= (\dagger^-_i- q^-_{i,i})q^+_{i,j} F^+_{i,j}(\{0\}) + (\dagger^+_j-q^+_{j,j}) \frac{\pi(j)}{\pi(i)}q^-_{j,i} F^-_{j,i}(\{0\}) - \nu_{i,j}(\{0\})\\
    &\quad- \sum_{k \neq i,j} \frac{\pi(k)}{\pi(i)} q^-_{k,i} q^+_{k,j} \tilde{F}{}^-_{k,i} \ast F^+_{k,j}(\{0\}).
  \end{align*}
  Taking everything together establishes Proposition \ref{prop: comp cond}.
  Let now $\overbar{F}_{i,j}(x) = F_{i,j}((x,\infty)) \one_{(0,\infty)}(x) + F_{i,j}((-\infty,x)) \one_{(-\infty,0)}(x).$ 
  Then, $(\varrho \overbar{F}_{i,j})' = -F_{i,j} + \delta$,
  and hence, together with the considerations above we obtain by taking primitives on \eqref{eq: vigon off1}
  \begin{align*}
    \varrho q_{i,j}\overline{F}_{i,j}  
    &= q^+_{i,j} \big((\dagger^-_i- q^-_{i,i})\one_{\R_{-}} 
    + d^-_i\delta + \overbar{\tilde{\Pi}}{}^-_i\big)  \ast F^+_{i,j} 
    - \frac{\pi(j)}{\pi(i)} q^-_{j,i}  \underbars{\tilde{F}}{}^-_{j,i} \ast \big((q^+_{j,j}- \dagger^+_j)\delta - d^+_j \delta^\prime + \bbGamma \Pi^+_j \big) \\
    &\qquad - \sum_{k \neq i,j} \frac{\pi(k)}{\pi(i)} q^-_{k,i}q^+_{k,j} \underbars{\tilde{F}}{}^-_{k,i} \ast F^+_{k,j} - q_{i,j}\one_{\R_-} + c_{\text{int}},
  \end{align*}
  where $c_{\text{int}}$ is an integration constant. Similarly to the on-diagonal case above, restricting to $(0,\infty)$ shows that $F^+_{i,j}$ possesses a càdlàg density $f^+_{i,j}$ whenever $d^-_i > 0$ (recall that $F^+_{i,j}$ was chosen to be trivial on $(0,\infty)$ when $q^+_{i,j} = 0)$ and that for a.e.\ $x > 0$  
  \begin{equation}\label{eq: vigon off2}
    \begin{split}
      q_{i,j}F_{i,j}(x,\infty) &= q^+_{i,j}\Big((\dagger^-_i - q^-_{i,i}) \overbar{F}{}^+_{i,j}(x) + d^-_i f^+_{i,j}(x) + \int_{x+}^{\infty} \overbar{\Pi}{}^-_i(y-x)\, F^+_{i,j}(\diff{y}) \Big)\\
      &\quad - \frac{\pi(j)}{\pi(i)} q^-_{j,i} \int_{x+}^{\infty} F^-_{j,i}([0,y-x])\, \Pi^+_j(\diff{y})\\
      &\quad -\sum_{k\neq i,j} \frac{\pi(k)}{\pi(i)}q^-_{k,i}q^+_{k,j}\int_{x+}^\infty F^-_{k,i}([0,y-x]) \, F^+_{k,j}(\diff{y})\\
      &= q^+_{i,j}\Big((\dagger^-_i - q^-_{i,i}) \overbar{F}{}^+_{i,j}(x) + d^-_i f^+_{i,j}(x) + \int_{x+}^{\infty} \overbar{\Pi}{}^-_i(y-x)\, F^+_{i,j}(\diff{y}) \Big)\\
      &\quad +\frac{\pi(j)}{\pi(i)}\Big( q^-_{j,i} \int_{x+}^{\infty} \overbar{F}^-_{j,i}(y-x)\, \Pi^+_j(\diff{y}) - q^-_{j,i} \overbar{\Pi}{}^+_j(x) \Big)\\
      &\quad +\sum_{k\neq i,j} \frac{\pi(k)}{\pi(i)} \Big(q^-_{k,i}q^+_{k,j}\int_{x+}^\infty \overbar{F}^-_{k,i}(y-x) \, F^+_{k,j}(\diff{y}) - q^-_{k,i}q^+_{k,j}\overbar{F}{}^+_{k,j}(x)\Big).
    \end{split}
  \end{equation}
  Above, $c_{\text{int}} = 0$ follows by arguing  as in the on-diagonal case and using \eqref{eq:calc_const2}, \eqref{eq:calc_const3}. 
  Combining \eqref{eq: vigon diag3} and \eqref{eq: vigon off2} yields \eqref{eq: vigon1}. Relation \eqref{eq: vigon2} and the claims on existence of càdlàg densities of $F^-_{i,j}$ and $\Pi^-_i$ whenever $d^+_i > 0$ are proved analogously. 
\end{proof}

\subsection{Characterisation of friendship}

We are now ready to fully characterise friendships of MAPs. Combining Theorem \ref{theo: friends} with Theorem \ref{theo: eq amicales} yields our main result Theorem \ref{theo: main}.
\begin{theorem}[Theorem of friends for MAPs] \label{theo: friends}
  Two MAP subordinators $(H^+,J^+)$ and $(H^-,J^-)$ are $\bm{\pi}$-friends if, and only
  if, they are $\bm{\pi}$-compatible and the matrix-valued function
  \begin{align*}
    \bm{\Upsilon}(x) &= \Big(\int_{x+}^\infty \bm{\Delta}_{\bm{\pi}}^{-1}\Big(\overbar{\bm{\Pi}}^-(y-x) - \bm{\Psi}^-(0)\Big)^\top \bm{\Delta}_{\bm{\pi}} \, \bm{\Pi}^+(\diff{y})  + \bm{\Delta}^-_{\bm{d}} \uppartial \bm{\Pi}^+(x)\Big) \one_{(0,\infty)}(x)\\
    &\quad + \Big(\int_{(-x)+}^\infty \bm{\Delta}_{\bm{\pi}}^{-1} \big(\bm{\Pi}^-(\diff{y}) \big)^\top \bm{\Delta}_{\bm{\pi}} \, \big(\overbar{\bm{\Pi}}{}^+(y+x) - \bm{\Psi}^+(0)\big)  +  \bm{\Delta}_{\bm{\pi}}^{-1} \big(\bm{\Delta}_{\bm{d}}^+\uppartial\bm{\Pi}^-(-x)\big)^\top \bm{\Delta}_{\bm{\pi}}\Big) \one_{(-\infty,0)}(x),
  \end{align*}
  where $x \in \R$, is a.e.\ equal to a function
  decreasing on $(0, \infty)$ and increasing on $(-\infty,0)$.
\end{theorem}
\begin{proof}
  By symmetry, we need only to prove that $(H^+,J^+)$ is a $\bm{\pi}$-friend of $(H^-,J^-)$
  if, and only if, $(H^+,J^+)$ is $\bm{\pi}$-compatible with $(H^-,J^-)$ and
  $\bm{\Upsilon}$ is a.e.\ equal to a function decreasing on $(0,\infty)$ and
  increasing on $(-\infty,0)$.

  Necessity of $\bm{\pi}$-compatibilty is an immediate consequence of the combined conclusions of Lemma \ref{lemma: map exp}, Proposition \ref{prop: dens}, Proposition \ref{prop: comp cond} and Proposition \ref{prop: levy dens}. Necessity of the monotonocity assumptions on $\bm{\Upsilon}$ follows from Theorem \ref{theo: eq amicales} once we notice that when $(H^+,J^+)$ and $(H^-, J^-)$ are $\bm{\pi}$-friends with bonding MAP $(\xi,J)$, we have 
  \[\bm{\Upsilon}(x) = \bm{\Pi}(x,\infty)\one_{(0,\infty)}(x) + \bm{\Pi}(-\infty,x)\one_{(-\infty,0)}(x), \quad x \in \R,\]
  by the équations amicales. Let us therefore turn to sufficiency.

  Condition \ref{friends cond3} 
  of $\bm{\pi}$-compatibility
  ensures that the right hand side of \eqref{eq: wienerhopf} satisfies condition \ref{cond3} of Lemma \ref{lemma: map exp}, so according to the same lemma, it suffices to check the following two properties:
  \begin{enumerate}[label =(\Alph*), ref = (\Alph*)]
    \item The diagonal elements of the right-hand side of \eqref{eq: wienerhopf} can be written as the Lévy--Khintchine exponent of a (killed) Lévy process, which is equivalent to requiring that for any $i \in [n]$, there exists a (positive) measure $\mu_i$ integrating $x \mapsto 1 \wedge x^2$ and constants $c_i \in \R, \mathtt{k}_i, \tau_i \in \R_+$ such that 
      \begin{equation} \label{eq: prop1}
        \begin{split}
          &- \Big(\big((q^-_{i,i} - \dagger^-_i)\delta + d^-_i\delta^\prime + \bbGamma \tilde{\Pi}{}^-_i\big) \ast \big((q^+_{i,i}- \dagger^+_i)\delta - d^+_i \delta^\prime + \bbGamma \Pi^+_i \big) + \sum_{k \neq i} \frac{\pi(k)}{\pi(i)} q^-_{k,i}q^+_{k,i} \tilde{F}{}^-_{k,i} \ast F^+_{k,i}\Big)\\
          &\quad = \bbGamma^2 \mu_i 
          - \mathtt{k}_i \delta 
          - c_i \delta^\prime 
          + \tau_i \delta^{\prime \prime}.
        \end{split}
      \end{equation} \label{prop1}
    \item The off-diagonal elements of \eqref{eq: wienerhopf} constitute the Fourier transform of a finite measure, which is equivalent to requiring that for any $i, j \in [n]$ with $i \neq j$, there exists a finite measure $\mu_{i,j}$ such that 
      \begin{equation} \label{eq: prop2}
        \begin{split}
          \mu_{i,j} &= - \Big(q^+_{i,j} \big((q^-_{i,i} - \dagger^-_i)\delta + d^-_i\delta^\prime + \bbGamma \tilde{\Pi}{}^-_i\big) \ast F^+_{i,j} + \frac{\pi(j)}{\pi(i)} q^-_{j,i}  \tilde{F}{}^-_{j,i} \ast \big((q^+_{j,j}- \dagger^+_j)\delta - d^+_j \delta^\prime + \bbGamma \Pi^+_j \big) \\
          &\qquad + \sum_{k \neq i,j} \frac{\pi(k)}{\pi(i)} q^-_{k,i}q^+_{k,j} \tilde{F}{}^-_{k,i} \ast F^+_{k,j}\Big).
        \end{split}
      \end{equation}\label{prop2}
  \end{enumerate}
  Let us start with \ref{prop1}. According to the first step of the proof of the Théorème des amis in \cite{vigondiss} (at this point \ref{comp cond1} 
  of $\bm{\pi}$-compatibility comes into play), 
  there exists a signed measure $\tilde{\mu}_i$ 
  without an atom at $0$
  integrating $x \mapsto 1 \wedge x^2$ and a constant $\tilde{c}_i \in \R_+$ such that 
  \[-\big((q^-_{i,i} - \dagger^-_i)\delta + d^-_i\delta^\prime + \bbGamma \tilde{\Pi}{}^-_i\big) \ast \big((q^+_{i,i}- \dagger^+_i)\delta - d^+_i \delta^\prime + \bbGamma \Pi^+_i \big) = \bbGamma^2 \tilde{\mu}_i - (q^-_{i,i} - \dagger^-_{i})(q^+_{i,i}- \dagger^+_i)\delta - \tilde{c}_i \delta^\prime + d^-_i d^+_i \delta^{\prime \prime}.\]
  Moreover, $\tilde{\nu}_i = - \sum_{k \neq i} \frac{\pi(k)}{\pi(i)} q^-_{k,i}q^+_{k,i} \tilde{F}{}^-_{k,i} \ast F^+_{k,i}(\cdot \cap \{0\}^{\mathsf{c}})$ is a signed finite measure without atom at $0$ and mass  
  \[\tilde{\nu}_i(\R) = -\sum_{k \neq i} \frac{\pi(k)}{\pi(i)} q^-_{k,i}q^+_{k,i}(1 -\tilde{F}^-_{k,i} \ast F^+_{k,i}(\{0\})),\]
  and thus, $\mu_i \coloneqq \tilde{\mu}_i + \tilde{\nu}_i$ is a signed measure without atom at $0$, integrating $x \mapsto 1 \wedge x^2$. Define
  \[
    \mathtt{k}_i = (q^-_{i,i} - \dagger^-_{i})(q^+_{i,i}- \dagger^+_i) + \sum_{k \neq i} \frac{\pi(k)}{\pi(i)} q^-_{k,i}q^+_{k,i}  \geq 0.
  \]
  Letting further
  $\tau_i = d^-_i d^+_i \geq 0$ and 
  \[
    c_i = \tilde{c}_i - \int_{[-1,1]} x \, \tilde{\nu}_i(\diff{x}) \in \R
  \]
  we obtain \eqref{eq: prop1}. It remains to check that $\mu_i$ is a positive measure. 
  To this end, observe that taking primitives (compare to the proof of Theorem \ref{theo: eq amicales}), we find 
  \begin{align*} 
    &\mathtt{k}_i\one_{\R_{-}} 
    + (c_i + b_i) \delta 
    - \tau_i\delta^\prime + \bbGamma \varrho\overbar{\mu}_i\\
    &\quad =  \big((\dagger^-_i - q^-_{i,i})\one_{\R_{-}} + d^-_i\delta + \overbar{\tilde{\Pi}}{}^-_i\big) \ast \big((q^+_{i,i}- \dagger^+_i)\delta - d^+_i \delta^\prime + \bbGamma \Pi^+_i  \big)\\
    &\qquad - \sum_{k \neq i} \frac{\pi(k)}{\pi(i)} q^-_{k,i}q^+_{k,i} \underbars{\tilde{F}}{}^-_{k,i} \ast F^+_{k,i} + c_{\text{int}},
  \end{align*}
  where $b_i \in \R$ and $c_{\text{int}}$ is an integration constant. Restricting to $(0,\infty)$ and letting $x \to \infty$ shows that $c_{\text{int}} = 0$ and that for a.e.\ $x > 0$,
  \[\mu_i(x,\infty) = \Upsilon_{i,i}(x).\]
  Similarly, we find for a.e.\ $x < 0$
  \[\mu_i(-\infty,x) = \Upsilon_{i,i}(x).\]
  Since ${\Upsilon}_{i,i}$ is a.e.\ increasing on $(0,\infty)$ and a.e.\ decreasing on $(-\infty,0)$ it follows that the tails of $\mu_i$ are a.e.\ increasing on $(-\infty,0)$ and a.e.\ decreasing on $(0,\infty)$.
  Since the tails of $\mu_i$  are càglàd on $(-\infty,0)$ and càdlàg on $(0,\infty)$, this establishes that the tails are increasing on all of $(-\infty,0)$ and decreasing on all of $(0,\infty)$.
  Thus, $\overline{\mu}_i$ is the tail of a positive measure, i.e., $\mu_i$ is a positive measure.

  We proceed with \ref{prop2}.
  Combining condition \ref{friends cond4} of $\bm{\pi}$-compatibility
  with Lemma \ref{lem: dist der}, it follows that 
  \[\big(d^-_i\delta^\prime + \bbGamma \tilde{\Pi}{}^-_i\big) \ast q^+_{i,j}F^+_{i,j} + \big(-d^+_j \delta^\prime + \bbGamma \Pi^+_j \big) \ast \frac{\pi(j)}{\pi(i)} q^-_{j,i}\tilde{F}{}^-_{j,i} = \Big(\tilde{\chi}{}^-_i \ast q^+_{i,j}F^+_{i,j} - \chi^+_j \ast \frac{\pi(j)}{\pi(i)} q^-_{j,i}\tilde{F}{}^-_{j,i}\Big)^\prime = \nu_{i,j},\]
  and thus, 
  \eqref{eq: prop2} indeed defines a finite signed measure. It remains to show that $\mu_{i,j}$ is positive. 
  It follows, again with condition \ref{friends cond4} of $\bm{\pi}$-compatibility,
  \begin{align*}
    \mu_{i,j}(\{0\}) &= (\dagger^-_i- q^-_{i,i})q^+_{i,j} F^+_{i,j}(\{0\}) + (\dagger^+_j-q^+_{j,j}) \frac{\pi(j)}{\pi(i)}q^-_{j,i} F^-_{j,i}(\{0\}) - \nu_{i,j}(\{0\})\\
    &\quad- \sum_{k \neq i,j} \frac{\pi(k)}{\pi(i)} q^-_{k,i} q^+_{k,j} \tilde{F}{}^-_{k,i} \ast F^+_{k,j}(\{0\}) \geq 0.
  \end{align*}
  Moreover, taking primitives and restricting to $(0,\infty)$ we find (compare this to \eqref{eq: vigon off2})  for a.e.\ $x > 0$
  \[\mu_{i,j}(x,\infty) = \Upsilon_{i,j}(x).\]
  Similarly, for a.e.\ $x < 0$,
  \[\mu_{i,j}(-\infty,x) = \Upsilon_{i,j}(x).\]
  Thus, our assumptions on $\bm{\Upsilon}$ guarantee that 
  the tails of $\mu_{i,j}$ are a.e.\ decreasing on $(0,\infty)$ and a.e.\ increasing on $(-\infty,0)$.
  Since the tails are càglàd on $(-\infty,0)$ and càdlàg on $(0,\infty)$ this establishes that the tails are increasing on all of $(-\infty,0)$ and decreasing on all of $(0,\infty)$.
  Together with $\mu_{i,j}(\{0\}) \geq 0$ 
  this establishes that $\mu_{i,j}$ is a finite positive measure.
\end{proof}

In general, not only the central monotonicity conditions on $\bm{\Upsilon}$ but also the necessary requirement of $\bm{\pi}$-compatibility makes engineering MAP friendships significantly harder than for Lévy processes. Especially condition \ref{friends cond4} of Definition \ref{def: comp} appears rather cumbersome and requires skilled matching of the Lévy measure matrices of potential friends. In particular the possible existence of atoms at $0$ for the transitional jumps poses a significant challenge in constructing explicit examples of friendships. We take care of this effect in Section \ref{sec: map philan}, where we develop an extension of Vigon's theory of philanthropy, see \cite[Section 6.6]{kyprianou2014}.

\subsection{Other properties of friendship}

In this section we collect some simple implications of $\bm{\pi}$-friendship.
In particular, in order to interpret the equation \eqref{eq: wienerhopf},
it is essential that the vector $\bm{\pi}$ be the invariant distribution
of (an unkilled version of) the bonding modulator $J$. The results in this section
give some sufficient conditions for this to hold.

\begin{lemma}
  \label{lem:unkilled}
  If $(H^+,J^+)$ and $(H^-,J^-)$ are $\bm{\pi}$-friends 
  and the bonding process $(\xi,J)$ is unkilled, then $\bm{\pi}$ is
  invariant for $J$.
\end{lemma}
\begin{proof}
When $J$ is unkilled, $\bm{\Psi}(0)\one = \bm{0}$.
Together with this,
  the condition $\bm{\pi}^\top\bm{\Psi}(0) \leq \bm{0}^\top$ of $\bm{\pi}$-friendship
  actually implies more:
  $\bm{\pi}^\top\bm{\Psi}(0) = \bm{0}^\top$.
  This means that $\bm{\pi}$ is invariant for $J$. 
\end{proof}

\begin{lemma}\label{lemma: killing}
  Suppose that $(H^+,J^+)$ is a $\bm{\pi}$-friend of $(H^-, J^-)$. 
  Then, the bonding MAP $(\xi,J)$ is unkilled if, and only if, either
  \begin{equation} \label{eq: killing}
    - \bm{\Delta}_{\bm{\pi}}^{-1}\bm{\Psi}^-(0)^\top\bm{\Delta}_{\bm{\pi}}
    \bm{\dag}^+
    = \bm{0}
  \end{equation}
  or
  \begin{equation}\label{eq: killing-dual}
    -\bm{\Delta}_{\bm{\pi}}^{-1} \bm{\Psi}^+(0)^\top \bm{\Delta}_{\bm{\pi}} \bm{\dag}^-
    = \bm{0}.
  \end{equation}
  Moreover, if both $(H^+,J^+)$ and $(H^-,J^-)$ are irreducible, then
  the bonding MAP is unkilled if, and only if, at least one of $(H^+,J^+)$ or
  $(H^-,J^-)$ is unkilled.
\end{lemma}
\begin{proof} 
  On the one hand, we have
  \begin{align*}
    \bm{\dag} = -\bm{\Delta}_{\bm{\pi}}^{-1}\bm{\Psi}^-(0)^\top\bm{\Delta}_{\bm{\pi}}
    \bm{\Psi}^+(0)\one
    = - \bm{\Delta}_{\bm{\pi}}^{-1}\bm{\Psi}^-(0)^\top\bm{\Delta}_{\bm{\pi}}
    \bm{\dag}^+.
  \end{align*}

  On the other hand, considering the bonding MAP $(\hat{\xi},\hat{J})$
  of $(H^-,J^-)$ with $(H^+,J^+)$,
  and denoting its killing rates by $\hat{\bm{\dag}}$, we see by the same argument that
  \[
    \hat{\bm{\dag}} =
    - \bm{\Delta}_{\bm{\pi}}^{-1} \bm{\Psi}^+(0)^\top \bm{\Delta}_{\bm{\pi}}
    \bm{\dag}^-.
  \]
  If this is the zero vector, then 
  by the preceding lemma, this implies that $\bm{\pi}$ is invariant for $\hat{J}$.
  Taking the $\bm{\pi}$-dual of $(\hat{\xi},\hat{J})$, as explained at the start
  of section~\ref{sec: map friends}, we obtain $(\xi,J)$ which is also killed at rate
  $\hat{\bm{\dag}} = \bm{0}$.

  We have proved that the bonding MAP is unkilled if, and only if, either
  \eqref{eq: killing} or \eqref{eq: killing-dual} holds. It follows immediately that, if either $(H^+,J^+)$ or $(H^-,J^-)$ is unkilled,
  then so is the bonding MAP.
  For the converse, suppose that both $(H^+,J^+)$
  and $(H^-,J^-)$ are irreducible and killed. Then $\bm{\Psi}^+(0)$ and $\bm{\Psi}^-(0)$
  are invertible by \cite[Theorem 6.2.26]{horn13}.

  If the bonding MAP were unkilled, then one of \eqref{eq: killing}
  or \eqref{eq: killing-dual} would have to be true; but the former implies
  $\bm{\dag}^+ = \bm{0}$, and the latter implies $\bm{\dag}^- = \bm{0}$,
  both of which are contradictions. Hence, $(\xi,J)$ is killed.
\end{proof}

\begin{lemma}
  Suppose that $(H^+,J^+)$ and $(H^-,J^-)$ are two unkilled MAP subordinators such that
  \eqref{eq: wienerhopf} is the matrix exponent of a MAP $(\xi,J)$.
  Then $(\xi,J)$ is unkilled and $\bm{\pi}$ is invariant for $J$.
\end{lemma}
\begin{proof}
  We observe that
  $\bm{\Psi}(0)\one = \bm{\Delta}_{\bm{\pi}}^{-1} \bm{\Psi}^-(0)^\top
  \bm{\Delta}_{\bm{\pi}} \bm{\dag}^+ = \bm{0}$
  and that
  $\bm{\pi}^\top\bm{\Psi}(0)
  = (\bm{\dag}^-)^\top \bm{\Delta}_{\bm{\pi}} \bm{\Psi}^+(0) = \bm{0}$,
  and the claim follows.
\end{proof}

The significance of the 
(admittedly straightforward) result
above is that it allows one to ignore the condition 
$\bm{\pi}^\top\bm{\Psi}(0)\le \bm{0}^\top$
of $\bm{\pi}$-friendship.

\begin{lemma}\label{lem:invariant_bonding}
  Suppose that $(H^+,J^+)$ and $(H^-,J^-)$ are two subordinator MAP exponents
  such that
  \eqref{eq: wienerhopf} is the matrix exponent of a MAP $(\xi,J)$,
  and let
  $\bm{Q}$ be the generator matrix of (the unkilled version of) the bonding
  modulator $J$; that is,
  \[ 
    - \bm{\Delta}^{-1}_{\bm{\pi}}\bm{\Psi}^-(0)^\top \bm{\Delta}_{\bm{\pi}} \bm{\Psi}^+(0) =  \bm{Q} - \bm{\Delta}_{\bm{\dagger}},
  \]
  where $\dag_i$ is the killing rate of $J$ in state $i$.
  Then, $\bm{\pi}^\top \bm{Q} = \bm{0}^\top$ if, and only if, 
  \begin{equation} \label{eq: vigon inv}
    \bm{\pi}^\top \big( \bm{\Delta}^-_{\bm{\dagger}}\bm{\Psi}^+(0) - \bm{\Delta}^+_{\bm{\dagger}}\bm{\Psi}^-(0) \big) = \bm{0}^\top.
  \end{equation}
\end{lemma}
\begin{proof}
  Since $\bm{Q}$ is a generator matrix we have
  \[\dagger_i =  \sum_{j=1}^n \sum_{k=1}^n \frac{\pi(k)}{\pi(i)} (\bm{\Psi}^-(0))_{k,i} (\bm{\Psi}^+(0))_{k,j}\] 
  and hence 
  \begin{align*}
    (\bm{\pi}^\top \bm{\Delta}_{\bm{\dagger}})_{i} = \pi(i) \dagger_i &= \sum_{j=1}^n\sum_{k=1}^n \pi(k)(\bm{\Psi}^-(0))_{k,i} (\bm{\Psi}^+(0))_{k,j}\\
    &= \sum_{k=1}^n \pi(k)(\bm{\Psi}^-(0))_{k,i} \sum_{j=1}^n (\bm{\Psi}^+(0))_{k,j}\\
    &= -\sum_{k=1}^n \pi(k)(\bm{\Psi}^-(0))_{k,i} \dagger^+_k,
  \end{align*}
  where the last line follows from  $\sum_{j=1}^n q^+_{k,j} = 0$ by definition of a generator matrix.
  Moreover, 
  \begin{align*}
    \big(\bm{\pi}^\top (- \bm{\Delta}^{-1}_{\bm{\pi}}\bm{\Psi}^-(0)^\top \bm{\Delta}_{\bm{\pi}} \bm{\Psi}^+(0))\big)_i &= -\sum_{j=1}^n \big(\bm{\Psi}^-(0)^\top \bm{\Delta}_{\pi} \bm{\Psi}^+(0)\big)_{j,i}\\
    &= -\sum_{j=1}^n \sum_{k=1}^n \pi(k)(\bm{\Psi}^-(0))_{k,j} (\bm{\Psi}^+(0))_{k,i}\\
    &= -\sum_{k=1}^n \pi(k)(\bm{\Psi}^+(0))_{k,i} \sum_{j=1}^n (\bm{\Psi}^-(0))_{k,j} \\
    &= \sum_{k=1}^n \pi(k)(\bm{\Psi}^+(0))_{k,i} \dagger^-_k,
  \end{align*} 
  where the last line is again a consequence of $\sum_{j=1}^n q^-_{k,j} = 0$ by definition of a generator matrix. 
  Thus, $\bm{\pi}^\top \bm{Q} = \bm{0}^\top$ if, and only if, for any $i \in [n]$
  \[0 = (\bm{\pi}^\top \bm{Q})_i = \sum_{k=1}^n \pi(k)\Big(-\bm{\Psi}^-(0)_{k,i} \dagger^+_k + \bm{\Psi}^+(0)_{k,i} \dagger^-_k\Big),\]
  which is satisfied if, and only if, 
  \[\bm{\pi}^\top \big(\bm{\Delta}^-_{\bm{\dagger}}\bm{\Psi}^+(0)  - \bm{\Delta}^+_{\bm{\dagger}}\bm{\Psi}^-(0) \big) = \bm{0}^\top,\]
  that is, if, and only if, $\eqref{eq: vigon inv}$ is satisfied. 
\end{proof}

\section{Markov additive fellowship} \label{sec: map philan}
Having found a characterisation of Markov additive friendships, 
our aim is now to find a version of Vigon's theory of 
philanthropy, as summarized in \cite[Section 6.6]{kyprianou2014}. 
It emerges that the situation is rather more complicated in the 
Markov additive world, and we instead term our relationship
\textit{fellowship}.

For Lévy processes, a philanthropist is a subordinator which is friends with an unkilled
pure drift \cite{vigondiss}, and it emerges that this is equivalent to having
a decreasing Lévy density.

\begin{definition}
  We say that a MAP $(\xi,J)$ is pure drift, if 
  \[\xi_t = \int_0^t d_{J_s} \diff{s}, \quad t \in [0, \zeta).\]
\end{definition}

Friendship of a MAP subordinator with a pure drift MAP subordinator can be characterised as follows with Theorem \ref{theo: friends}.

\begin{lemma}\label{prop: philan friend}
  A MAP subordinator $(H^+,J^+)$ is the $\bm{\pi}$-friend of a pure drift subordinator $(H^-,J^-)$ if, and only if, $(H^+,J^+)$ is $\bm{\pi}$-compatible with $(H^-,J^-)$ and the matrix function
  \[
    - \bm{\Delta}_{\bm{\pi}}^{-1} \bm{\Psi}^-(0)^\top \bm{\Delta}_{\bm{\pi}} \overbar{\bm{\Pi}}{}^+(x) + \bm{\Delta}^-_{\bm{d}} \uppartial \bm{\Pi}^+(x),
    \quad x > 0,
  \]
  is decreasing.
\end{lemma}
\begin{proof}
  This is a direct consequence of Theorem \ref{theo: friends} once we observe that since $\bm{\Pi}^-$ is trivial on $(0,\infty)$ we have 
  \[
    \bm{\Upsilon}(x) = - \bm{\Delta}_{\bm{\pi}}^{-1} \bm{\Psi}^-(0)^\top \bm{\Delta}_{\bm{\pi}} \overbar{\bm{\Pi}}{}^+(x) + \bm{\Delta}^-_{\bm{d}} \uppartial \bm{\Pi}^+(x), \quad x > 0.
    \qedhere
  \]
\end{proof}

For MAPs, neither a decreasing Lévy measure matrix nor $\bm{\pi}$-friendship
with a pure drift is a workable criterion for $\bm{\pi}$-friendship with
another process, or even with another pure drift.  This is already suggested by
necessity of $\bm{\pi}$-compatibility for $\bm{\pi}$-friendship, which entails
specific balance conditions between the characteristics of two friends, and
cannot easily be reduced to a condition on only one of the pair.  But even when
$\bm{\pi}$-compatibility holds, $\bm{\pi}$-friendship may not.  To be specific,
take $n=2$, $\bm{\pi} = (1/2,1/2)^\top$, and consider the example of a MAP
subordinator $(H^+,J^+)$ with transition rate matrix 
$\bm{Q}^+ = \bigl(\begin{smallmatrix} -1 & 1 \\ 1 & -1 \end{smallmatrix}\bigr)$,
drift $+1$ in all phases,
and jumps given by $\Pi_i^+(\diff{x}) = e^{-x}\one_{\{x>0\}} \diff{x}$,
for $i =1,2$
and $F_{i,j}^+(\diff{x}) = \frac{1}{2} (\delta_{\{0\}}(\diff{x}) + e^{-x} \one_{\{x>0\}}\diff{x})$,
for $i,j=1,2, i\ne j$.
Let $(H^-,J^-)$ be the pure drift subordinator with $\bm{Q}^- = \bm{Q}^+$ and
$H^-_t = at$.  By directly checking the conditions of
Theorem~\ref{theo: friends}, one sees that $(H^+,J^+)$ is $\bm{\pi}$-friends
with $(H^-,J^-)$ when $a=2$ but not when $a=1/2$.

We are motivated, therefore, to find conditions not on a single MAP
subordinator, but on a pair, which leads us to the notion of
$\bm{\pi}$-fellowship developed below. 

We say that a MAP subordinator has a continuous, decreasing, differentiable or
convex Lévy density matrix $\uppartial\bm{\Pi}^+$ on $(0,\infty)$ if for every
$i,j \in [n]$, $\bm{\Pi}_{i,j}^+$ has a continuous, decreasing, differentiable
or convex density, respectively, on $(0,\infty)$.  In case of differentiability
we denote by $\uppartial^2 \bm{\Pi}^+$ the matrix-valued function on
$(0,\infty)$ defined by 
\[(\uppartial^2 \bm{\Pi}^+(x))_{i,j} = \frac{\uppartial}{\uppartial x} (\uppartial\bm{\Pi}{}^+(x))_{i,j}, \quad i,j \in [n], x > 0.\]

\begin{definition}\label{def: philan}
  We say that a MAP subordinator $(H^+,J^+)$ is a $\bm{\pi}$-fellow of another
  MAP subordinator $(H^-,J^-)$
  if they have decreasing Lévy density matrices $\uppartial \bm{\Pi}^+$ and
  $\uppartial \bm{\Pi}^-$
  on $(0,\infty)$,
  and the matrix functions
  \begin{equation}\label{eq: philan decr}
    - \bm{\Delta}_{\bm{\pi}}^{-1} \bm{\Psi}^-(0)^\top \bm{\Delta}_{\bm{\pi}} \overbar{\bm{\Pi}}{}^+(x) + \bm{\Delta}^-_{\bm{d}} \uppartial \bm{\Pi}^+(x),
    \quad x > 0,
  \end{equation}
  and
  \begin{equation}\label{eq: philan decr dual}
    - \bm{\Delta}_{\bm{\pi}}^{-1} \bm{\Psi}^+(0)^\top \bm{\Delta}_{\bm{\pi}} \overbar{\bm{\Pi}}{}^-(x) + \bm{\Delta}^+_{\bm{d}} \uppartial \bm{\Pi}^-(x),
    \quad x > 0,
  \end{equation}
  are decreasing. 
\end{definition}
Evidently, $\bm{\pi}$-fellowship is a symmetric relation.
We also note that any two Lévy philanthropists are automatically fellows. 
With this terminology, Lemma \ref{prop: philan friend} can be reformulated in the following form: 

\begin{center} 
  \textit{A MAP subordinator $(H^+,J^+)$ 
    with a decreasing Lévy density matrix on $(0,\infty)$
    is a $\bm{\pi}$-friend of a 
    pure drift MAP subordinator $(H^-,J^-)$ 
    if, and  only if, it is a $\bm{\pi}$-compatible 
  $\bm{\pi}$-fellow of $(H^-,J^-)$.}
\end{center}

\begin{theorem} \label{theo: philan}
  Two $\bm{\pi}$-compatble MAP subordinators 
  that are $\bm{\pi}$-fellows of each other are $\bm{\pi}$-friends.
\end{theorem}
\begin{proof}
  By assumption, $(H^+,J^+)$ is $\bm{\pi}$-compatible with $(H^-,J^-)$. Therefore, it follows from Theorem \ref{theo: friends} that if we can show that $\bm{\Upsilon}$ is decreasing on $(0,\infty)$ and increasing on $(-\infty,0)$, then $(H^+,J^+)$ is a $\bm{\pi}$-friend of  $(H^-,J^-)$. For $x > 0$ we have
  \[\bm{\Upsilon}(x) = \int_{0+}^\infty \bm{\Delta}_{\bm{\pi}}^{-1}\overbar{\bm{\Pi}}^-(y)^\top \bm{\Delta}_{\bm{\pi}} \uppartial\bm{\Pi}{}^+(x+y) \diff{y} -  \bm{\Delta}_{\bm{\pi}}^{-1} \bm{\Psi}^-(0)^\top \bm{\Delta}_{\bm{\pi}} \overbar{\bm{\Pi}}{}^+(x) + \bm{\Delta}^-_{\bm{d}} \uppartial \bm{\Pi}^+(x) \eqqcolon \bm{\Upsilon}^{(1)}(x) + \bm{\Upsilon}^{(2)}(x),\]
  where $\bm{\Upsilon}^{(1)}(x) = \int_{0+}^\infty \bm{\Delta}_{\bm{\pi}}^{-1}\overbar{\bm{\Pi}}^-(y)^\top \bm{\Delta}_{\bm{\pi}} \uppartial\bm{\Pi}{}^+(x+y) \diff{y}$ and $\bm{\Upsilon}^{(2)}(x) = \bm{\Upsilon}(x)-\bm{\Upsilon}^{(1)}(x)$. Since $\uppartial \bm{\Pi}^+$ is decreasing by assumption, it follows that $\bm{\Upsilon}^{(1)}$ is decreasing as well. Moreover, since $(H^+,J^+)$ is a $\bm{\pi}$-fellow of $(H^-,J^-)$, $\bm{\Upsilon}^{(2)}$ is decreasing. Hence,
  $\bm{\Upsilon}$ is decreasing on $(0,\infty)$. For $x < 0$ we have 
  \begin{align*}
    \bm{\Upsilon}(x) &= \int_{0+}^\infty \bm{\Delta}_{\bm{\pi}}^{-1} \uppartial\bm{\Pi}^-(y-x) ^\top \bm{\Delta}_{\bm{\pi}} \overbar{\bm{\Pi}}{}^+(y) \diff{y} - \bm{\Delta}_{\bm{\pi}}^{-1} \big(\overbar{\bm{\Pi}}^-(-x) \big)^\top \bm{\Delta}_{\bm{\pi}}\bm{\Psi}^+(0)  +  \bm{\Delta}_{\bm{\pi}}^{-1} \big(\bm{\Delta}_{\bm{d}}^+\uppartial\bm{\Pi}^-(-x)\big)^\top \bm{\Delta}_{\bm{\pi}}\\
    &\eqqcolon \bm{\Upsilon}^{(3)}(x) + \bm{\Upsilon}^{(4)}(x),
  \end{align*}
  where $\bm{\Upsilon}^{(3)}(x) = \int_{0+}^\infty \bm{\Delta}_{\bm{\pi}}^{-1} \uppartial\bm{\Pi}^-(y-x) ^\top \bm{\Delta}_{\bm{\pi}} \overbar{\bm{\Pi}}{}^+(y) \diff{y}$ and $\bm{\Upsilon}^{(4)}(x) = \bm{\Upsilon}(x)-\bm{\Upsilon}^{(3)}(x)$. By assumption, $\uppartial \bm{\Pi}^-$ is decreasing, which implies that $\bm{\Upsilon}^{(3)}$ is increasing on $(-\infty,0)$. Finally, observe that on $(-\infty,0)$,
  \[\bm{\Delta}_{\bm{\pi}}^{-1} \bm{\Upsilon}^{(4)}(x)^\top \bm{\Delta}_{\bm{\pi}} = - \bm{\Delta}_{\bm{\pi}}^{-1} \bm{\Psi}^+(0)^\top \bm{\Delta}_{\bm{\pi}} \overbar{\bm{\Pi}}^-(-x) + \bm{\Delta}^+_{\bm{d}} \uppartial \bm{\Pi}^-(-x), \quad x < 0,\]
  is increasing by \eqref{eq: philan decr dual}. Since $\bm{\Delta}_{\bm{\pi}}$ is a strictly positive diagonal matrix, this shows that $\bm{\Upsilon}^{(4)}(x)$ is increasing on $(-\infty,0)$. With the above, this now implies that $\bm{\Upsilon}$ is increasing on $(-\infty,0)$ and it follows that $(H^+,J^+)$ is a $\bm{\pi}$-friend of $(H^-,J^-)$. 
\end{proof}

\subsection{Construction of spectrally one-sided MAPs}

In the setting of Lévy processes, the  concept of philanthropy makes the construction
of Lévy processes out of friends simple, since we can simply combine any two Lévy 
subordinators with decreasing Lévy density. 
For MAPs, the story is much more involved. 

The major stumbling block in proving $\bm{\pi}$-friendship is not in
satisfying the decreasing matrix condition of either
Theorem~\ref{theo: friends} or the notion of $\bm{\pi}$-fellowship, but rather in
proving $\bm{\pi}$-compatibility. In order to provide simpler constructions
for MAPs, we focus on finding sufficient conditions for $\bm{\pi}$-compatibility,
beginning with the case of spectrally one-sided MAPs. We set $\bm{F}^\pm \coloneqq (F^\pm_{i,j})_{i,j \in [n]}$ and say that a matrix $\bm{A}$ is an ML-matrix, if it has nonnegative off-diagonal entries.

\begin{definition}\label{def: quasi}
Let $(H^\pm,J^\pm)$ be MAP subordinators with decreasing Lévy density matrices $\bm{\Pi}^\pm$ on $(0,\infty)$ and let $f^\pm_{i,j}$ be càdlàg versions of the Lebesgue densities of $F^\pm_{i,j}$ on $(0,\infty)$. A MAP subordinator $(H^+,J^+)$
is $\bm{\pi}$-quasicompatible with another MAP subordinator $(H^-,J^-)$ 
  if:
  \begin{enumerate}[label = (\roman*), ref =(\roman*)]
    \item for all $i,j \in [n]$ with $i \neq j$, the functions $\psi^\pm_{i,j}$ defined by 
      \begin{align}
      \psi^+_{i,j}(x) &=  \Big(d^-_i q^+_{i,j} f^+_{i,j}(x) - \frac{\pi(j)}{\pi(i)} q^-_{j,i}F^-_{j,i}(\{0\}) \overbar{\Pi}{}^+_j(x) \Big) \one_{(0,\infty)}(x), \quad x \in \R,\label{eq:bv1}\\ 
      \psi^-_{i,j}(x) &=  \Bigl(
        \overbar{\Pi}^-_i(-x) q^+_{i,j} F^+_{i,j}(\{0\}) - \frac{\pi(j)}{\pi(i)} q^-_{j,i} f^-_{j,i}(-x) d^+_j \Bigr) \one_{(0,\infty)}(-x), \quad x \in \R, \label{eq:bv2}
      \end{align}
      are of bounded variation on $\R$. Moreover, for 
      \begin{align} \label{eq: conv0}
        \alpha^+_{i,j}
        \coloneqq \lim_{x \downarrow 0} \psi^+_{i,j}(x), \quad 
        \alpha^-_{i,j} \coloneqq \lim_{x \uparrow 0 }\psi^-_{i,j}(x),
      \end{align}
     the matrix 
      \begin{equation}\label{eq: philan atom}
        -\bm{\Delta}_{\bm{\pi}}^{-1}\big(\bm{\Psi}^-(0) \odot \bm{F}^-(\{0\}\big)^\top\bm{\Delta}_{\bm{\pi}} \big(\bm{\Psi}^+(0) \odot \bm{F}{}^+(\{0\})\big) - \bm{A}
      \end{equation}
      is an ML-matrix, where $\bm{A}_{i,j} = \alpha^+_{i,j} - \alpha^-_{i,j}$ for $i \neq j$ and $\bm{A}_{i,i} = 0$;
      \label{cond philan1}
    \item it holds that 
      \begin{equation}\label{eq: quasifellow}
        \bm{\Delta}^-_{\bm{d}}\bm{\Pi}^+(\{0\}) = \bm{\Delta}_{\bm{\pi}}^{-1} \big(\bm{\Delta}_{\bm{d}}^+\bm{\Pi}^-(\{0\})\big)^\top \bm{\Delta}_{\bm{\pi}};
      \end{equation}  \label{cond philan2}
    \item the vector $\bm{\Delta}_{\bm{\pi}}^{-1}\bm{\Psi}^-(0)^\top\bm{\Delta}_{\bm{\pi}} \bm{\Psi}^+(0)\one$ is nonnegative. \label{cond philan4}
    \item the vector $\bm{\pi}^\top\bm{\Delta}_{\bm{\pi}}^{-1}\bm{\Psi}^-(0)^\top\bm{\Delta}_{\bm{\pi}} \bm{\Psi}^+(0)$ is nonnegative. 
      \label{cond philan5}
  \end{enumerate}
\end{definition}
\begin{remark} 
If $f^+_{i,j}$ is also differentiable on $(0,\infty)$, it is not hard to show that $\psi^+_{i,j}$ has bounded variation if, for some $\varepsilon > 0$, the function $x \mapsto d^-_iq^+_{i,j} \tfrac{\uppartial}{\uppartial x} f^+_{i,j}(x) +\tfrac{\pi(j)}{\pi(i)}q^-_{j,i}F^-_{j,i}(\{0\}) \uppartial \Pi^+_j(x)$ is integrable on $(0,\varepsilon)$. An analogous statement is true for $\psi^-_{i,j}$. These conditions should therefore be understood as a way to say that $d^-_if^+_{i,j}$ and $F^-_{j,i}(\{0\})\overbar{\Pi}^+_j$ must compensate each other appropriately.
\end{remark}

Again, $\bm{\pi}$-quasicompatibility is a symmetric relation:
\begin{lemma}\label{lem: quasi sym}
  If $(H^+,J^+)$ is $\bm{\pi}$-quasicompatible with $(H^-,J^-)$,
  then $(H^-,J^-)$ is $\bm{\pi}$-quasicompatible with $(H^+,J^+)$.
\end{lemma}
\begin{proof}
  It is clearly enough to verify \eqref{eq: philan atom} and \eqref{eq: quasifellow} with swapped roles of $\bm{\Psi}^+$ and $\bm{\Psi}^-$. Since $(H^+,J^+)$ is $\bm{\pi}$-quasicompatible with $(H^-,J^-)$, it follows from \eqref{eq: conv0} that 
  \[\lim_{x \downarrow 0} d^+_i q^-_{i,j} f^-_{i,j}(x) - \frac{\pi(j)}{\pi(i)}q^+_{j,i} F^+_{j,i}(\{0\}) \overbar{\Pi}{}^-_j(x) = - \frac{\pi(j)}{\pi(i)} \alpha^-_{j,i} \eqqcolon \tilde{\alpha}^+_{i,j}\]
  and 
  \[\lim_{x \uparrow 0} \overbar{\Pi}{}^+_i(-x) q^-_{i,j} F^-_{i,j}(\{0\}) - \frac{\pi(j)}{\pi(i)} q^+_{j,i} f^+_{j,i}(-x) d^-_j = - \frac{\pi(j)}{\pi(i)} \alpha^+_{j,i} \eqqcolon \tilde{\alpha}^-_{i,j}.\]
  Thus, if we denote $\tilde{\bm{A}} = (\tilde{\alpha}^+_{i,j}- \tilde{\alpha}^-_{i,j})_{i,j \in [n]}$ with $\tilde{\alpha}^+_{i,i} = \tilde{\alpha}^-_{i,i} = 0$, we have $\tilde{\bm{A}} = \bm{\Delta}_{\bm{\pi}}^{-1} \bm{A}^\top \bm{\Delta}_{\bm{\pi}}$ and therefore
  \begin{align*} 
    &\bm{\Delta}_{\bm{\pi}}^{-1}\Big(-\bm{\Delta}_{\bm{\pi}}^{-1}\big(\bm{\Psi}^+(0) \odot \bm{F}^+(\{0\}\big)^\top\bm{\Delta}_{\bm{\pi}} \big(\bm{\Psi}^-(0) \odot \bm{F}^-(\{0\})\big) - \tilde{\bm{A}}\Big)^\top \bm{\Delta}_{\bm{\pi}}\\
    &\quad= -\bm{\Delta}_{\bm{\pi}}^{-1}\big(\bm{\Psi}^-(0) \odot \bm{F}^-(\{0\}\big)^\top\bm{\Delta}_{\bm{\pi}} \big(\bm{\Psi}^+(0) \odot \bm{F}{}^+(\{0\})\big) - \bm{A}.
  \end{align*}
  Since a matrix $\bm{M}$ is an ML-matrix iff $\bm{\Delta}^{-1}_{\bm{\pi}} \bm{M}^\top \bm{\Delta}_{\bm{\pi}}$ is an ML-matrix and 
  \[-\bm{\Delta}_{\bm{\pi}}^{-1}\big(\bm{\Psi}^-(0) \odot \bm{F}^-(\{0\}\big)^\top\bm{\Delta}_{\bm{\pi}} \big(\bm{\Psi}^+(0) \odot \bm{F}{}^+(\{0\})\big) - \bm{A}\]
  is an ML-matrix by assumption, it follows that 
  \[-\bm{\Delta}_{\bm{\pi}}^{-1}\big(\bm{\Psi}^+(0) \odot \bm{F}^+(\{0\}\big)^\top\bm{\Delta}_{\bm{\pi}} \big(\bm{\Psi}^-(0) \odot \bm{F}{}^-(\{0\})\big) - \tilde{\bm{A}}\]
  is an ML-matrix. Moreover, rearranging \eqref{eq: quasifellow} yields that 
  \[\bm{\Delta}_{\bm{d}}^+ \bm{\Pi}^-(\{0\}) = \bm{\Delta}_{\bm{\pi}}^{-1} \big(\bm{\Delta}^-_{\bm{d}}\bm{\Pi}^+(\{0\})\big)^\top \bm{\Delta}_{\bm{\pi}}\]
  and hence $(H^-,J^-)$ is $\bm{\pi}$-quasicompatible with $(H^+,J^+)$.
\end{proof}

This yields the following characterisation of $\bm{\pi}$-friendship with a pure drift.
\begin{theorem}\label{theo: friend drift}
Let $(H^+,J^+)$ be a MAP subordinator with 
decreasing Lévy density matrix 
on $(0,\infty)$
and $(H^-,J^-)$ be a pure drift MAP subordinator. 
Then, $(H^+,J^+)$ and $(H^-,J^-)$ are $\bm{\pi}$-compatible if, and only if, they are $\bm{\pi}$-quasicompatible. In particular, they are $\bm{\pi}$-friends
if, and only if, they are $\bm{\pi}$-quasicompatible $\bm{\pi}$-fellows.
\end{theorem}
\begin{proof}
  Given Lemma \ref{prop: philan friend} it is enough to show that, under the conditions in the result, the two processes
  are 
  $\bm{\pi}$\mbox{-}\nobreak\hspace{0pt}quasicompatible
  if, and only if, they are $\bm{\pi}$-compatible.
  To this end, first note that condition \ref{comp cond1} of 
  Definition~\ref{def: comp} is automatically satisfied by assumption,
  and that \ref{cond philan4} and \ref{cond philan5} of the definition
  of $\bm{\pi}$-quasicompatibility are the same as 
  conditions \ref{friends cond3} and \ref{friends cond4} in Definition~\ref{def: comp}.
  Moreover, \ref{cond philan2} is necessary for $\bm{\pi}$-compatibility by Corollary \ref{coro: balance atoms}. Let us therefore assume that \ref{cond philan2} holds. It now remains to show that condition \ref{cond philan1} is equivalent to condition \ref{friends cond4} of Definition \ref{def: comp} in the special case of $(H^+,J^+)$ having a decreasing and differentiable Lévy density matrix on $(0,\infty)$ and $(H^-,J^-)$ being pure drift. In this scenario, with \ref{cond philan2} in place we have 
  \begin{align*}
    \vartheta_{i,j}(\diff{x}) &\coloneqq q^+_{i,j} F^+_{i,j} \ast \tilde{\chi}{}^-_i(\diff{x}) - \frac{\pi(j)}{\pi(i)} q^-_{j,i} \tilde{F}{}^-_{j,i} \ast \chi^+_j(\diff{x})\\
    &= \one_{(0,\infty)}(x)\big(d^-_i q^+_{i,j} f^+_{i,j}(x) - \tfrac{\pi(j)}{\pi(i)}q^-_{j,i} \overbar{\Pi}{}^+_j(x) \big) \diff{x},
  \end{align*}
  for $x \in \R$.
  Thus, $\vartheta_{i,j}$ is absolutely continuous and its density is  equal to $\psi^+_{i,j}$. Since $\lim_{x \to \infty} f^+_{i,j}(x) = 0$  by assumption and $\psi^+_{i,j}$ is right-continuous a.e., it follows that $\psi^+_{i,j} = \nu_{i,j}((-\infty,\cdot])$ holds a.e.\ for a finite signed measure $\nu_{i,j}$ if, and only if, $\psi^+_{i,j}$ has bounded variation. Taking into account that $\psi^-_{i,j} \equiv 0$, this shows that \eqref{eq: friend dens} is satisfied if, and only if, $\psi^+_{i,j}$ has bounded variation.
  Finally, $\nu_{i,j}(\{0\}) = \psi^+_{i,j}(0+) = \alpha^+_{i,j}$ and $F^-_{i,j}(\{0\}) = 1$ for all $i \neq j$ yields
  \begin{align*}
    &(\dagger^-_i- q^-_{i,i})q^+_{i,j} F^+_{i,j}(\{0\}) + (\dagger^+_j-q^+_{j,j}) \frac{\pi(j)}{\pi(i)}q^-_{j,i} F^-_{j,i}(\{0\}) - \nu_{i,j}(\{0\})\\
    &\quad- \sum_{k \neq i,j} \frac{\pi(k)}{\pi(i)} q^-_{k,i} q^+_{k,j} F^-_{k,i}(\{0\})F^+_{k,j}(\{0\})\\
    &= (\dagger^-_i- q^-_{i,i})q^+_{i,j} F^+_{i,j}(\{0\}) + (\dagger^+_j-q^+_{j,j}) \frac{\pi(j)}{\pi(i)}q^-_{j,i}  - \alpha_{i,j} - \sum_{k \neq i,j} \frac{\pi(k)}{\pi(i)} q^-_{k,i} q^+_{k,j} F^+_{k,j}(\{0\}),
  \end{align*}
  showing that \eqref{eq: friends atom0} is satisfied if, and only if, the off-diagonal elements of \eqref{eq: philan atom} are nonnegative.
\end{proof}

As a simple consequence of this result we obtain simple criteria for two drifts to be friends, which yields a blue print for generating Markov modulated Brownian motions, i.e., MAPs whose Lévy components are potentially killed linear Brownian motions, having an explicit Wiener--Hopf factorisation. 

\begin{corollary} 
Two pure drift MAPs $(H^+,J^+)$ and $(H^-,J^-)$ are $\bm{\pi}$-friends if, and only if, the matrix $\bm{\Xi} \coloneqq -\bm{\Delta}_{\bm{\pi}}^{-1} \bm{\Psi}^-(0)^\top \bm{\Delta}_{\bm{\pi}} \bm{\Psi}^+(0)$ is an ML-matrix satisfying $\bm{\Xi}\one \leq \bm{0}$, $\bm{\pi}^\top\bm{\Xi} \leq \bm{0}^\top$, and 
\[\pi(i)q^+_{i,j} d^-_i = \pi(j) q^-_{j,i}d^+_j, \quad i,j \in [n], i\neq j.\]
\end{corollary}
\begin{proof} 
  It is clear that $(H^+,J^+)$ and $(H^-,J^-)$ are $\bm{\pi}$-fellows.
  Checking the meaning of $\bm{\pi}$-quasi\-com\-pa\-ti\-bi\-li\-ty, we see that there
  $\psi_{i,j}^\pm \equiv 0$, so this is equivalent to the conditions listed
  in the statement.
  The result follows from Theorem \ref{theo: friend drift}.
\end{proof}

More generally, Theorem \ref{theo: friend drift} can be thought of as a construction principle for spectrally one-sided MAPs with known Wiener--Hopf factorisation. 
Such MAPs have seen numerous applications in the modeling of 
insurance risk, storage models and queuing theory \cite{AA-ruin-prob}.
Having access to the ladder height processes is useful in these cases since the distributional properties of first passage events can be expressed using their characteristics. 
Let us investigate a specific class of examples. Recall that a function $f \colon (0,\infty) \to \R$ is called \textit{completely monotone} if $f \in \mathcal{C}^\infty((0,\infty))$ and for any $n \in \N_0$ it holds that 
\[(-1)^n\frac{\diff^n}{\diff{x^n}} f(x) \geq 0, \quad x > 0.\]
By Bernstein's Theorem, see \cite[Theorem 1.4]{schilling2012}, $f \in \mathcal{C}^\infty((0,\infty))$ is completely monotone if, and only if, there exists a \textit{representing measure} $\mu$ on $(\R_+,\mathcal{B}(\R_+))$ such that 
\[f(x) = \int_{\R_+} \mathrm{e}^{-xy}\, \mu(\diff{y}), \quad x > 0,\]
i.e., $f$ is given as the Laplace transform of $\mu$.
A completely monotone function $f$ is the Lévy density of some subordinator if, and only if,
\begin{equation}
  \label{eq:cbf}
  \mu(\{0\}) = 0
  \quad \text{and}\quad 
  \int_{(0,1)} y^{-1} \, \mu(\diff{y}) + \int_{(1,\infty)} y^{-2} \, \mu(\diff{y}) < \infty
\end{equation}
hold \cite[Theorem 6.2]{schilling2012}.

Examples giving rise to completely monotone Lévy density include jumps of exponential size
($\mu$ a multiple of a Dirac mass), jumps of mixed exponential size
($\mu$ a discrete measure with finitely many atoms), subordinators
in the `meromorphic class' of Lévy processes \cite{kkp2012}
($\mu$ a discrete measure);
see \cite[chapter~16]{schilling2012}
for an extensive list.

\begin{example}\label{ex: spec}
  Let $\bm{Q}^-,\bm{Q}^+ \in \R^{2 \times 2}$ be irreducible generator matrices
  (of unkilled Markov processes)
  and $\bm{\pi} \in (0,1]^2$ be a stochastic vector. Let $\Pi^+_i$ be absolutely continuous with completely monotone densities with representing measures $\mu^+_i$ for $i=1,2$ each satisfying \eqref{eq:cbf}, i.e.,
  \[\Pi^+_i(\diff{x}) = \int_{\R_+} \mathrm{e}^{-xy} \, \mu^+_i(\diff{y}) \diff{x}, \quad x > 0, i =1,2,\]
  and $\bm{d}^+, \bm{d}^- \in (0,\infty)^2$ such that 
  \begin{enumerate}[label = (\roman*), ref = (\roman*)]
    \item \label{cond: spec0}
      for $i = 1,2$,
      \[
        \int_{(0,1)} y^{-3} \, \mu_i^+(\diff{y}) < \infty;
      \]
    \item \label{cond: spec1}
      \[d^+_1 + \int_{0+}^\infty \overbar{\Pi}{}^+_1(x)\diff{x} = \frac{\pi(2)}{\pi(1)} \frac{q^+_{2,1}d^-_2}{q^-_{1,2}}, \quad d^+_2 + \int_{0+}^\infty \overbar{\Pi}{}^+_2(x)\diff{x} = \frac{\pi(1)}{\pi(2)} \frac{q^+_{1,2}d^-_1}{q^-_{2,1}};\]
    \item \label{cond: spec2} for any $x > 0$, and $i\ne j$,
      \[\int_{\R_+} \Big(1 + \frac{d_i^-}{q_{i,j}^-}y - \frac{q_{j,i}^-}{d_j^-} \frac{1}{y}\Big) \mathrm{e}^{-xy}\, \mu^+_{i}(\diff{y}) > 0.\] 
  \end{enumerate}
  For example, when the representing measures are supported away from zero, i.e., $\mathrm{supp}(\mu^+_i) \subset (a^+_i,\infty)$ for some $a^+_i > 0$, then \ref{cond: spec2} is satisfied whenever 
  \[a^+_i\Big(1+ \frac{d_i^-}{q_{i,j}^-}a^+_i \Big) > \frac{q_{j,i}^-}{d_j^-}\]
  and thus examples fufilling both of the above conditions can be easily constructed by choosing large enough and matching drifts $\bm{d}^+,\bm{d}^-$.

  Choose now probability measures $F^+_{1,2}, F^+_{2,1}$ with support $\R_+$ and decreasing and differentiable densities $f^+_{1,2}, f^+_{2,1}$ on $(0,\infty)$ such that 
  \[f^+_{1,2}(x) = \frac{\pi(2)}{\pi(1)}\frac{q_{2,1}^-}{q^+_{1,2}d_1^-}\overbar{\Pi}^+_2(x), \quad f^+_{2,1}(x) = \frac{\pi(1)}{\pi(2)}\frac{q_{1,2}^-}{q^+_{2,1}d_2^-}\overbar{\Pi}^+_1(x),\]
  for $x > 0$. Then, condition \ref{cond philan1} of Definition \ref{def: quasi} is fulfilled with $\psi_{1,2} = \psi_{2,1} \equiv 0$ and $\alpha_{1,2} = \alpha_{2,1} = 0$ and condition \ref{cond philan2} of Definition \ref{def: quasi} holds with 
  \begin{equation} \label{eq: atoms spec pos}
    F^+_{1,2}(\{0\}) = \frac{\pi(2)}{\pi(1)}\frac{q^-_{2,1}d^+_2}{q^+_{1,2}d^-_1}, \quad F^+_{2,1}(\{0\}) = \frac{\pi(1)}{\pi(2)}\frac{q^-_{1,2}d^+_1}{q^+_{2,1}d^-_2}.
  \end{equation}
  Our assumption~\ref{cond: spec1} ensures that $F_{i,j}^+$ are probability measures
  for each $i\ne j$.
  Moreover, it is easy to see that \eqref{eq: philan atom} is satisfied since $\bm{A} = \bm{0}_{2 \times 2}$ and the first term is always an ML-matrix when the matrices have dimension $2 \times 2$. Let now $(H^+,J^+)$ be a MAP subordinator with characteristic matrix exponent $\bm{\Psi}^+$ associated to the drift vector $\bm{d}^+$ and Lévy measure matrix 
  \[\bm{\Pi}^+ = \begin{bmatrix} \Pi^+_1 & q^+_{1,2} F^+_{1,2} \\ q^+_{2,1} F^+_{2,1} & \Pi^+_2 \end{bmatrix}\]
  and $(H^-,J^-)$ the pure drift MAP associated to $\bm{d}^-$ and generator matrix $\bm{Q}^-$. 
  Straightforward calculations show that for any $x > 0$,
  \begin{equation}\label{eq: dens spec pos}
    \begin{split}
      &\frac{\uppartial}{\uppartial x} \Big\{\bm{\Delta}_{\bm{\pi}}^{-1} \bm{\Psi}^-(0)^\top \bm{\Delta}_{\bm{\pi}} \overbar{\bm{\Pi}}^+(x) - \bm{\Delta}_{\bm{d}}^- \uppartial \bm{\Pi}^+(x) \Big\}\\
      &\quad =-\bm{\Delta}_{\bm{\pi}}^{-1}\bm{\Psi}^-(0)^\top \bm{\Delta}_{\bm{\pi}} \uppartial \bm{\Pi}^+(x) - \bm{\Delta}^-_{\bm{d}} \uppartial^2 \bm{\Pi}^+(x)\\
      &\quad= \begin{bmatrix} 
        q_{1,2}^- \int_{\R_+} \Big(1 + \frac{d_1^-}{q_{1,2}^-}y - \frac{q_{2,1}^-}{d_2^-} \frac{1}{y}\Big) \mathrm{e}^{-xy}\, \mu^+_{1}(\diff{y}) & \frac{\pi(2)}{\pi(1)}\frac{q_{1,2}^-q_{2,1}^-}{d_1^-} \int_{\R_+} \frac{1}{y} \mathrm{e}^{-xy}\, \mu^+_{2}(\diff{y})\\[6pt]
        \frac{\pi(1)}{\pi(2)}\frac{q_{1,2}^-q_{2,1}^-}{d_2^-} \int_{\R_+} \frac{1}{y} \mathrm{e}^{-xy}\, \mu^+_{1}(\diff{y}) & q_{2,1}^- \int_{\R_+} \Big(1 + \frac{d_2^-}{q_{2,1}^-}y - \frac{q_{1,2}^-}{d_1^-} \frac{1}{y}\Big) \mathrm{e}^{-xy}\, \mu^+_{2}(\diff{y})
      \end{bmatrix},
    \end{split}
  \end{equation}
  which is nonnegative by \ref{cond: spec2},
  directly implying that our two processes are $\bm{\pi}$-fellows.
  Property \ref{cond philan4} of $\bm{\pi}$-quasicompatibility is satisfied
  since $(H^+,J^+)$ is unkilled, and property \ref{cond philan5} is satisfied since
  $(H^-,J^-)$ is unkilled.
  Taking all this into account, 
  we see that $(H^+,J^+)$ is a $\bm{\pi}$-quasicompatible
  $\bm{\pi}$-fellow of $(H^-,J^-)$ 
  and hence, with Proposition \ref{theo: friend drift},  that the MAP subordinator $(H^+,J^+)$ is a $\bm{\pi}$-friend of the pure drift $(H^-,J^-)$. Moreover, since $q^+_{i,j} = - q^+_{i,i}$ and $q^-_{i,j} = - q^-_{i,i}$ for $i,j \in \{1,2\}$ with $i \neq j$, it follows that for the bonding MAP $(\xi,J)$, the matrix
  \[\bm{Q} = -\bm{\Delta}_{\bm{\pi}}^{-1} (\bm{Q}^-)^\top \bm{\Delta}_{\bm{\pi}} \bm{Q}^+ = \begin{bmatrix} -q^+_{1,1}q^-_{1,1} - \frac{\pi(2)}{\pi(1)} q^+_{2,2} q^-_{2,2} & q^+_{1,1} q^-_{1,1} + \frac{\pi(2)}{\pi(1)} q^+_{2,2} q^-_{2,2} \\[6pt]
  q^+_{2,2}q^-_{2,2} + \frac{\pi(1)}{\pi(2)} q^+_{1,1} q^-_{1,1} & -q^+_{2,2}q^-_{2,2} - \frac{\pi(1)}{\pi(2)} q^+_{1,1} q^-_{1,1}\end{bmatrix}\]
  is an irreducible generator matrix,
  and hence $\bm{\pi}$ is the unique invariant distribution of $J$.
  From \cite[Lemma 3.20]{doering21} we know that the Lévy components $\xi^{(1)}, \xi^{(2)}$ of $\xi$ have non-trivial Brownian part with scaling factor $\sigma_i^2 = 2 d^+_id^-_i$ for $i =1,2$. 
  Furthermore, combining \eqref{eq: atoms spec pos} and \eqref{eq: dens spec pos} we conclude with Proposition \ref{prop: comp cond} and Theorem \ref{theo: eq amicales} that the Lévy measure matrix of $(\xi,J)$ is given by 
  \begin{align*}
    &\bm{\Pi}(\diff{x})\\
    &\hspace{1pt}= \left[\scalemath{0.7}{\begin{array}{cc} \one_{\{x > 0\}}q_{1,2}^- \int_{\R_+} \Big(1 + \frac{d_1^-y}{q_{1,2}^-} - \frac{q_{2,1}^-}{d_2^-y}\Big) \mathrm{e}^{-xy}\, \mu^+_{1}(\diff{y})\diff{x} & \frac{\pi(2)}{\pi(1)}q_{2,2}^- \Big\{\Big(q^+_{2,2} + q^-_{1,1}\frac{d^+_2}{d^-_1} \Big)\delta_0(\diff{x}) + \one_{\{x>0\}}\frac{q_{1,1}^-}{d_1^-} \int_{\R_+} \frac{\mathrm{e}^{-xy}}{y}\, \mu^+_{2}(\diff{y})\diff{x}\Big\} \\[6pt]
          \frac{\pi(1)}{\pi(2)}q_{1,1}^- \Big\{\Big(q^+_{1,1} + q^-_{2,2}\frac{d^+_1}{d^-_2} \Big)\delta_0(\diff{x}) + \one_{\{x>0\}} \frac{q_{2,2}^-}{d_2^-} \int_{\R_+} \frac{\mathrm{e}^{-xy}}{y}\, \mu^+_{1}(\diff{y})\diff{x}\Big\} & \one_{\{x>0\}}q_{2,1}^- \int_{\R_+} \Big(1 + \frac{d_2^-y}{q_{2,1}^-} - \frac{q_{1,2}^-}{d_1^-y}\Big) \mathrm{e}^{-xy}\, \mu^+_{2}(\diff{y})\diff{x}
    \end{array}}\right].
  \end{align*}
  As outlined in Remark~\ref{r:unique-examples}, the MAP Wiener--Hopf factorisation is unique for this process,
  and so the ladder height processes for the bonding MAP $(\xi,J)$ are indeed
  $(H^+,J^+)$ and $(H^-,J^-)$.
\end{example}

\subsection{Construction of MAPs jumping in both directions}

Moving away from the spectrally one-sided assumption, we structure our analysis
as follows. Firstly, we concentrate on finding explicit criteria for 
$\bm{\pi}$-compatibility of two pure jump MAP subordinators with diffuse
jump structure. Secondly, we show that the $\bm{\pi}$-friendship of
two general MAP subordinators can be studied by splitting them into
their diffuse pure jump parts and the remainder, for which the notion of
$\bm{\pi}$-quasicompatibility can be used to simplify conditions.

We begin with the diffuse pure jump part, as follows.

\begin{proposition}\label{prop: fellow}
  Let $(H^+,J^+)$ and $(H^-,J^-)$ be MAP subordinators with decreasing, differentiable and convex Lévy density matrices on $(0,\infty)$ such that 
  \[\bm{\Pi}{}^+(\{0\}) = \bm{\Pi}^-(\{0\}) = \bm{\Delta}^+_{\bm{d}} = \bm{\Delta}^-_{\bm{d}} = \mathbf{0}_{n \times n}.\]
  Then, $(H^+,J^+)$ is $\bm{\pi}$-compatible with $(H^-,J^-)$ if 
  \begin{enumerate}[label = (\roman*), ref =(\roman*)]
    \item for any $i \in [n]$ we have $\E^0\big[H^{+,(i)}_1\big] < \infty$ and $\E^0\big[H^{-,(i)}_1\big] < \infty$;\label{cond philan12}
    \item for $i \neq j$ we define
      \[ 
        v_{i,j}(x) \coloneqq q^+_{i,j} f^+_{i,j}(x) \overbar{\Pi}{}^-_i(x), \quad w_{i,j}(x) \coloneqq \frac{\pi(j)}{\pi(i)} q^-_{j,i} f^-_{j,i}(x) \overbar{\Pi}^+_j(x), 
      \]
      for $x > 0$, then $v_{i,j}, w_{i,j} \in L^1((0,\infty))$; \label{cond philan13}
    \item the vector $\bm{\Delta}_{\bm{\pi}}^{-1}\bm{\Psi}^-(0)^\top\bm{\Delta}_{\bm{\pi}} \bm{\Psi}^+(0)\one$ is nonnegative;\label{cond philan14}
    \item
      the vector $\bm{\pi}^\top\bm{\Delta}_{\bm{\pi}}^{-1}\bm{\Psi}^-(0)^\top\bm{\Delta}_{\bm{\pi}} \bm{\Psi}^+(0)$ is nonnegative.
  \end{enumerate}
\end{proposition}
\begin{proof}
  Since the transitional jumps have no atom at $0$, this boils down to showing that under \ref{cond philan12} and \ref{cond philan13}, \eqref{eq: friend dens} holds for some finite signed measure $\nu_{i,j}$ with $\nu_{i,j}(\{0\}) = \nu_{i,j}(\R) = 0.$ Fix $i,j \in [n]$ with $i \neq j$. Since $F^+_{i,j}(\{0\}) = 0$ and $d^-_i = 0$, we obtain 
  \begin{align*}
    \vartheta^{(1)}_{i,j}(\diff{x}) &\coloneqq q^+_{i,j} F^+_{i,j} \ast \tilde{\chi}{}^-_i(\diff{x})\\
    &= \one_{\R \setminus \{0\}}(x) \int_{0+}^{\infty} \one_{(0,\infty)}(y-x) q^+_{i,j} f^+_{i,j}(y) \overbar{\Pi}{}^-_i(y-x) \diff{y} \diff{x}\\
    &= \Big[\one_{(-\infty,0)}(x) \int_{0+}^{\infty} q^+_{i,j} f^+_{i,j}(y) \overbar{\Pi}{}^-_i(y-x) \diff{y}  + \one_{(0,\infty)}(x) \int_{x+}^\infty q^+_{i,j} f^+_{i,j}(y) \overbar{\Pi}{}^-_i(y-x) \diff{y} \Big]\diff{x}\\
    &= \Big[\one_{(-\infty,0)}(x) \int_{0+}^{\infty} q^+_{i,j} f^+_{i,j}(y) \overbar{\Pi}{}^-_i(y-x) \diff{y}  + \one_{(0,\infty)}(x) \int_{0+}^\infty q^+_{i,j} f^+_{i,j}(y+x) \overbar{\Pi}{}^-_i(y) \diff{y} \Big]\diff{x}\\
    &\eqqcolon \big(\one_{(-\infty,0)}(x) \theta^-_{i,j}(x) + \one_{(0,\infty)}(x) \theta^+_{i,j}(x)\big)\diff{x}.
  \end{align*}
  Now, since $f^+_{i,j}$ is decreasing by assumption and $\overbar{\Pi}{}^-$ is decreasing as well, it follows by monotone convergence that 
  \[\lim_{x \uparrow 0} \theta^-_{i,j}(x)  = \lim_{x \downarrow 0} \theta^+_{i,j}(x) = \int_{0+}^\infty v_{i,j}(y) \diff{y} < \infty.\]
  Thus, if we define 
  \[\theta_{i,j}(x) = \begin{cases} \theta^-_{i,j}(x), &x < 0, \\ \int_{0+}^\infty v_{i,j}(y) \diff{y}, &x= 0,\\ \theta^+_{i,j}(x), &x > 0, \end{cases}\]
  then $\vartheta_{i,j}^{(1)}$ is absolutely continuous with continuous density $\theta_{i,j}$. 
  Using that $\uppartial\bm{\Pi}^-$ decreases on $(0,\infty)$, we find for any $y > 0$ and  $x < 0$ that
  \[\frac{\uppartial}{\uppartial x} q^+_{i,j} f^+_{i,j}(y) \overbar{\Pi}{}^-_i(y-x) = q^+_{i,j} f^+_{i,j}(y) \uppartial \Pi^-_i(y-x) \leq q^+_{i,j} f^+_{i,j}(y) \uppartial \Pi^-_i(-x).\]
  The Lévy density $\uppartial\Pi^-_i$ is bounded away from zero by our assumptions and $f^+_{i,j} \in L^1((0,\infty))$, implying that we may differentiate under the integral such that $\theta_{i,j}$ is differentiable on $(-\infty,0)$ with 
  \[\theta^\prime_{i,j}(x) = \int_{0+}^\infty q^+_{i,j} f^+_{i,j}(y) \uppartial \Pi^-_i(y-x) \diff{y}, \quad x < 0.\]
  Moreover, using convexity of $\bm{\Pi}^+$ on $(0,\infty)$, it follows for any $x> 0$ and $y > 0$ that 
  \[\big\lvert \frac{\uppartial}{\uppartial x} q^+_{i,j} f^+_{i,j}(x+y) \overbar{\Pi}{}^-_i(y) \big\rvert = -q^+_{i,j} \tfrac{\uppartial}{\uppartial x} f^+_{i,j}(x+y) \overbar{\Pi}{}^-_i(y) \leq -q^+_{i,j} \tfrac{\uppartial}{\uppartial x} f^+_{i,j}(x) \overbar{\Pi}{}^-_i(y).\]
  Local boundedness of $\tfrac{\uppartial}{\uppartial x} f^+_{i,j}(x) $ and $\overbar{\Pi}{}^-_i \in L^1((0,\infty))$ thanks to \ref{cond philan12} now also imply that $\theta_{i,j}$ is differentiable on $(0,\infty)$ with derivative 
  \[\theta^\prime_{i,j}(x) = \int_{0+}^\infty q^+_{i,j} \tfrac{\uppartial}{\uppartial x} f^+_{i,j}(x+y) \overbar{\Pi}{}^-_i(y) \diff{y}, \quad x > 0.\]
  Thus, if we let $\nu^{(1)}_{i,j}$ be the signed measure with density $\theta^\prime_{i,j} \one_{\R \setminus \{0\}}$, it follows that $\theta_{i,j}(x) = \nu^{(1)}_{i,j}((-\infty,x])$ for $x \in \R$. 
  Finiteness of $\nu_{i,j}^{(1)}$ follows now from the fact that 
  \[\lvert \nu^{(1)}_{i,j}\rvert(\R_-) = \theta_{i,j}(0) = \lvert \nu^{(1)}_{i,j}\rvert(\R_+)\]
  and $\theta_{i,j}(0) = \int_{0+}^\infty v_{i,j}(y) \diff{y} < \infty$. We therefore conclude that there exists a finite signed measure $\nu^{(1)}_{i,j}$ without atom at $0$ such that $\vartheta^{(1)}_{i,j}$ is absolutely continuous with density $\nu^{(1)}_{i,j}((-\infty,\cdot])$. Similarly, it follows that for 
  \[\vartheta^{(2)}_{i,j}(\diff{x}) = \frac{\pi(j)}{\pi(i)} q^-_{j,i} F^-_{j,i} \ast \chi^+_{j}(\diff{x}), \quad x \in \R,\]
  there exists some finite signed measure $\nu^{(2)}_{i,j}$ without atom at $0$ such that $\vartheta_{i,j}^{(2)}$ is absolutely continuous with density $\nu_{i,j}^{(2)}((-\infty,\cdot])$. Finally, since 
  \[q^+_{i,j} F^+_{i,j} \ast \tilde{\chi}{}^-_i - \frac{\pi(j)}{\pi(i)} q^-_{j,i} \tilde{F}{}^-_{j,i} \ast \chi^+_{j} = \vartheta_{i,j}^{(1)} - \vartheta_{i,j}^{(2)},\]
  it follows that \eqref{eq: friend dens} is satisfied for the finite signed measure $\nu_{i,j} = \nu^{(1)}_{i,j} - \nu^{(2)}_{i,j}$, which has no atom at $0$ and satisfies $\nu_{i,j}(\R) = 0$ since by monotone convergence $\lim_{x \to \infty} \nu_{i,j}^{(1)}((-\infty,x]) = \lim_{x \to \infty} \theta^+_{i,j}(x) = 0$, and similarly $\lim_{x \to \infty} \nu_{i,j}^{(2)}((-\infty,x]) = 0$. This proves the assertion.
\end{proof}

We can use
Theorem~\ref{theo: friend drift} to give simpler conditions for $\bm{\pi}$-compatibility
with a drift, and Proposition~\ref{prop: fellow} to deal with $\bm{\pi}$-compatibility
of processes with diffuse pure jump structure.
In general, we can reduce the $\bm{\pi}$-compatibility condition to these special
cases by splitting up the structure of the MAPs.
To this end, let $(H^\pm,J^\pm)$ MAPs such that $F^\pm_{i,j}(\{0\}) \neq 1$ for all $i \neq j$ and define MAPs $(H^{\pm,\circ}, J^{\pm,\circ})$ by setting $\bm{\Psi}^{\pm,\circ}(0) = \bm{\Psi}^{\pm}(0)$, $\bm{\Delta}_{\bm{d}^{\pm,\circ}} = \bm{\Pi}^{\pm,\circ}(\{0\}) = \bm{0}_{n\times n}$ and 
\[\bm{\Pi}^{\pm,\circ}(\diff{x}) = \bm{\Theta}^{\pm} \odot \bm{\Pi}^{\pm}(\diff{x}), \quad x > 0,\]
where $\bm{\Theta}^{\pm}_{i,i} = 1$ for any $i \in [n]$ and
\[\bm{\Theta}^{\pm}_{i,j} = \frac{1}{F^+_{i,j}((0,\infty))}, \quad i,j\in [n], i \neq j.\]
In other words, $(H^{\pm,\circ},J^{\pm,\circ})$ are MAPs obtained from $(H^{\pm},J^{\pm})$ by eliminating the drift part and conditioning the transitional jumps to be strictly positive. 
Moreover, if the Lévy density matrices of $(H^+,J^+)$ and $(H^-,J^-)$ are decreasing, differentiable and convex on $(0,\infty)$, the MAPs $(H^{\pm,\circ},J^{\pm,\circ})$ fall into the class of MAPs considered in Proposition \ref{prop: fellow}, where constructive criteria for 
$\bm{\pi}$\mbox{-}\nobreak\hspace{0pt}compatibility are established. 
\begin{theorem}\label{theo: quasi help}
  Let $(H^+,J^+)$ and $(H^-,J^-)$ be MAP subordinators with continuous Lévy density matrices on $(0,\infty)$ and such that $F^\pm_{i,j}(\{0\}) \neq 1$ for all $i,j \in [n]$ with $i\ne j$.
  If $(H^+,J^+)$ and $(H^-,J^-)$  are $\bm{\pi}$-quasicompatible
  $\bm{\pi}$-fellows,
  and either
  \begin{enumerate}[ref={(\roman*)}, label={(\roman*)}]
    \item\label{i:qh:hyp-prop}
      the conditions of Proposition~\ref{prop: fellow} are satisfied for
      $(H^{+,\circ},J^{+,\circ})$ and $(H^{-,\circ},J^{-,\circ})$, or
    \item\label{i:qh:hyp-compat}
      $\bm{F}^+(\{0\}) = (\bm{F}^-(\{0\}))^\top$ 
      and $(H^{+,\circ},J^{+,\circ})$ and $(H^{-,\circ},J^{-,\circ})$ are 
      $\bm{\pi}$-compatible, 
  \end{enumerate}
  then $(H^+,J^+)$ and $(H^-,J^-)$ are $\bm{\pi}$-friends.
\end{theorem}

The following lemma will be used in the proof of this theorem.

\begin{lemma}\label{lemma: nodrifts}
  Suppose that $(H^+,J^+)$ and $(H^-,J^-)$ are MAP subordinators with decreasing,
  continuous Lévy density matrices on $(0,\infty)$ such that 
  \[\bm{\Pi}{}^+(\{0\}) = \bm{\Pi}^-(\{0\}) = \bm{\Delta}^+_{\bm{d}} = \bm{\Delta}^-_{\bm{d}} = \mathbf{0}_{n \times n}.\] 
  If $(H^+,J^+)$ is $\bm{\pi}$-compatible with $(H^-,J^-)$, then $\underbars{\nu}_{i,j}$ from Definition \ref{def: comp} is a continuous function for any $i,j \in [n]$ with $i \neq j$.
\end{lemma}
\begin{proof}
  The $\bm{\pi}$-compatibility of the two MAP subordinators implies that there is a finite
  signed measure $\nu_{i,j}$ such that
  $\underbars{\nu}_{i,j}(x) = \vartheta^{(1)}_{i,j}(\diff{x}) - \vartheta^{(2)}_{i,j}(\diff{x})$,
  in the sense of distributions and using notation from
  the proof of Proposition \ref{prop: fellow}.
  Considering the representation of $\vartheta^{(1)}_{i,j}$ given in said proof,
  we see that, under our assumptions, it is absolutely continuous with continuous
  density. The same holds for $\vartheta^{(2)}_{i,j}$, and this shows
  that $\underbars{\nu}_{i,j}$ is continuous.
\end{proof}

\begin{proof}[Proof of Theorem~\ref{theo: quasi help}]
  We will show that $(H^+,J^+)$ and $(H^-,J^-)$ are $\bm{\pi}$-compatible
  $\bm{\pi}$-fellows, and the
  conclusion will then follow from Theorem~\ref{theo: philan}.
  It is enough to establish $\bm{\pi}$-compatibility of $(H^+,J^+)$ with $(H^-,J^-)$, which under the given assumptions boils down to showing that condition \ref{friends cond4} of Definition \ref{def: comp} holds. To this end, let us write 
  \[
    \vartheta_{i,j}(\diff{x}) \coloneqq q^+_{i,j} F^+_{i,j} \ast \tilde{\chi}{}^-_i(\diff{x}) - \frac{\pi(j)}{\pi(i)} q^-_{j,i} \tilde{F}{}^-_{j,i} \ast \chi^+_j(\diff{x}), \quad x \in \R,
  \]
  and notice that since $(H^+,J^+)$ is $\bm{\pi}$-quasicompatible with $(H^-,J^-)$ we have 
  \[d^-_iq^+_{i,j}F^+_{i,j}(\{0\}) = \tfrac{\pi(j)}{\pi(i)}d^+_jq^-_{j,i}F^-_{j,i}(\{0\}),\]
  such that we may write $\vartheta_{i,j}(\diff{x}) = \vartheta_{i,j}^{\circ}(\diff{x}) + \vartheta_{i,j}^{\rightsquigarrow}(\diff{x})$, where
  \begin{align*}
    \vartheta_{i,j}^{\rightsquigarrow}(\diff{x}) &= \one_{(0,\infty)}(x)\big(d^-_i q^+_{i,j} f^+_{i,j}(x) - \tfrac{\pi(j)}{\pi(i)}q^-_{j,i} F^-_{j,i}(\{0\})\overbar{\Pi}{}^+_j(x) \big) \diff{x}\\
    &\quad + \one_{(-\infty,0)}(x)\big( \overbar{\Pi}^-_i(-x) q^+_{i,j} F^+_{i,j}(\{0\}) - \tfrac{\pi(j)}{\pi(i)} q^-_{j,i} f^-_{j,i}(-x) d^+_j \big) \diff{x}, \quad x \in \R,
  \end{align*}
  and 
  \begin{align*}
    \vartheta_{i,j}^\circ(\diff{x}) 
    &=  F^+_{i,j}((0,\infty))q^+_{i,j} F^{+,\circ}_{i,j} \ast \tilde{\chi}{}^{-,\circ}_i(\diff{x}) - F^-_{j,i}((0,\infty))\frac{\pi(j)}{\pi(i)} q^-_{j,i} \tilde{F}{}^{-,\circ}_{j,i} \ast \chi^{+,\circ}_j(\diff{x}), \quad x\in \R,
  \end{align*}
  where $\chi^{\pm,\circ}$ play the same role for $(H^{\pm,\circ},J^{\pm,\circ})$ as do the measures $\chi^\pm$ for $(H^\pm,J^\pm)$.
  If hypothesis \ref{i:qh:hyp-prop} holds, then both summands in $\vartheta^\circ_{i,j}$
  represent distribution functions of a finite signed measure, and checking
  the proof of said proposition, we can see that these distribution functions are
  continuous; hence, there exists some finite signed measure $\nu^\circ_{i,j}$ with continuous distribution function $\underbars{\nu}^\circ_{i,j}$ such that
  \begin{equation}\label{e:nucirc}
    \vartheta_{i,j}^\circ(\diff{x}) = \underbars{\nu}^\circ_{i,j}(x) \diff{x}, \quad x \in \R.
  \end{equation}
  On the other hand, if hypothesis \ref{i:qh:hyp-compat} holds, then we can write
  \begin{align*}
    \vartheta_{i,j}^\circ(\diff{x}) 
    &=  F^+_{i,j}((0,\infty))\Big(q^+_{i,j} F^{+,\circ}_{i,j} \ast \tilde{\chi}{}^{-,\circ}_i(\diff{x}) - \frac{\pi(j)}{\pi(i)} q^-_{j,i} \tilde{F}{}^{-,\circ}_{j,i} \ast \chi^{+,\circ}_j(\diff{x})\Big), \quad x \in \R.
  \end{align*}
   Since $(H^{+,\circ},J^{+,\circ})$ is $\bm{\pi}$-compatible with $(H^{-,\circ},J^{-,\circ})$ and both processes have zero drift vectors and their Lévy measure matrices have no atoms at $0$ it follows with Lemma \ref{lemma: nodrifts} 
   that for any $i \neq  j$ there exists some finite signed measure $\nu^\circ_{i,j}$ with continuous distribution function $\underbars{\nu}^\circ_{i,j}$ once again satisfying 
   \eqref{e:nucirc}.

  Following the arguments from the proof of Theorem \ref{theo: friend drift}, the assumption that $(H^{+},J^{+})$ is $\bm{\pi}$-quasicompatible with $(H^{-},J^-)$ guarantees the existence of a finite signed measure   $\nu^{\rightsquigarrow}_{i,j}$ such that for a.e.\ $x \in \R$,
  \[\vartheta_{i,j}^\rightsquigarrow(\diff{x}) = \underbars{\nu}^\rightsquigarrow_{i,j}(x) \diff{x}.\]
  To conclude, $\nu_{i,j} \coloneqq \nu_{i,j}^\circ + \nu_{i,j}^{\rightsquigarrow}$ is a finite signed measure such that $\vartheta_{i,j}(\diff{x}) = \underbars{\nu}_{i,j}(x) \diff{x}$, i.e.\ \eqref{eq: friend dens} is satisfied, and since $\nu^{\circ}_{i,j}(\{0\}) = 0$ it follows that \eqref{eq: philan atom} being an ML-matrix guarantees that \eqref{eq: friends atom0} holds as well.
  Therefore, $(H^+,J^+)$ is $\bm{\pi}$-compatible with $(H^-,J^-)$.
\end{proof}

\begin{example}\label{ex:double-exp}
  Let irreducible generator matrices $\bm{Q}^+,\bm{Q}^- \in \R^{2\times 2}$ 
  (of unkilled Markov processes),
  a stochastic vector $\bm{\pi} \in (0,1)^2$
  and vectors $\bm{\gamma}^+, \bm{\gamma}^-, \bm{\beta}^+, \bm{\beta}^-,\bm{d}^+, \bm{d}^- \in (0,\infty)^2$ be given such that the following conditions are satisfied: 
  \begin{enumerate}[label = (\roman*), ref =(\roman*)] 
    \item 
      \begin{equation}\label{eq: exp comp1}
        \begin{split}
          \frac{\pi(i)}{\pi(j)}q_{i,j}^-\Big(d^+_i + \frac{\gamma^+_i}{(\beta^+_i)^2} \Big) &= q^+_{j,i}\Big(d^-_j + \frac{\gamma^-_j}{(\beta^-_j)^2}\Big), \quad i,j \in \{1,2\}, i \neq j;
        \end{split}
      \end{equation}
    \item 
      \begin{equation}\label{eq: exp comp2}
        d^+_1 > \frac{\pi(2)}{\pi(1)} \frac{q^+_{2,1}}{q_{1,2}^-} \frac{\gamma^-_2}{(\beta^-_2)^2} \quad \text{and} \quad d^+_2 > \frac{\pi(1)}{\pi(2)} \frac{q^+_{1,2}}{q_{2,1}^-} \frac{\gamma^-_1}{(\beta^-_1)^2}
      \end{equation}
    \item 
      \begin{equation}\label{eq: exp comp3}
        \begin{split}
          0 &< 1 + \frac{\beta^\pm_i d_i^\mp}{q^\mp_{i,j}}-q_{j,i}^\mp \beta^\pm_i \frac{d^\pm_i (\beta^\mp_j)^2 - \tfrac{\pi(j)}{\pi(i)} \tfrac{q^\pm_{j,i}}{q^\mp_{i,j}}\gamma^\mp_j}{d^\pm_i d_j^\mp (\beta^\pm_i)^2(\beta^\mp_j)^2 - \gamma^\mp_j\gamma^\pm_i}, \quad i,j\in \{1,2\}, i \neq j.
        \end{split}
      \end{equation}
  \end{enumerate}
  Note that this can, e.g., easily be achieved by fixing all vectors but $\bm{d}^+,\bm{d}^-$ and then choosing $\bm{d}^+,\bm{d}^-$ large enough such that \eqref{eq: exp comp2} and \eqref{eq: exp comp3} are satisfied, while simultaneosuly ensuring that $\bm{d}^+$, $\bm{d}^-$ solve \eqref{eq: exp comp1}. Let now 
  \[\Pi^{\pm}_i(\diff{x}) = \gamma^\pm_i \mathrm{e}^{-\beta^\pm_i x}\diff{x}, \quad x > 0, i = 1,2,\]
  and $F^\pm_{i,j}$ be measures on $\R_+$ with density $f^\pm_{i,j}$ on $(0,\infty)$ satisfying 
  \begin{equation}\label{sle: exp}
    f^\pm_{i,j}(x) = \frac{\pi(j)}{\pi(i)} \frac{q_{j,i}^\mp}{q^\pm_{i,j} d^\mp_i} F^\mp_{j,i}(\{0\}) \overbar{\Pi}^\pm_j(x) = \frac{\pi(j)}{\pi(i)} \frac{q_{j,i}^\mp}{q^\pm_{i,j} d^\mp_i} F^\mp_{j,i}(\{0\}) \frac{\gamma^\pm_j}{\beta^\pm_j} \mathrm{e}^{-\beta^\pm_j x}, \quad x > 0, i,j \in \{1,2\}, i \neq j.
  \end{equation}
  Integrating \eqref{sle: exp} under the restriction that $F^\pm_{i,j}$ as defined above are probability distributions, we obtain the system of linear equations 
  \[1- F^\pm_{i,j}(\{0\}) = \frac{\pi(j)}{\pi(i)} \frac{q_{j,i}^\mp}{q^\pm_{i,j} d^\mp_i} F^\mp_{j,i}(\{0\}) \frac{\gamma^\pm_j}{(\beta^\pm_j)^2}, \quad i \in \{1,2\}, i \neq j,\]
  for $(F^+_{1,2}(\{0\}),F^+_{2,1}(\{0\}),F^-_{1,2}(\{0\}),F^-_{2,1}(\{0\}))$. Solving the system yields 
  \begin{equation}\label{eq: drift exp}
    F^\pm_{i,j}(\{0\}) = \frac{d^\pm_j d^\mp_i- d^\pm_j \tfrac{\pi(j)}{\pi(i)}\tfrac{q^\mp_{j,i}}{q^\pm_{i,j}}\tfrac{\gamma^\pm_j}{(\beta^\pm_j)^2}}{d^\pm_j d^\mp_i - \tfrac{\gamma^\pm_j \gamma^\mp_i}{(\beta^\pm_j)^2(\beta^\mp_i)^2}}, \quad i,j \in \{1,2\}, i \neq j.
  \end{equation}
  Note that combining \eqref{eq: exp comp1} and \eqref{eq: exp comp2} shows that indeed $F^\pm_{i,j}(\{0\}) \in (0,1)$ and thus $F^\pm_{i,j}$ are probability distributions as desired. Let now $(H^\pm,J^\pm)$ be unkilled MAP subordinators with modulator generator matrices $\bm{Q}^\pm$, Lévy measure matrices 
  \[\bm{\Pi}^\pm = \begin{bmatrix} \Pi^\pm_1 & q^\pm_{1,2} F^\pm_{1,2}\\ q^\pm_{2,1} F^\pm_{2,1} & \Pi^\pm_2\end{bmatrix}\]
  and drifts $\bm{d}^\pm$. 
  We then have 
  \begin{equation}\label{eq: dens exp pos}
    \begin{split}
      &-\bm{\Delta}_{\bm{\pi}}^{-1}\bm{\Psi}^\mp(0)^\top \bm{\Delta}_{\bm{\pi}} \uppartial \bm{\Pi}^\pm(x) - \bm{\Delta}^\mp_{\bm{d}} \uppartial^2 \bm{\Pi}^\pm(x)\\
      &\hspace{1pt}= \left[\scalemath{0.9}{\begin{array}{cc}
            q_{1,2}^\mp \gamma^\pm_1 \bigg(1 + \frac{\beta^\pm_1 d_1^\mp}{q^\mp_{1,2}}-q_{2,1}^\mp\beta^\pm_1 \frac{d^\pm_1 (\beta^\mp_2)^2 - \tfrac{\pi(2)}{\pi(1)} \tfrac{q^\pm_{2,1}}{q^\mp_{1,2}}\gamma^\mp_2}{d^\pm_1 d_2^\mp (\beta^\pm_1)^2(\beta^\mp_2)^2 - \gamma^\mp_2\gamma^\pm_1} \bigg)\mathrm{e}^{-\beta^\pm_1 x} & \frac{\pi(2)}{\pi(1)} q_{1,2}^\mp q_{2,1}^\mp\gamma^\pm_2 \beta^\pm_2 \frac{d^\pm_2 (\beta^\mp_1)^2 - \tfrac{\pi(1)}{\pi(2)} \tfrac{q^\pm_{1,2}}{q^\mp_{2,1}}\gamma^\mp_1}{d_2^\pm d^\mp_1 (\beta^\pm_2)^2(\beta^\mp_1)^2 - \gamma^\mp_1\gamma^\pm_2}\mathrm{e}^{-\beta^\pm_2x}\\[6pt]
            \frac{\pi(1)}{\pi(2)}q_{1,2}^\mp q_{2,1}^\mp\gamma^\pm_1  \beta^\pm_1 \frac{d^\pm_1 (\beta^\mp_2)^2 - \tfrac{\pi(2)}{\pi(1)} \tfrac{q^\pm_{2,1}}{q^\mp_{1,2}}\gamma^\mp_2}{d^\pm_1 d_2^\mp (\beta^\pm_1)^2(\beta^\mp_2)^2 - \gamma^\mp_2\gamma^\pm_1} \mathrm{e}^{-\beta^\pm_1x} & q_{2,1}^\mp \gamma^\pm_2 \bigg(1 + \frac{\beta^\pm_2 d_2^\mp}{q^\mp_{2,1}}-q_{1,2}^\mp \beta^\pm_2\frac{d^\pm_2 (\beta^\mp_1)^2 - \tfrac{\pi(1)}{\pi(2)} \tfrac{q^\pm_{1,2}}{q^\mp_{2,1}}\gamma^\mp_1}{d_2^\pm d^\mp_1 (\beta^\pm_2)^2(\beta^\mp_1)^2 - \gamma^\mp_1\gamma^\pm_2} \bigg)\mathrm{e}^{-\beta^\pm_2 x}
      \end{array}}\right]\\
      &\hspace{1pt}= \begin{bmatrix} \big(q^\mp_{1,2} + (d^\mp_1 - \zeta^\pm_{2,1}) \beta^\pm_1\big)\gamma^\pm_1 \mathrm{e}^{-\beta^\pm_1 x} & \frac{\pi(2)}{\pi(1)} \zeta^\pm_{1,2} \beta^\pm_{2} \gamma^\pm_2 \mathrm{e}^{-\beta^\pm_2 x}\\\frac{\pi(1)}{\pi(2)} \zeta^\pm_{2,1} \beta^\pm_{1} \gamma^\pm_1 \mathrm{e}^{-\beta^\pm_1 x} & \big(q^\mp_{2,1} + (d^\mp_2 - \zeta^\pm_{1,2}) \beta^\pm_2\big)\gamma^\pm_2 \mathrm{e}^{-\beta^\pm_2 x} \end{bmatrix},
    \end{split}
  \end{equation}
  which is nonnegative for any $x > 0$ by \eqref{eq: exp comp1}-\eqref{eq: exp comp3}.  Above, we denoted 
  \[\zeta^\pm_{i,j} = q^\mp_{i,j} q^\mp_{j,i} \frac{d^\pm_j (\beta^\mp_i)^2 - \tfrac{\pi(i)}{\pi(j)} \tfrac{q^\pm_{i,j}}{q^\mp_{j,i}}\gamma^\mp_i}{d^\pm_j d_i^\mp (\beta^\pm_j)^2(\beta^\mp_i)^2 - \gamma^\mp_i\gamma^\pm_j}.\]
  This shows that $(H^+,J^+)$ and $(H^-,J^-)$ are $\bm{\pi}$-fellows.
  Moreover, \eqref{sle: exp} implies \ref{cond philan1} of Definition \ref{def: quasi},
  and a short calculation reveals that the matrices
  $-\bm{\Delta}_{\bm{\pi}}^{-1} (\bm{Q}^{\mp})^\top \bm{\Delta}_{\bm{\pi}} \bm{Q}^\pm$ 
  are generators of unkilled Markov processes for which $\bm{\pi}$ is invariant.
  Also observe that \eqref{eq: exp comp1} together with \eqref{eq: drift exp} shows that \eqref{eq: quasifellow} is satisfied, while the choice \eqref{sle: exp} ensures that condition \ref{cond philan1} of $\bm{\pi}$-quasicompatibility holds. Consequently, $(H^+,J^+)$ and $(H^-,J^-)$ are $\bm{\pi}$-quasicompatible $\bm{\pi}$-fellows. 
  Moreover, it is obvious that $(H^{\pm,\circ},J^{\pm,\circ})$ satisfy the conditions of 
  Proposition \ref{prop: fellow}
  and are thus $\bm{\pi}$-compatible with one another. Theorem \ref{theo: quasi help} thus demonstrates that $(H^+,J^+)$ and $(H^-,J^-)$ are $\bm{\pi}$-friends.

  Let us now calculate the Lévy measure matrix 
  \[\bm{\Pi} = \begin{bmatrix} \Pi_1 & q_{1,2} F_{1,2} \\ q_{2,1} F_{2,1} & \Pi_2 \end{bmatrix}\]
  of the bonding MAP $(\xi,J)$. Plugging into the équations amicales from Theorem \ref{theo: eq amicales} and using the expression for the transitional atoms at $0$ from Proposition \ref{prop: comp cond} in conjunction with \eqref{eq: drift exp} we obtain 
  \begin{align*}
    \Pi_i(\diff{x}) &= \one_{(0,\infty)}(x) \bigg\{\frac{\gamma^-_i}{\beta^+_i+\beta^-_i}\bigg(1+ \frac{\pi(j)}{\pi(i)} \frac{\zeta^+_{j,i} \zeta^-_{j,i}\beta^+_i\beta^-_i}{q^+_{i,j}q^+_{j,i}q^-_{i,j}q^-_{j,i}}\bigg) + q^-_{i,j} + (d^-_i - \zeta^+_{j,i})\beta^+_i\bigg\} \gamma^+_i\mathrm{e}^{-\beta^+_i x} \diff{x}\\
    &\quad + \one_{(-\infty,0)}(x) \bigg\{\frac{\gamma^+_i}{\beta^+_i+\beta^-_i}\bigg(1+ \frac{\pi(j)}{\pi(i)} \frac{\zeta^+_{j,i} \zeta^-_{j,i}\beta^+_i\beta^-_i}{q^+_{i,j}q^+_{j,i}q^-_{i,j}q^-_{j,i}}\bigg) +  q^+_{i,j} + (d^+_i - \zeta^-_{j,i})\beta^-_i\bigg\} \gamma^-_i\mathrm{e}^{-\beta^-_i \lvert x \rvert} \diff{x} 
  \end{align*}
  and 
  \begin{align*}
    q_{i,j} F_{i,j}(\diff{x}) &= \one_{(0,\infty)}(x) \bigg\{\frac{\gamma^-_i}{\beta^+_j+\beta^-_i}\bigg(\frac{\zeta^-_{j,i}\beta^-_i}{q^+_{i,j}q^+_{j,i}}+ \frac{\pi(j)}{\pi(i)}\zeta^+_{i,j}\beta^+_j\bigg(\frac{\beta^+_j + \beta^-_i}{\gamma^-_i}+\frac{1}{q^-_{i,j}q^-_{j,i}}\bigg)\bigg)\bigg\} \gamma^+_j\mathrm{e}^{-\beta^+_j x} \diff{x}\\
    &\quad + \one_{(-\infty,0)}(x) \bigg\{\frac{\gamma^+_j}{\beta^+_j + \beta^-_i}\bigg(\frac{\zeta^+_{i,j}\beta^+_j}{q^-_{i,j}q^-_{j,i}}+ \frac{\pi(i)}{\pi(j)}\zeta^-_{j,i}\beta^-_i\bigg(\frac{\beta^+_j+ \beta^-_i}{\gamma^+_j}+\frac{1}{q^+_{i,j}q^+_{j,i}}\bigg)\bigg)\bigg\} \gamma^-_i\mathrm{e}^{-\beta^-_i \lvert x \rvert} \diff{x}\\
    &\quad + \bigg(\frac{q^-_{i,j}}{q^+_{j,i}}d^+_j \zeta^-_{j,i} + \frac{\pi(j)}{\pi(i)}\frac{q^+_{j,i}}{q^-_{i,j}} d^-_i \zeta^+_{i,j} \bigg)\delta_0(\diff{x}),
  \end{align*}
  for $i,j \in \{1,2\}$ with $i \neq j$. Moreover, we know from \cite[Lemma 3.20]{doering21} that the Lévy components $\xi^{(1)}, \xi^{(2)}$ of $\xi$ have non-trivial Brownian part with scaling factor $\sigma_i^2 = 2 d^+_id^-_i$ for $i =1,2$. Finally, $(\xi,J)$ is unkilled since the same is true for $(H^\pm,J^\pm)$ by construction, see Lemma \ref{lemma: killing}. 

  As explained in Remark~\ref{r:unique-examples}, the Wiener--Hopf factorisation is
  unique in this case.

  The Lévy processes belonging to the bonding MAP $(\xi,J)$ belong to the family of \textit{double exponential jump diffusions} \cite{kou2003} and the transitional jumps form a mixture distribution of a two-sided exponential distribution and a point mass at $0$. These processes can be interpreted as a natural extension of double exponential jump diffusions, which we call \textit{Markov modulated double exponential jump diffusions}. In \cite{kou2003} the overshoot distribution of double exponential jump diffusions is calculated, from which the characteristics of the ascending ladder height process---which is  a subordinator with strictly positive drift and exponentially distributed jumps---can be inferred via overshoot convergence. Our approach therefore allows to go the inverse route for the Markov modulated version by constructing the bonding MAP for a given parametrization of the ascending/descending ladder height processes.
\end{example}
\begin{remark}
  In principle, the above construction can be carried out for arbitrary  ladder height Lévy measures as long as the 
  integrability conditions of Proposition \ref{prop: fellow}
  are satisfied and \eqref{eq: exp comp1}-\eqref{eq: exp comp3} are replaced by appropriate conditions ensuring that the analogue of \eqref{eq: dens exp pos} is nonnegative. As in Example \ref{ex: spec}, promising candidates for this purpose are ladder height Lévy measures with completely monotone densities.
\end{remark}

\section{Uniqueness of the Wiener--Hopf factorisation}\label{sec: unique}
Throughout this section $\bm{\Psi}$ is the exponent of an irreducible MAP with invariant modulating distribution $\bm{\pi}$. Let $\C_+ \coloneqq \{z \in \mathbb{C}: \Re z \geq 0\}$. For a given MAP subordinator $(\xi,J)$, its Laplace exponent is the unique matrix valued function $\bm{\Phi}\colon \C_+ \to \C^{n\times n}$ such that 
\[\E^{0,i}[\exp(-z\xi_1)\,;\, J_1 = j] = (\mathrm{e}^{-\bm{\Phi}(z)})_{i,j}, \quad i,j \in [n], z \in \mathbb{C}_+.\]
Let $\mathcal{A}$ be a class of MAP subordinator Laplace exponents and suppose that $\bm{\Psi}$ has a Wiener--Hopf factorisation in $\mathcal{A}$, i.e., for some $\bm{F},\hat{\bm{F}} \in \mathcal{A}$ it holds
\[\bm{\Psi}(\theta) = -\bm{\Delta}_{\bm{\pi}}^{-1} \hat{\bm{F}}(\mathrm{i}\theta)^\top \bm{\Delta}_{\bm{\pi}} \bm{F}(-\mathrm{i}\theta), \quad \theta \in \R.\]
In our language, this means that $\bm{\Psi}$ is a bonding MAP of two friends belonging to $\mathcal{A}$.
We say that \emph{$\bm{\Psi}$ has a unique MAP Wiener--Hopf factorisation in $\mathcal{A}$}
if, for any other pair of MAP subordinator exponents  $\bm{G},\hat{\bm{G}} \in \mathcal{A}$  such that 
\[\bm{\Psi}(\theta) = -\bm{\Delta}_{\bm{\pi}}^{-1} \hat{\bm{G}}(\mathrm{i}\theta)^\top \bm{\Delta}_{\bm{\pi}} \bm{G}(-\mathrm{i}\theta), \quad \theta \in \R,\]
there exists a diagonal matrix $\bm{\Delta}$ with strictly positive diagonal entries such that
\[\bm{G} = \bm{\Delta} \bm{F} \quad \text{ and } \quad \hat{\bm{G}} = \bm{\Delta}^{-1} \hat{\bm{F}}.\]
In this case, a MAP $(\xi^G,J^G)$ corresponding to $\bm{G}$ is obtained from the MAP $(\xi^F,J^F)$ corresponding to $\bm{F}$ by performing the linear time changes $\xi^{(i),G}_t = \xi^{(i),F}_{\bm{\Delta}_{i,i}t}$ and $\bm{Q}^{G} = \bm{\Delta} \bm{Q}^{F}$. Consequently, if we can prove uniqueness in $\mathcal{A}$ and the ascending, resp.\ descending ladder height MAP exponents $\bm{\Psi}^\pm$ belong to $\mathcal{A}$, it follows that any MAP Wiener--Hopf factorisation in $\mathcal{A}$ represents ascending and descending ladder height processes with different scaling of local times expressed through arbitrary choices of diagonal matrices $\bm{\Delta}$. In other words, the friends giving rise to the bonding MAP $\bm{\Psi}$ carry the probabilistic interpretation of ladder height subordinators. 

It is essential to be able to prove uniqueness of the MAP Wiener--Hopf factorisation
in order to endow $\bm{\pi}$-friendship, which is an essentially analytic condition, 
with probabilistic meaning.
Turning to Lévy processes, uniqueness has long been known in the case of a killed process;
as intimated in \cite[p.~165]{bertoin1996}, threre is a probabilistic proof,
and this is mirrored by an analytic argument based on Liouville's theorem
(see, for example, \cite[Theorem 1(e,f)]{kuznetsov10}.)
However, without imposing further conditions, these techniques do not extend
to the case of an unkilled Lévy process.
This case was finally settled following communication with Mladen Savov, and is still
to be published.

As a consequence, we approach the uniqueness question for MAPs in two ways:
firstly by assuming the MAP in question is killed, in which situation uniqueness
follows relatively straightforwardly; and secondly by assuming not, in which
case we need to impose some additional conditions reflecting the requirements
of Liouville's theorem.

The idea for killed MAPs is to follow Vigon's distributional approach. 
\begin{theorem}\label{theo:unique_killed}
If $\bm{\Psi}$ is killed and has a MAP Wiener--Hopf factorisation in the class of irreducible MAP subordinators, then the factorisation is unique in this class.
\end{theorem}
\begin{proof}
Let $\bm{F},\hat{\bm{F}},\bm{G},\hat{\bm{G}}$ be MAP subordinator Laplace exponents such that 
\[\bm{\Psi}(\theta) = - \bm{\Delta}_{\bm{\pi}}^{-1} \hat{\bm{F}}(\mathrm{i}\theta)^\top \bm{\Delta}_{\bm{\pi}} \bm{F}(-\mathrm{i}\theta) = - \bm{\Delta}_{\bm{\pi}}^{-1} \hat{\bm{G}}(\mathrm{i}\theta)^\top \bm{\Delta}_{\bm{\pi}} \bm{G}(-\mathrm{i}\theta), \quad \theta \in \R.\]
We first note that this implies that $\bm{F},\hat{\bm{F}},\bm{G},\hat{\bm{G}}$ are all killed MAP exponents as well: Since $\bm{\Psi}$ is killed it holds that 
\[\bm{\pi}^\top \bm{\Psi}(0) \one = -\sum_{i=1}^n \pi(i) \dagger_i < 0\] 
whence, 
\[0 > -\bm{\pi}^\top \bm{\Delta}_{\bm{\pi}}^{-1} \hat{\bm{F}}(0)^\top \bm{\Delta}_{\bm{\pi}} \bm{F}(0) \one = -(\hat{\bm{F}}(0)\one)^\top\bm{\Delta}_{\bm{\pi}} \bm{F}(0) \one,\] 
which shows that $\hat{\bm{F}}, \bm{F}$ are killed MAP exponents as well. The statement for $\hat{\bm{G}}, \bm{G}$ follows in the same way.

Since $\bm{\Psi}$ represents an irreducible and killed MAP, $\bm{\Psi}(\theta) \in \mathrm{GL}_n(\CC)$ for any $\theta \in \R$. 
To see this, note that if we denote by $\zeta$ the a.s.\ finite killing time of $(\xi,J)$ we have 
\begin{equation}\label{eq:invert}
\int_0^\infty \big\lvert (\mathrm{e}^{t \bm{\Psi}(\theta)})_{i,j} \rvert \diff{t} = \int_0^\infty \lvert \E^{0,i}[\exp(\mathrm{i}\theta\xi_t)\,;\, J_t = j, t < \zeta] \rvert \diff{t} \leq \int_0^\infty \PP^{0,i}(t < \zeta) \diff{t} = \E^{0,i}[\zeta] < \infty,
\end{equation}
where finiteness of the mean comes from the fact that under $\PP^{0,i}$, $\zeta$ has a phase-type distribution. Thus, the integral $\int_0^\infty \mathrm{e}^{t\bm{\Psi}(\theta)} \diff{t}$ converges absolutely and we obtain from \cite[Lemma A.9]{ivanovs2007} that the maximal real part of the eigenvalues of $\bm{\Psi}(\theta)$ is strictly smaller than $0$, whence $\bm{\Psi}(\theta)$ is invertible.

Consequently, the factorisations imply that also $\bm{G}(z), \hat{\bm{F}}(z) \in \mathrm{GL}_n(\mathbb{C})$ for all $z \in \mathrm{i}\R$. 
We may therefore write
\begin{equation}\label{eq:wh1}
\bm{F}(-\mathrm{i}\theta) \bm{G}(-\mathrm{i}\theta)^{-1} = \bm{\Delta}_{\bm{\pi}}^{-1} \Big(\hat{\bm{G}}(\mathrm{i}\theta)\hat{\bm{F}}(\mathrm{i}\theta)^{-1} \Big)^\top \bm{\Delta}_{\bm{\pi}}, \quad \theta \in \R.
\end{equation}
 Since $\bm{\dagger}^G, \bm{\dagger}^{\hat{F}} \neq \bm{0}$ and $\bm{G}$ and $\hat{\bm{F}}$ are irreducible, the measures $U^G_{i,j}, U^{\hat{F}}_{i,j}$ are finite for any $i,j \in [n]$ and their Fourier transforms are well defined. Performing the calculation in  \eqref{eq:invert} for these MAP subordinators and using \cite[Lemma A.9]{ivanovs2007} then shows   $\mathscr{F} U^G_{i,j}(\theta) = -(\bm{G}(-\mathrm{i}\theta))^{-1}_{i,j}$, $\mathscr{F} U^{\hat{F}}_{i,j}(\theta) = -(\hat{\bm{F}}(-\mathrm{i}\theta))^{-1}_{i,j}$. By uniqueness of the Fourier transform of tempered distributions, this gives the following equalities in $\mathcal{S}^\prime(\R)$: for all $i,j \in [n]$,
\begin{align*}
&\big((\dagger^F_i - q^F_{i,i}) \delta + d^F_i \delta^\prime - \bbGamma \Pi^F_i \big) \ast U^{G}_{i,j}  - \sum_{k \neq j} q_{i,k}^F F^F_{i,k} \ast U^{G}_{k,j} \\
&\, = \frac{\pi(j)}{\pi(i)}\big((\dagger^{\hat{G}}_j - q^{\hat{G}}_{j,j}) \delta - d^{\hat{G}}_j \delta^\prime - \bbGamma \tilde{\Pi}^{\hat{G}}_j \big) \ast \tilde{U}^{\hat{F}}_{j,i} - \frac{\pi(j)}{\pi(i)}\sum_{k \neq j}  q_{j,k}^{\hat{G}} \tilde{F}^{\hat{G}}_{j,k} \ast \tilde{U}^{\hat{F}}_{k,i},
\end{align*}
where all convolutions are well defined since the potential measures $U^G_{i,j}, U^{\hat{F}}_{i,j}$ are finite for all $i,j \in [n].$  
Taking primitives, we obtain 
\begin{equation}\label{eq:wh3}
\eta^{+}_{i,j} + \eta^{-}_{i,j} = c_{i,j},
\end{equation}
where 
\begin{align*}
\eta^{+}_{i,j} &\coloneqq \big((\dagger^F_i - q^F_{i,i}) \one_{\R_+} + d^F_i \delta + \overbar{\Pi}^F_i \big) \ast U^{G}_{i,j}  - \sum_{k \neq j} q_{i,k}^F \underbars{F}^F_{i,k} \ast U^{G}_{k,j},
\end{align*}
and 
\begin{align*} 
\eta^{-}_{i,j} &\coloneqq \frac{\pi(j)}{\pi(i)}\big((\dagger^{\hat{G}}_j - q^{\hat{G}}_{j,j}) \one_{\R_-} + d^{\hat{G}}_j \delta + \overbar{\tilde{\Pi}}{}^{\hat{G}}_j \big) \ast \tilde{U}^{\hat{F}}_{j,i} - \frac{\pi(j)}{\pi(i)}\sum_{k \neq j}  q_{j,k}^{\hat{G}} \underbars{\tilde{F}}{}^{\hat{G}}_{j,k} \ast \tilde{U}^{\hat{F}}_{k,i},
\end{align*}
and $c_{i,j} \in \R$ is some integration constant. 
Note that this implies that the measures $d^F_i U^G_{i,j}$ and $d^{\hat{G}}_i \tilde{U}_{i,j}^{\hat{F}}$ are absolutely continuous and thus both $\eta^\pm_{i,j}$ are induced by a function. Further, since $\eta^+_{i,j}$ is a tempered distribution concentrated on $\R_+$ and $\eta^-_{i,j}$ is a tempered distribution concentrated on $\R_-$, \eqref{eq:wh3} forces 
\begin{equation}\label{eq:wh4}
\eta^+_{i,j} = c_{i,j} \one_{\R_+}
\end{equation}
and
\begin{equation}\label{eq:wh5}
\eta^{-}_{i,j} = c_{i,j} \one_{\R_-}.
\end{equation}
Consequently, taking Laplace transforms on the implied equality of measures $\eta^+_{i,j}(\diff{x}) = \eta^+_{i,j}(x) \diff{x} = c_{i,j}\one_{\R_+}(x) \diff{x}$ yields 
\begin{align*}\frac{c_{i,j}}{\lambda} = \mathscr{L}(\eta^+_{i,j})(\lambda) &= \Big((\dagger^F_i - q^F_{i,i})/\lambda + d^F_i + \int_0^\infty (1- \mathrm{e}^{-\lambda x}) \, \Pi^F_i(\diff{x})/\lambda \Big) \cdot (\bm G(\lambda))^{-1}_{i,j}\\
&\quad - \frac{1}{\lambda}\sum_{k \neq j} q^F_{i,k} \int_0^\infty \mathrm{e}^{-\lambda x} \, F^F_{i,k}(\diff{x}) \cdot (\bm G(\lambda))^{-1}_{k,j}, \quad \lambda > 0,
\end{align*}
where we used invertibility of $\bm{G}(\lambda)$ giving $\mathscr{L}(U_{i,j}^G)(\lambda) = (\bm G(\lambda))^{-1}_{i,j}.$ 
Multiplying both sides of the equality by $\lambda$  yields 
\[\bm{C} \bm{G}(\lambda) = \bm{F}(\lambda), \quad \lambda > 0,\]
for $\bm{C} \coloneqq (c_{i,j})_{i,j \in [n]}$. Lemma \ref{lem:triangular} now implies that $\bm{C} = \bm{\Delta}$ for some diagonal matrix $\bm{\Delta}$ with strictly positive diagonal entries. This implies  $\bm{F}(z) = \bm{\Delta} \bm{G}(z)$ for $z \in \mathrm{i}\R$ and \eqref{eq:wh1} yields 
$\hat{\bm{F}}(z) = \bm{\Delta}^{-1} \hat{\bm{G}}(z)$ for $z \in \mathrm{i}\R$ as well.
\end{proof}

In the non-killed case, the singularity of at least one of the MAP friends at $0$ (which translates to the potential measures being infinite) prevents us from pursuing the same strategy.  Instead, we proceed with a proof that is complex analytic in nature and follows ideas that have been successfully employed in the literature for different kinds of Wiener--Hopf type equations. See, e.g., \cite{kranzer67} or \cite{kuznetsov10}, with the latter essentially dealing with uniqueness of the Wiener--Hopf factorisation of killed Lévy processes. To illustrate the idea in the Lévy case, suppose that we are given two Lévy Wiener--Hopf factorisations 
\[\psi(\theta) = -f(-\mathrm{i}\theta)\hat{f}(\mathrm{i}\theta) = -g(-\mathrm{i}\theta)\hat{g}(\mathrm{i}\theta), \quad \theta \in \R,\]
and let us define a function 
\[h(z) \coloneqq \begin{cases} f(z)/g(z), & z \in \mathbb{C}_+\setminus\{0\} \\ \hat{g}(-z)/\hat{f}(-z), &z \in \mathbb{C}_-\setminus\{0\},\end{cases}\]
where we use the unique analytic extensions of the functions $f,g,\hat{f},\hat{g}$ to $\mathbb{C}_+ \coloneqq \{z \in \mathbb{C}: \Re z \geq 0\}$. Here, dividing by $g$ on $\mathbb{C}_+ \setminus\{0\}$ and $\hat{f}(-\cdot)$ on $\mathbb{C}_- \setminus\{0\}$ is well-defined when the corresponding Lévy processes have non-lattice support due to the following classical result, which can be traced back at least to \cite{wintner36,lukacs56}.

\begin{proposition} \label{prop:lukacs}
  Let $\xi$ be an unkilled Lévy process with characteristic exponent $\psi$. Then, for any $\theta \neq 0$,
  \[\psi(\theta) = 0 \iff \forall t\geq 0: \,\PP\Big(\xi_t \in \frac{2\pi}{\theta} \mathbb{Z}\Big) = 1.\]
\end{proposition}

With this definition of $h$, the goal is to show that $h$ can be extended to an analytic function of $\mathbb{C}$ and to establish sublinear growth of $h$ at $\infty$ such that by the extended Liouville theorem we can conclude that $h$ is in fact equal to some constant $c$, which implies that $f = cg$ and $\hat{f} = c^{-1} \hat{g}$.

To follow such analytic approach in the more general MAP context we must first deal with the open question of invertibility of unkilled characteristic MAP exponents away from $0$, i.e., we want to find a natural analogue to Proposition \ref{prop:lukacs} in the MAP context. A related question was pursued in \cite{ivanovs2010}, where invertibility of the analytic extension of the matrix exponents of spectrally one-sided MAPs (including MAPs with monotone paths) away from the real axis was studied.

\begin{proposition}\label{prop:invertmap}
  Let $(\xi,J)$ be an unkilled MAP with characteristic exponent $\bm{\Psi}$. If $J$ is irreducible and none of the Lévy components has lattice support, then $\bm{\Psi}(\theta) \in \mathrm{GL}_n(\mathbb C)$ for any $\theta \in \R \setminus \{0\}.$
\end{proposition}
\begin{proof}
 We argue by contradiction. Suppose that $\det \bm{\Psi}(\theta) = 0$ for some $\theta \neq 0$. Then $\lambda = 0$ is a left-eigenvalue of $\bm{\Psi}(\theta)$ and there is a left eigenvector $\bm{v} \in \mathbb{C}^n$ such that $\sum_{i=1}^n \lvert v_i  \rvert = 1$. Then, for any $t > 0$, $\bm{v}$ is a left-eigenvector with eigenvalue $\tilde{\lambda} = 1$ for the matrix $\mathrm{e}^{t \bm\Psi(\theta)}$. Hence, 
  \[\forall j \in [n]:\quad v_j = \sum_{i=1}^n v_i \E^{0,i}[\exp(\mathrm{i}\theta \xi_t)\, ;\ J_t = j].\]
  Writing $v_i = \lvert v_i \rvert \mathrm{e}^{\mathrm{i}\varphi_i}$ it follows that 
  \[\forall j \in [n]:\quad  \lvert v_j \rvert  = \sum_{i=1}^n \lvert v_i \rvert \E^{0,i}[\exp(\mathrm{i}(\theta \xi_t + \varphi_i - \varphi_j))\, ;\ J_t = j]\]
  and therefore, by summing over $j$,
  \[1 = \sum_{j=1}^n \sum_{i=1}^n \lvert v_i \rvert \E^{0,i}[\exp(\mathrm{i}(\theta \xi_t + \varphi_i - \varphi_j))\, ;\ J_t = j].\]
  By taking the real part of the right-hand side it follows that
  \[1 =  \sum_{i=1}^n \lvert v_i \rvert \sum_{j=1}^n \E^{0,i}[\cos(\theta \xi_t + \varphi_i - \varphi_j)\, ;\ J_t = j].\]
  Since 
  \[\sum_{j=1}^n \E^{0,i}[\cos(\theta \xi_t + \varphi_i - \varphi_j)\, ;\ J_t = j] \leq \sum_{j=1}^n \PP^{0,i}(J_t = j) \leq 1,\]
  and $\sum_{i=1}^n \lvert v_i \rvert = 1$, it follows that for any $i \in [n]$ such that $v_i \neq 0$ we have
  \[\sum_{j=1}^n \E^{0,i}[\cos(\theta \xi_t + \varphi_i - \varphi_j)\, ;\ J_t = j] = 1.\]
  Pick such $i \in [n]$. Noting that 
  \[\sum_{j=1}^n \E^{0,i}[\cos(\theta \xi_t + \varphi_i - \varphi_j)\, ;\ J_t = j] = \sum_{j=1}^n \E^{0,i}[\cos(\theta \xi_t + \varphi_i - \varphi_j) \, \vert \ J_t = j] \PP^{0,i}(J_t = j),\]
  we now obtain from $\sum_{j=1}^n \PP^{0,i}(J_t = j) = 1$, $\PP^{0,i}(J_t = j) > 0$ and $\E^{0,i}[\cos(\theta \xi_t + \varphi_i - \varphi_j) \, \vert \ J_t = j] \leq 1$ that 
  \[\forall j \in [n]: \quad \quad \E^{0,i}[\cos(\theta \xi_t + \varphi_i - \varphi_j)\, \vert \ J_t = j] = 1.\]
  In particular,
  \[\E^{0,i}[\cos(\theta \xi_t) \ \vert \ J_t  = i] = 1,\]
  which shows that $\xi_t$ is supported on $\tfrac{2\pi}{\theta} \mathbb{Z}$ under $\PP^{0,i}(\cdot \ \vert \ J_t = i)$. Noting that, for $\sigma_1$ denoting the first jump time of $J$, we have $\{\sigma_1 > t, J_0 = i\} \subset \{J_0 = i, J_t = i\}$ and under $\PP^{0,i}$ we have $\xi_t \overset{\mathrm{d}}{=} \xi_t^{(i)}$ on $\{\sigma_1  > t\}$ it follows that for any $t > 0$, $\xi_t^{(i)}$ is supported on the lattice $\tfrac{2\pi}{\theta} \mathbb{Z}.$ Proposition 24.14 in \cite{sato2013} therefore yields that $\xi^{(i)}$ has lattice support.
\end{proof}

Let $\mathcal{A}_0$ be the class of finite mean MAP subordinator Laplace exponents with non-trivial Lévy components and irreducible modulators. Moreover define $\mathcal{A}_1$ to be the class of MAP subordinator Laplace exponents $\bm{\Phi}$ such that 
\begin{align*}
  \forall i \in [n]:& \lim_{\lvert z \rvert \to \infty, \Re z \geq 0} \lvert \phi_i(z) \rvert = \infty\\
  & \quad\vee \, \big(\phi_i \text{ is a compound Poisson Laplace exponent and }  \forall j \in [n]: \bm{\Pi}_{i,j} \ll \mathrm{Leb}\big).
\end{align*}
Note that $\lim_{\lvert z \rvert \to \infty, \Re z \geq 0} \lvert \phi_i(z) \rvert = \infty$ whenever $d_i > 0$ or $\Pi_i$ is non-finite and absolutely continuous, cf.\ Lemma \ref{lem:explosion}.
Moreover, if we define the extremal classes $\mathcal{A}_\infty, \mathcal{A}_{\ll}$ to be the respective families of MAP subordinator Laplace exponents such that for all $i \in [n]$, $\lim_{\lvert z \rvert \to \infty, \Re z \geq 0} \lvert \phi_i(z) \rvert = \infty$ or such that the associated Lévy measure matrices are absolutely continuous, then clearly $\mathcal{A}_\infty \cup \mathcal{A}_{\ll} \subset \mathcal{A}_1$.

\begin{theorem}\label{theo:unique_unkilled}
If $\bm{\Psi}$ has a MAP Wiener--Hopf factorisation in $\mathcal{A}_0 \cap \mathcal{A}_1$, then the factorisation is unique in this class.
\end{theorem}
\begin{proof}
  By Theorem \ref{theo:unique_killed} we only have to deal with unkilled MAPs, i.e., $\bm{\Psi}(0) \notin \mathrm{GL}_n(\mathbb{C})$. Let $\bm{F},\hat{\bm{F}},\bm{G},\hat{\bm{G}}$ be MAP subordinator exponents that all belong to  $\mathcal{A}_0 \cap \mathcal{A}_1$ such that 
  \[\bm{\Psi}(\theta) = - \bm{\Delta}_{\bm{\pi}}^{-1} \hat{\bm{F}}(\mathrm{i}\theta)^\top \bm{\Delta}_{\bm{\pi}} \bm{F}(-\mathrm{i}\theta) = - \bm{\Delta}_{\bm{\pi}}^{-1} \hat{\bm{G}}(\mathrm{i}\theta)^\top \bm{\Delta}_{\bm{\pi}} \bm{G}(-\mathrm{i}\theta), \quad \theta \in \R,\]
  for a given MAP exponent $\bm{\Psi}$. Clearly, the assumptions imply that all Lévy components of $\bm{F},\hat{\bm{F}},\bm{G},\hat{\bm{G}}$ have non-lattice support. Hence, the combined conclusions of Proposition \ref{prop:invertmap} and \cite[Theorem 1]{ivanovs2010} together with Hurwitz' theorem (see also Remark 2.2 and the remarks following Theorem 9 of the same paper) yield that all of $\bm{F}(z),\hat{\bm{F}}(z), \bm{G}(z), \hat{\bm{G}}(z)$ are non-singular in $\CC_+ \setminus \{0\}$. Therefore, $\bm{H}\colon \CC\setminus \{0\} \to \CC^n$  given by 
  \[
    \bm{H}(z)
    =
    \begin{cases}
      \bm{F}(-z)\bm{G}(-z)^{-1}, & \Re z \le 0, \\
      \bm{\Delta}_{\bm{\pi}}^{-1} \big(\hat{\bm{G}}(z)\hat{\bm{F}}(z)^{-1} \big)^\top \bm{\Delta}_{\bm{\pi}}, & \Re z \ge 0,
    \end{cases}
  \]
  is well-defined. Moreover, $\bm{H}$ is holomorphic in $\{z \in \CC: \Re z \neq 0\}$ and continuous on $\{z \in \CC \setminus \{0\} : \Re z = 0\}$. Consequently, by a classical result of Walsh \cite{walsh1933},
  for any $i,j \in [n]$, the integral of $H_{i,j}$ over a rectifiable Jordan curve contained in
  $\{ z \in \CC: \Re z\ge 0, z\ne 0\}$ is zero (and analogously
  in the left half-plane).

  Take any triangle in $\CC \setminus \{0\}$ not enclosing $0$. If it lies in
  one of the half-planes, the integral of $H_{i,j}$ over the triangle is zero.
  If it crosses $\mathrm{i}\R$, then it can be decomposed into two quadrangles,
  each lying in one of the closed half-planes excluding $0$,
  and the integral of $H_{i,j}$ over each of these is zero.

  Hence, the integral of $H_{i,j}$ over any triangle in $\CC \setminus \{0\}$
  is zero, and by Morera's theorem,
  $H_{i,j}$ is holomorphic on $\CC\setminus\{0\}$. Thus, $\bm{H}$ is holomorphic on $\CC \setminus \{0\}$. 

  By Riemann's theorem on removable singularities, $\bm{H}$ can be uniquely extended to
  a holomorphic function on $\CC$ if, and only if, $\lim_{z\to 0} z \bm{H}(z) = \bm{0}_{n \times n}$. 
  Recall that we assumed $\bm{\Psi}(0) \notin \mathrm{GL}_n(\mathbb{C})$ and assume initially that also $(H^G,J^G)$ is unkilled. Observe that if $(U^G_{i,j})_{i,j \in [n]}$ denote the potential measures associated to $\bm{G}$ with $U^G_{i,j}(y) \coloneqq U^G_{i,j}([0,y])$, we have 
  \[\bm{G}(z)^{-1} = \Big(\int_0^\infty \mathrm{e}^{- z y} \, U^G_{i,j}(\diff{y}) \Big)_{i,j \in [n]}, \quad z \in (0,\infty).\]
  Hence, for $z > 0$,
  \begin{equation}\label{eq:riemann}
    z \bm{G}(z)^{-1}_{i,j} = z \int_0^\infty \mathrm{e}^{-zy} \, U^G_{i,j}(\diff{y}) 
    = z^2 \int_0^\infty \mathrm{e}^{-zy} U^G_{i,j}(y) \diff{y}
    = \int_0^\infty \mathrm{e}^{-y} z U^G_{i,j}(y/z) \diff{y}.
  \end{equation}
  By the Markov renewal theorem, see Theorem 28 in \cite{dereich2017}, we have 
  \begin{equation}\label{eq:markov renewal}
    \lim_{z \downarrow 0} z U^G_{i,j}(y/z) = y \lim_{x \to \infty} \frac{U^G_{i,j}(x)}{x} = y\frac{\pi^G(j)}{\E^{0,\bm{\pi}^{\bm{G}}}[H^G_1]},
  \end{equation} 
  where $H^G$ is the ordinator and $\bm{\pi}^{\bm{G}}$ the invariant distribution of the modulator associated to $\bm{G}$. 
  Again, by the Markov renewal theorem, there exits $c, a > 0$ such that for $x > c$, $U^G_{i,j}(x) \leq ax$. Hence, for $z \leq 1$
  \[z U^G_{i,j}(y/z) \leq ay \one_{\{y/z > c\}} + U^G_{i,j}(y/z) \one_{\{y/z \leq c\}} \leq ay + \sup_{x \in [0,c]} U^G_{i,j}(x) = ay + U^G_{i,j}(c) \eqqcolon f_{i,j}(y).\]
  Thus, the function $y \mapsto \mathrm{e}^{-y} f_{i,j}(y)$ is an integrable majorant of $y \mapsto z\mathrm{e}^{-y} U_{i,j}(y\slash z) \diff{y}$ for any $z \leq 1$. We can therefore apply dominated convergence in \eqref{eq:riemann} to obtain with \eqref{eq:markov renewal} that
  \[\lim_{z \downarrow 0} z \bm{G}(z)^{-1}_{i,j} = \frac{\pi^G(j)}{\E^{0,\bm{\pi}^{\bm{G}}}[H^G_1]} \int_0^\infty y \mathrm{e}^{-y} \diff{y} = \frac{\pi^G(j)}{\E^{0,\bm{\pi}^{\bm{G}}}[H^G_1]}.\]
  Since the determinant of a matrix is a polynomial of its entries and the one-sided derivatives 
  \[\lim_{z \to 0, \Re z \geq 0} \frac{\bm{G}_{i,j}(z)- \bm{G}_{i,j}(0)}{z} = \begin{cases} \E[H^{G,(i)}_1], & i =j,\\ q_{i,j}\E[\Delta^G_{i,j}], & i\neq j, \end{cases}\]
  exist and are finite by assumption for all $i,j \in [n]$, it follows that 
  \[\lim_{z \to 0, \Re z \geq 0} \frac{\det \bm{G}(z)}{z} = \lim_{z \to 0, \Re z \geq 0} \frac{\det \bm{G}(z) - \det \bm{G}(0)}{z},\]
  exists and is equal to 
  \[\lim_{z \downarrow 0, z \in \R} \frac{\det \bm G(z)}{z} = -\E^{0,\bm{\pi}^{\bm{G}}}[H_1^G] \prod_{i=1}^{n-1} (-\lambda_i) \in (-\infty,0),\]
  where $\lambda_i$ denote the eigenvalues of $\bm{Q}^{\bm{G}}$ with principal eigenvalue $\lambda_n = 0$, see \cite[Lemma 10]{ivanovs2010}. Using that 
  \[\bm{G}(z)^{-1} = \operatorname{adj}(\bm{G}(z))\slash \det \bm{G}(z), \quad z \in \{z \in \CC \setminus \{0\} : \Re z \geq 0\},\]
  and that $\operatorname{adj}(\bm{G}(\cdot))$ is continuous on $\CC_+$, it therefore follows that $\lim_{z \to 0, \Re z \geq 0} z\bm{G}^{-1}(z)$ exists and is given by 
  \[\lim_{z \to 0, \Re z \geq 0} z\bm{G}(z)^{-1}_{i,j} = \lim_{z \downarrow 0} z \bm{G}(z)^{-1}_{i,j} = \frac{\pi^G(j)}{\E^{0,\bm{\pi}^{\bm{G}}}[H^G_1]}, \quad i,j \in [n].\]
  Note that $-\bm{F}(0)$ is a generator matrix iff none of the Lévy components is killed, which in turn holds iff $\bm{F}(0) \notin \mathrm{GL}_n(\CC)$, see Corollary 1.4 in \cite{stephenson2018}. Thus, if $\bm{F}(0) \notin \mathrm{GL}_n(\CC)$, we have 
  \[\lim_{z \to 0, \Re z \geq 0} z\bm{F}(z)\bm{G}(z)^{-1} = \bm{F}(0) \cdot \one \cdot \bm{\pi}^G/\E^{0,\bm{\pi}^G}[H_1^G] = \bm{0}_{n \times n},\]
  since $\one$ is a right eigenvector for the eigenvalue $\lambda = 0$ of the generator matrix $-\bm{F}(0)$. 
  Thus, we have shown that
  \begin{align*}
    \lim_{z \to 0, \Re z \le 0} z\bm{H}(z)
    & = -\lim_{z \to 0, \Re z \geq 0} z\bm{F}(z) \bm{G}(z)^{-1}\\
    & =\bm{0}_{n \times n}, \quad \text{if }   \bm{F}(0) \notin \mathrm{GL}_n(\CC) \text{ or } \bm{G}(0) \in \mathrm{GL}_n(\CC)
  \end{align*}
  and analogously we find
  \[
    \lim_{z\to 0, \Re z \ge 0} z\bm{H}(z)
    = \bm{0}_{n \times n}, \quad \text{if } \hat{\bm{G}}(0) \notin \mathrm{GL}_n(\CC) \text{ or } \hat{\bm{F}}(0) \in \mathrm{GL}_n(\CC),
  \]
  We now show that the remaining cases cannot occur.
  Suppose initially that $\bm{F}(0) \in \mathrm{GL}_n(\CC) \text{ and } \bm{G}(0) \notin \mathrm{GL}_n(\CC)$. Since $\bm{\Psi}(0) \notin \mathrm{GL}_n(\CC)$, it follows from the Wiener--Hopf factorisation that $\hat{\bm{F}}(0) \notin \mathrm{GL}_n(\CC)$. Hence, using again Lemma 10 in \cite{ivanovs2010}, it follows that
  \[\lim_{z \to 0, \mathrm \Re z = 0} \frac{1}{z} \det \bm{\hat{F}}(z) \neq 0,\]
  with 
  \[\mathrm{sgn} \Big(\lim_{z \to 0, \mathrm \Re z = 0} \frac{1}{z} \det \bm{\hat{F}}(z) \Big) = - \operatorname{sgn}\E^{0,\bm{\pi}^{\hat{\bm{F}}}}[H_1^{\hat{F}}] = -1.\]
  Similarly,
  \[\mathrm{sgn}\lim_{z \to 0, \Re z = 0} \frac{1}{z} \det \bm{G}(-z) = - \mathrm{sgn}\lim_{z \to 0, \Re z = 0} \frac{1}{z} \det \bm{G}(z) = 1.\]
  Thus,
  \begin{equation}\label{eq:contra}
    \begin{split}
      - \operatorname{sgn} (\det \bm{F}(0)) &=  \operatorname{sgn} \lim_{z \to 0, \Re z = 0} \frac{1}{z}\det\Big(\bm{\Delta}_{\bm{\pi}}^{-1} \hat{\bm{F}}(z)^\top \bm{\Delta}_{\bm{\pi}} \bm{F}(-z) \Big)\\
      &=  \operatorname{sgn} \lim_{z \to 0, \Re z = 0} \frac{1}{z}\det\Big(\bm{\Delta}_{\bm{\pi}}^{-1} \hat{\bm{G}}(z)^\top \bm{\Delta}_{\bm{\pi}} \bm{G}(-z) \Big)\\
      &= \operatorname{sgn}(\det \hat{\bm{G}}(0)).
    \end{split}
  \end{equation}
  Depending on whether $-\hat{\bm{G}}(0)$ is a generator matrix or not we have $\operatorname{sgn}(\det \hat{\bm{G}}(0)) \in \{0,1\}$ and since $\bm{F}(0)$ is invertible by assumption, $\operatorname{sgn} (\det \bm{F}(0)) = 1$. To see that this is true, note that the real parts of the eigenvalues of $\bm{Q}^{\bm{F}}$ are non-positive and hence the real parts of the eigenvalues of $-\bm{F}(0) = \bm{Q}^{\bm{F}} - \Delta_{\bm{\dagger}^{\bm{F}}}$, where $\bm\dagger^{\bm{F}}$ has non-negative entries and is not equal to the zero vector by assumption, are strictly negative, see e.g.\ Proposition 1.3 in \cite{stephenson2018}. If an eigenvalue $\lambda_i$ of the real matrix $-\bm{F}(0)$ is of multiplicity $m$ and has non-trivial imaginary part, then $\overbar{\lambda_i}$ is also an eigenvalue of multiplicity $m$ and the product of these $2m$ eigenvalues is strictly positive. Thus, 
  \[\det \bm{F}(0) = (-1)^n \prod_{i=1}^n \lambda_i = \prod_{i=1}^n (-\lambda_i) > 0.\]
  The argument for the determinant of $\hat{\bm{G}}(0)$ is analogous. Hence, \eqref{eq:contra} yields a contradiction. Similarly, it follows that the case $\hat{\bm{G}}(0) \in \mathrm{GL}_n(\CC) \text{ and } \hat{\bm{F}}(0) \notin \mathrm{GL}_n(\CC)$ cannot occur. 

  It follows that $\lim_{z \to 0, \Re z \geq 0} z \bm{H}(z) = \lim_{z \to 0, \Re z \leq 0} z \bm{H}(z) = \bm{0}_{n \times n}$. Riemann's theorem therefore implies that $\bm{H}$ can be extended to a holomorphic function on $\CC$.

  We proceed by showing that $\bm{H}$ has component wise sublinear growth. 
  If $\phi_i^G$ does not diverge at $\infty$, by our assumption that $\bm{G} \in \mathcal{A}_1$, we have $\bm{\Pi}^G_{i,j} \ll \mathrm{Leb}$ for all $j \in [n]$ and  $H^{G,(i)}$ is compound Poisson. Thus, in this case it follows from the Riemann--Lebesgue lemma for Laplace transforms of absolutely continuous measures supported on $(0,\infty)$ that $\lim_{\lvert z \rvert \to \infty, \Re z \geq 0} \mathscr{L}(q_{i,j}F^G_{i,j})(z) = 0$ for all $j \neq i$ and  $\lim_{\lvert z \rvert \to \infty, \Re z \geq 0} \phi^G_i(z) = \lambda^G_i$, where $\lambda^G_i$ is the jump intensity of $H^{G,(i)}$. This demonstrates that for any $i,j \in [n], i \neq j$,
  \[\lim_{\lvert z \rvert \to \infty, \Re z \geq 0} \frac{\bm{G}_{i,j}(z)}{\bm{G}_{i,i}(z)} = 0.\]
  Hence, for given $\varepsilon > 0$ there exists $M > 0$ such that 
  \[\big \lVert \bm{D}^{\bm{G}}(z)^{-1}(\bm{G}(z) - \bm{D}^{\bm{G}}(z))\big\rVert \leq \varepsilon, \quad \lvert z \rvert \geq M, \Re z \geq 0.\]
  Hence, for such $z$
  \begin{align*} 
    \big \lVert \big(\mathbb{I} + \bm{D}^{\bm{G}}(z)^{-1}(\bm{G}(z) - \bm{D}^{\bm{G}}(z))\big)^{-1} - \mathbb{I}\big\rVert &= \Big \lVert \sum_{n=1}^{\infty} \Big(-\bm{D}^{\bm{G}}(z)^{-1}(\bm{G}(z) - \bm{D}^{\bm{G}}(z)) \Big)^n \Big\rVert \\
    &\leq \sum_{n=1}^{\infty} \big \lVert \bm{D}^{\bm{G}}(z)^{-1}(\bm{G}(z) - \bm{D}^{\bm{G}}(z)) \big\rVert^n \\
    &\leq \frac{\varepsilon}{1 - \varepsilon},
  \end{align*}
  and therefore 
  \[\lim_{\lvert z \rvert \to \infty \Re z \geq 0} \big(\mathbb{I} + \bm{D}^{\bm{G}}(z)^{-1}(\bm{G}(z) - \bm{D}^{\bm{G}}(z))\big)^{-1} = \mathbb{I}.\]
  Moreover, since $\bm{G} \in \mathcal{A}_1$ we have
  \[\alpha^G_i \coloneqq \lim_{\lvert z \rvert \to \infty, \Re z \geq 0} (\bm{D}^{\bm{G}}(z))^{-1}_{i,i} = \begin{cases} \frac{1}{\lambda_i^G + \dagger^G_i - q^G_{i,i}}, &\text{ if } H^{G,(i)} \text{ is compound Poisson},\\ 0, &\text{ if } H^{G,(i)} \text{ is not compound Poisson}.
  \end{cases}\]
  Thus, using
  \[\bm{G}(z) = \bm{D}^{\bm{G}}(z)\big(\mathbb{I} + \bm{D}^{\bm{G}}(z)^{-1}(\bm{G}(z) - \bm{D}^{\bm{G}}(z))\big).\]
  we find
  \[\lim_{\lvert z \rvert \to \infty, \Re z \geq 0} \bm{G}(z)^{-1}  = \mathrm{diag}((\alpha^G_i)_{i \in [n]}),\]
  Together with 
  \[\lim_{\lvert z \rvert \to \infty, \Re z \geq 0} \frac{1}{z} \bm{F}(z) = \mathrm{diag}((d^F_i)_{i \in [n]}),\]
  we therefore obtain
  \begin{equation}\label{eq:limatinf}
    \lim_{\lvert z \rvert \to \infty, \Re z \geq 0} \frac{1}{z}\bm{F}(z)\bm{G}(z)^{-1} = \mathrm{diag}((\alpha^G_i d^F_i)_{i \in [n]}).
  \end{equation}
  We now argue by contradiction that $d^F_i > 0$ implies that $H^{G,(i)}$ is not compound Poisson. If $d^F_i > 0$, since by assumption either $H^{\hat{F},(i)}$ is compound Poisson with absolutely continuous Lévy density or $\lim_{\lvert \theta \rvert \to \infty} \lvert \phi_{i,i}^{\hat{F}}(\mathrm{i}\theta) \rvert = \infty$ and, moreover, it always holds $\lim_{\theta \to \infty }\frac{1}{\mathrm{i}\theta} \bm{F}_{i,i}(-\mathrm{i}\theta) = -d^F_i$, it follows from the Riemann--Lebesgue lemma that $\lim_{\theta \to \infty} \tfrac{1}{\mathrm{i}\theta} (\bm{\Delta}_{\bm{\pi}}^{-1}\hat{\bm{F}}(\mathrm{i}\theta)^\top\bm{\Delta}_{\bm{\pi}}\bm{F}(-\mathrm{i}\theta))_{i,i}$ diverges if $H^{\hat{F},(i)}$ is not compound Poisson or else converges to a strictly negative limit. On the other hand, if $H^{G,(i)}$ is compound Poisson, then by definition of $\mathcal{A}_1$, necessarily $\Pi^{G,(i)} \ll \mathrm{Leb}$, such that again by the Riemann--Lebesgue lemma 
  \[\lim_{\theta \to \infty} \tfrac{1}{\mathrm{i}\theta} (\bm{\Delta}_{\bm{\pi}}^{-1}\hat{\bm{G}}(\mathrm{i}\theta)^\top\bm{\Delta}_{\bm{\pi}}\bm{G}(-\mathrm{i}\theta))_{i,i} = \begin{cases} 0, &\text{ if } d^{\hat{G}}_i = 0,\\ d^{\hat{G}}_i (-q_{i,i}^G + \dagger^G_i + \lambda^G_i) > 0, &\text{ if } d^{\hat{G}}_i > 0.\end{cases}\]
  This yields a contradiction and therefore proves that $d^F_i > 0$ implies $\alpha^G_i = 0$. Hence, by \eqref{eq:limatinf}, 
  \[\lim_{\lvert z \rvert \to \infty, \Re z \geq 0} \frac{1}{z}\bm{F}(z)\bm{G}(z)^{-1} = \mathbf{0}_{n \times n}.\]
  In the same way, we can prove that
  \[\lim_{\lvert z \rvert \to \infty, \Re z \geq 0} \frac{1}{z}\bm{\Delta}_{\bm{\pi}}^{-1} \big(\hat{\bm{G}}(z)\hat{\bm{F}}(z)^{-1} \big)^\top \bm{\Delta}_{\bm{\pi}} = \mathbf{0}_{n \times n}.\]
  Thus, under the given assumptions on the MAP Wiener--Hopf factorisation, we have shown that 
  \[\lim_{\lvert z \rvert \to \infty} \frac{1}{z}\bm{H}(z) = \bm{0}_{n \times n}.\]
  Since we have demonstrated that $\bm{H}$ has a unique holomorphic extension to $\mathbb{C}$, the extended Liouville theorem now allows us to conclude that $\bm{H} \equiv \bm{C}$ for some constant matrix $\bm{C} \in \mathbb{C}^{n \times n}$. 

  Finally, Lemma \ref{lem:triangular} shows that $\mathbf{C} = \bm{\Delta}$ for some diagonal matrix $\bm{\Delta}$ with strictly positive diagonal entries. This gives us the desired conclusion by construction of $\bm{H}$.
\end{proof}

From \cite[Lemma 3.20]{doering21} we immediately obtain the following result.
\begin{lemma}\label{lem:gaussian}
  Let $\bm{\Psi}$ be a MAP exponent. For any $i \in [n]$ such that $\psi_i$ has non-trivial Gaussian part, it holds that $\lim_{\lvert z \rvert \to \infty, \Re z \geq 0}\lvert \phi^\pm_i(z) \rvert = \infty$.
\end{lemma}

Absolute continuity of the Lévy measure matrix on $(0,\infty)$ of a bonding MAP together with information on regularity at $0$ of its Lévy components also gives useful properties to determine uniqueness of its Wiener--Hopf factorisation.
\begin{lemma}\label{lem:upward}
  If the Lévy measure matrix is absolutely continuous on $(0,\infty)$ (resp.\ $(-\infty,0)$), then the same is true for the Lévy measure matrix of the ascending (resp.\ descending) ladder height MAP. In this case, if 
  $X^{(i)}$ is  upward (resp.\ downward) regular at $0$, then the Laplace exponent $\phi^+_i$ (resp.\ $\phi^-_i$) diverges at $\infty$. Otherwise, $H^{+,(i)}$ (resp. $H^{-,(i)}$) is compound Poisson.
\end{lemma}
\begin{proof}
  We only deal with the statements on $(H^+,J^+)$, the statement for $(H^-,J^-)$ follows from symmetric arguments. From the construction of the ascending ladder height process $(H^+,J^+)$ it is immediate that $H^{+,(i)}$ is strictly increasing iff $X^{(i)}$ is upward regular. Equivalently, $H^{+,(i)}$ is not compound Poisson iff $X^{(i)}$ is upward regular. Since the Lévy measure matrix $\bm{\Pi}$ is absolutely continuous on $(0,\infty)$, it follows from \cite[Theorem 4.3]{doering21} that $\bm{\Pi}^+$ is absolutely continuous on $(0, \infty)$. In particular, $\Pi^+_i \ll \mathrm{Leb}$ for all $i \in [n]$ and hence $\lim_{\lvert z \rvert \to \infty, \Re z \geq 0} \lvert \phi^+_i(z) \rvert = \infty$ by Lemma \ref{lem:explosion} if $X^{(i)}$ is upward regular.
\end{proof}

The uniqueness results finally allow us to give the following probabilistic interpretation of $\bm{\pi}$-friendships. 

\begin{theorem}\label{theo:unique}
  Let $(H^+,J^+)$ and $(H^-,J^-)$ be irreducible $\bm{\pi}$-friends with matrix Laplace exponents $\bm{\Phi}^\pm$ such that one of the following
  sets of conditions holds:
  \begin{enumerate}[label=(\roman*), ref=(\roman*)]
    \item
      \begin{enumerate}[label=(\alph*), ref=(\roman*)]
        \item
          the bonding MAP is irreducible and killed,
        \item the ascending and descending ladder height processes of the bonding MAP are irreducible, and
        \item 
          $\bm{\pi}$ is invariant for the bonding MAP;
      \end{enumerate}
    \item
      \begin{enumerate}[label=(\alph*), ref=(\roman*)]
        \item 
          the bonding MAP is irreducible and unkilled,
          \label{cond:mapwh3}
        \item
          $\bm{\Phi}^\pm \in \mathcal{A}_0\cap\mathcal{A}_1$, and
        \item 
          the MAP exponents of the ascending and descending ladder height processes of the bonding MAP belong to $\mathcal{A}_0 \cap \mathcal{A}_1$.
          \label{cond:mapwh2}
      \end{enumerate}
  \end{enumerate}
  Then $(H^+,J^+)$ and $(H^-,J^-)$ have the same distribution as the
  ascending and descending ladder height processes of the bonding MAP,
  for an appropriate scaling of local times.
\end{theorem}
\begin{proof}
  \begin{enumerate}[label=(\roman*), ref=(\roman*)]
    \item
      Theorem \ref{c:whf-norm} ensures that the bonding MAP has a Wiener--Hopf factorisation
      \eqref{eq: wienerhopf} in terms of its ladder height processes.
      The result for killed MAPs follows by uniqueness of the Wiener--Hopf factorisation provided by Theorem \ref{theo:unique_killed}.
    \item
      If $J$ is unkilled, Lemma~\ref{lem:unkilled} implies
      that
      $\bm{\pi}$ is invariant for the
      modulator $J$ of the bonding MAP. 
      The Wiener--Hopf factorisation \eqref{eq: wienerhopf} for the bonding
      MAP holds again by Theorem \ref{c:whf-norm}.
      The result for unkilled MAPs then follows from
      the uniqueness result Theorem~\ref{theo:unique_unkilled}. 
  \end{enumerate}
\end{proof}

\begin{remark}\label{r:unique-examples}
  Sufficient conditions for irreducibility of the ladder height modulators and necessary and sufficient conditions for finiteness of the ordinators' mean needed for the ladder height MAPs to belong to $\mathcal{A}_0$ can be found in \cite[Proposition 3.5]{doering21} and \cite[Theorem 35]{dereich2017}, respectively. Lemma \ref{lem:gaussian} and Lemma \ref{lem:upward} give criteria that allow to check whether ladder height MAPs belong to $\mathcal{A}_1$. A characterisation of the lifetime of the bonding MAP in terms of the characteristics of $(H^\pm,J^\pm)$ is given in Lemma \ref{lemma: killing}. 

  We return to our previous examples of friendship. In Example~\ref{ex: spec},
  we found that the bonding MAP had positive Gaussian part in every component,
  which by Lemma~\ref{lem:gaussian} implies that its ladder height processes are in $\mathcal{A}_1$.
  Furthermore, examining the form of the Lévy measure matrix $\bm{\Pi}$
  of the bonding MAP, we note that assumption~\ref{cond: spec0} in the example implies that
  $\E^{0,i}[\lvert \xi_1\rvert]<\infty$.

  Since both the $\bm{\pi}$-friends in the example are unkilled, we obtain, by
  differentiating
  in \eqref{eq: wh map} and using \cite[Propositions~2.13 and~2.15]{ivanovs2007},
  that the bonding MAP $\xi$ oscillates.
  In order to check the finiteness of the mean of the ladder height process,\
  we can look at condition (TO) in
  \cite{dereich2017}, which amounts to showing that 
  $\int^\infty x\overbar{\bm{\Pi}}(x)\diff{x} < \infty$.
  We note that $\int^\infty x\overbar{\bm{\Pi}}_{i,1}(x)\diff{x}$
  can be expressed in terms of an integral over $\mu_1^+$.
  Conditions~\ref{cond: spec0}
  and~\ref{cond: spec2} together imply that $\int_{(0,1)} y^{-4}\, \mu_1^+(\diff{y})<\infty$,
  which in turn yields that $\int^\infty x\overbar{\bm{\Pi}}_{i,1}(x)\diff{x} < \infty$
  for $i=1,2$.
  Symmetrical considerations apply to 
  $\int^\infty x\overbar{\bm{\Pi}}_{i,2}(x)\diff{x}$ for $i=1,2$, and hence
  condition (TO) of \cite{dereich2017} is satisfied.
  It follows that the ladder height process has finite mean.

  Finally, \cite[Proposition~5.10]{doering21} gives that the ladder height processes
  are irreducible, which shows that they are in $\mathcal{A}_0$. Therefore,
  the Wiener--Hopf factorisation in this example consists of the identified pair of
  $\bm{\pi}$-friends.
  The same argument applies, with little variation, to Example~\ref{ex:double-exp}.
\end{remark}

\appendix
\section{Some technical lemmas}
\begin{lemma}\label{lemma: map exp}
  A matrix-valued function $\bm{\Psi}\colon \R \to \mathbb{C}^{n\times n}$ is the characteristic exponent of some $\R \times [n]$-valued MAP if, and only if, all of the following conditions are satisfied:
  \begin{enumerate}[label = (\roman*), ref = (\roman*)]
    \item for all $i \in [n]$, $\bm{\Psi}_{i,i}$ is the characteristic exponent of a killed Lévy process;\label{cond1}
    \item for all $i, j \in [n]$ with $i \neq j$ there exists some finite measure $\rho_{i,j}$ such that $\bm{\Psi}_{i,j} = \mathscr{F} \rho_{i,j}$;
    \item the vector $-\bm{\Psi}(0)\one$ is nonnegative.
      \label{cond3}
  \end{enumerate}
\end{lemma}
\begin{proof}
  Necessity is obvious by definition of a MAP exponent, so assume that \ref{cond1}-\ref{cond3} hold. Let $(a_i,\sigma^2_i,\Pi_i)$ and $\tilde{\dagger}_i$ be the Lévy triplet and killing rate, resp., associated to $\bm{\Psi}_{i,i}$. Let $q_{i,j} = \rho_{i,j}(\R)$. Then, if $q_{i,j} > 0$, $F_{i,j} = \rho_{i,j}/q_{i,j}$ is a probability measure. For $q_{i,j} = 0$ let $F_{i,j} = \delta_0$. Moreover, for
  \[q_{i,i} \coloneqq - \sum_{j \neq i} q_{i,j},\]
  we have 
  \[\dagger_i \coloneqq \tilde{\dagger}_i + q_{i,i} = -\sum_{j=1}^n \bm{\Psi}_{i,j}(0) \geq 0,\]
  by assumption and $\bm{Q} = (q_{i,j})_{i,j = 1,\ldots, n}$ is a generator matrix. Thus, if we let $\bm{G}(\theta) = (\{\mathscr{F} F_{i,j}\}(\theta))_{i,j = 1,\ldots n}$ and $\psi_i$ be the  Lévy--Khintchine exponent corresponding to the triplet $(a_i,\sigma^2_i,\Pi_i)$ and killing rate $\dagger_i$, then 
  \[\bm{\Psi}(\theta) = \mathrm{diag}\big((\psi_i(\theta))_{i \in [n]} \big) + \bm{Q} \odot \bm{G}(\theta), \quad \theta \in \R,\]
  is a characteristic MAP exponent.
\end{proof}

\begin{lemma} \label{lem: dist der}
  Let $\mu$ be a signed measure on $(\R, \mathcal{B}(\R)).$ Then, $\mu$ induces a tempered distribution with $\mu^\prime = \nu$ for some finite signed measure $\nu$ if, and only if, $\mu$ is absolutely continuous with density $\underbars{\nu} +c$, where $c \in \R$ and $\underbars{\nu}(x) = \nu((-\infty,x])$, $x \in \R$.
\end{lemma}
\begin{proof}
  Suppose first that $\mu = \underbars{\nu} + c$. Since $\nu$ is finite, we have $\lvert \nu((-\infty,x]) \rvert \leq \lvert \nu \rvert(\R) < \infty$ and hence $\mu$ induces a tempered distribution. Then, for any $\phi \in \mathcal{S}(\R)$ we have with an integration by parts
  \[\langle \mu^\prime, \phi \rangle = - \int \phi^\prime(x) \, \mu(\diff{x}) = - \int \phi^\prime(x)  \nu((-\infty, x]) \diff{x} = \int \phi(x) \, \nu(\diff{x}),\]
  which shows $\mu^\prime = \nu$. Conversely, if $\mu^\prime = \nu$, it follows that $\mu = \underbars{\nu} + c $ in the sense of distributions for some constant $c \in \R$ since $\underbars{\nu}^\prime = \nu.$ 
\end{proof}

\begin{lemma}{\cite[Lemme 1.5.5]{vigondiss}} \label{lem:asymp real}
  Suppose that $X$ has absolutely continuous Lévy measure. 
  \begin{enumerate}[label=(\roman*), ref =(\roman*)]
    \item If $X$ is not compound Poisson, then $\lim_{\lvert \theta \rvert \to \infty} \Re \psi(\theta) = - \infty.$ 
    \item If $X$ is compound Poisson, then $\lim_{\lvert \theta \rvert \to \infty} \psi(\theta) = - \Pi(\R) - \dagger.$ 
  \end{enumerate}
\end{lemma}

The same idea used to obtain Lemma \ref{lem:asymp real} allows us to prove the following result.
\begin{lemma}\label{lem:explosion}
  Let $X$ be a Lévy subordinator that is not compound Poisson with Laplace exponent $\phi$. Then
  \[\lim_{\lvert z \rvert \to \infty, \Re z \geq 0} \lvert \phi(z) \rvert = \infty \iff \text{for some } q> 0, \lim_{\lvert z \rvert \to \infty, \Re z \geq 0} \mathscr{L}U_q^{\mathrm{sing}}(z) = 0,\]
  where $U^{\mathrm{sing}}_q$ denotes the continuous singular part of the resolvent measure $U_q$ of $X$.
\end{lemma}
\begin{proof}
  Since $X$ is not compound Poisson, $U_q$ is a continuous measure \cite[Proposition I.15]{bertoin1996} and hence $U_q = U_q^{\mathrm{cont}} + U_q^{\mathrm{sing}}$, where $U_q^{\mathrm{cont}}(\diff{x}) = u_q(x) \diff{x}$ for some $L^1$-density $u_q$. Since $\Re \phi(z) \geq 0$ it follows that
  \[\mathscr{L}U^{\mathrm{cont}}_q(z) + \mathscr{L}U^{\mathrm{sing}}_q(z) = \mathscr{L}U_q(z) = \int_0^\infty \mathrm{e}^{(-q - \phi(z)) t} \diff{t} = \frac{1}{q+ \phi(z)}, \quad z \in \CC_+.\]
  By the Riemann--Lebesgue lemma for Laplace transforms of finite, absolutely continuous measures supported on $(0,\infty)$, we have $\lim_{\lvert z \rvert \to \infty, \Re z \geq 0} \mathscr{L}U_q^{\mathrm{cont}}(z) = 0$, which implies that
  \[\lim_{\lvert z \rvert \to \infty, \Re z \geq 0} \mathscr{L}U_q^{\mathrm{sing}}(z) = 0\]
  if, and only if,
  \[\lim_{\lvert z \rvert \to \infty, \Re z \geq 0} \lvert \phi(z) \rvert = \infty.\]
\end{proof}

A natural criterion for divergence of $\lvert \phi \rvert$ at $\infty$ is therefore absolute continuity of $U_q$, which is guaranteed whenever $X$ has strictly positive drift or $\PP_{X_t} \ll \mathrm{Leb}$ for all $t > 0$. A convenient sufficient criterion for the latter to hold is $\Pi \ll \mathrm{Leb}$, see \cite[Theorem 27.7]{sato2013}, as already indicated by Lemma \ref{lem:asymp real}.

\begin{lemma}\label{lem:triangular}
  Let $\bm{\Phi}$ be the Laplace exponent of a MAP subordinator $(\xi,J)$ such that $\bm{\Phi}_{i,i}(\lambda) > 0$ for all $\lambda > 0$ and $i \in [n]$ and let $\bm C \in \R^{n \times n}$ be a matrix such that $\bm{C}\bm{\Phi}$ is again a Laplace exponent of a MAP subordinator. Then, $\bm{C}$ is diagonal with positive diagonal entries.
\end{lemma} 
\begin{proof} 
  Suppose that there exist $i \neq j$ such that $c_{i,j} \neq 0$. Letting $\tilde{\bm{\Phi}} = \bm{C\Phi}$ it follows that 
  \[-q_{j,j}^\Phi + \phi^\Phi_j(\lambda) = \frac{1}{c_{i,j}}\Big(-q^{\tilde{\Phi}}_{i,j}\mathscr{L}\big\{\Delta^{\tilde{\Phi}}_{i,j}\big\}(\lambda) + \sum_{k \neq j}c_{i,k}q^{\Phi}_{k,j} \mathscr{L}\big\{\Delta^\Phi_{k,j} \big\}(\lambda) \Big), \quad \lambda \geq 0.\]
  Taking $\lambda \to \infty$, if $\xi^{(j)}$ is not compound Poisson, the LHS converges to $\infty$, otherwise it converges to $-q_{j,j}^\Phi + \dagger^\Phi_j + \Pi_j^\Phi(\R_+) > 0$, where the strict inequality follows from our assumption on $\bm{\Phi}$. On the other hand, for any finite measure $\mu$ concentrated on $\R_+$, $\mathscr{L}\{\mu\}(\lambda) \to 0$ as $\lambda \to \infty$ by monotone convergence. Thus, the RHS converges to $0$ as $\lambda \to \infty$, which yields a contradiction. Thus $\bm{C}$ is diagonal. Since the diagonal entries of $\tilde{\bm{\Phi}}(\lambda)$ must be nonnegative and the diagonal entries of $\bm{\Phi}(\lambda)$ are strictly positive by assumption, it follows that the diagonal of $\bm{C}$ has nonnegative entries.
\end{proof}
\begin{remark}
  The restriction on $\bm{\Phi}$ is not a serious one since it is satisfied for any MAP subordinator with non-trivial Lévy components or, e.g., irreducible modulator $J$. It also holds whenever $\bm{\Phi}(\lambda)$ is invertible for some $\lambda > 0$.
\end{remark}

\section{General Wiener--Hopf factorisation}
\label{s:whf}

This section is dedicated to a proof of the Wiener--Hopf factorisation
of MAPs. It extends the results of \citet{dereich2017}, which dealt
with situations where every component of the MAP was killed at the same
(possibly zero) rate, and \citet{ivanovs17}, where every component
of the MAP was killed at strictly positive rate.
In what follows, we adopt Ivanovs' notation, whereby,
when $T$ is a
(possibly random) time, we write 
$\bm{T} = \Bigl( \int_0^T \one_{\{J_t = i\}} \, \diff{t} : i \in [n]\Bigr)$,
and for a functional $F_T$, declare 
$\E[F_T; J_T]$ to be the matrix whose $(i,j)$-th
entry is $\E^{0,i}[F_T; J_T = j]$.

Write $\PP_*$ for the probabilities associated with an unkilled version
of $(\xi,J)$; that is, a version whose exponent is 
$\theta \mapsto \bm{\Psi}(\theta) + \bm{\Delta}_{\bm{\dag}}$.
We define $\bm{\kappa}$ as the matrix exponent of the ascending ladder
process
under $\PP_*$, in the sense that 
$\E_*[\mathrm{e}^{-\langle \bm{\gamma},\bm{L}^{-1}_t\rangle - \alpha H^+_t}; J^+]
= \mathrm{e}^{-\bm{\kappa}(\bm{\gamma},\alpha)}$;
note the different convention in comparison with $\bm{\Psi}$.
We can also define the descending ladder process and its
matrix exponent $\hat{\bm{\kappa}}$, by taking the dual
of $(\xi,J)$. However, as alluded to in section~\ref{sec: map intro},
for this purpose a slightly different choice of local time is needed,
partly in case of compound Poisson components
and partly due to the effects of state changes;
this is done carefully by \citet{ivanovs17}.

\begin{theorem}\label{t:whf-gen}
  There exists some vector $\bm{c}$ with positive entries such that
  \[
    -(\bm{\Psi}(\theta) - \bm{\Delta}_{\bm{\beta}})
    =
    \bm{\Delta}_{\bm{\pi}}^{-1}
    \hat{\bm{\kappa}}(\bm{\dag} + \bm{\beta}, \iu\theta)^\top
    \bm{\Delta}_{\bm{\pi}}
    \bm{\Delta}_{\bm{c}}
    \bm{\kappa}(\bm{\dag}+\bm{\beta},-\iu\theta).
  \]
\end{theorem}
\begin{proof}
  Let $\bm{\beta}$ be a vector with non-negative entries.
  We will need to initially consider MAPs with positive killing rate in every
  state. For this reason, let us start by
  replacing the rate $\dag_i$ with $\dag_i^\epsilon\coloneqq \dag_i+\epsilon$;
  that is, we consider
  the MAP with exponent $\theta \mapsto \bm{\Psi}(\theta) - \bm{\Delta}_{\epsilon}$.
  Following \cite{ivanovs17}, let
  \begin{align*}
    \overbar{\xi}_t &= \sup\{ \xi_s : s\le t\}, \\
    \overbar{G}(t) & = \inf\{ s\le t : \xi_s\vee \xi_{s-} = \overbar{\xi}_t\},
    \text{ and} \\
    \overbar{J}_t &= J_{\overbar{G}(t) -} \one_{\{\xi_{\overbar{G}(t) -} = \overbar{\xi}_t\}}
    + J_{\overbar{G}(t)} \one_{\{\xi_{\overbar{G}(t)-} < \overbar{\xi}_t\}}.
  \end{align*}
  When $\overbar{J}_t = j$ for some state $j$ in which $\xi^{(j)}$ is irregular
  for $(0,\infty)$ and regular for $(-\infty,0]$,
  $\overbar{\xi}_t = \xi_{\overbar{G}(t)}$;
  otherwise, $\overbar{\xi}_t = \xi_{\overbar{G}(t)-}$.
  Assume initially that we are in the latter case, and consider the following
  calculation from excursion theory, in which $a_j$ is the drift of $L^{-1}$
  in state $j$, and $n_j$ is the excursion measure of $(\xi,J)$ away
  from its maximum when starting in state $j$:
  \begin{align*}
    \E^{0,i}\bigl[ 
      \mathrm{e}^{\iu\theta \overbar{\xi}_{\zeta-} - \langle \bm{\beta},\overbar{\bm{G}}(\zeta-) \rangle} ;
      \overbar{J}_\zeta = j
    \bigr]
    &= \E_*^{0,i} \biggl[
      \int_0^\infty \Bigl( \sum\nolimits_{k\in [n]} \dag_k^\epsilon \one_{\{J_t = k\}} \Bigr)
      \mathrm{e}^{-\langle \bm{\dag}^\epsilon, \bm{t}\rangle + \iu\theta \xi_{\overbar{G}(t)-}
      - \langle \bm{\beta}, \overbar{\bm{G}}(t)\rangle}
      \one_{\{J_{\overbar{G}(t)-} = j\}}
      \, \diff{t}
    \biggr]
    \\
    &= \E_*^{0,i} \biggl[
      \int_0^\infty \one_{\{\overbar{\xi}_t = \xi_t\}} \dag^\epsilon_j
      \mathrm{e}^{-\langle \bm{\dag}^\epsilon + \bm{\beta}, \bm{t}\rangle + \iu\theta \xi_{t}}
      \one_{\{J_{t} = j\}}
      \, \diff{t}
    \biggr]
    \\
    &\quad {} 
    + \E_*^{0,i} \biggl[
      \sum\nolimits_{g}
      \int_g^d \Bigl( \sum\nolimits_{k} \dag_k^\epsilon \one_{\{J_t = k\}} \Bigr)
      \mathrm{e}^{-\langle \bm{\dag}^\epsilon, \bm{t}\rangle + \iu\theta \xi_{g-}
      - \langle \bm{\beta}, \bm{g}\rangle}
      \one_{\{J_{g-} = j\}}
      \, \diff{t}
    \biggr]
    \\
    &= \E_*^{0,i} \biggl[
      \int_0^\infty a_j \dag^\epsilon_j
      \mathrm{e}^{-\langle \bm{\dag}^\epsilon + \bm{\beta}, \bm{t}\rangle + \iu\theta \xi_{t}}
      \one_{\{J_{t} = j\}}
      \, \diff{L_t}
    \biggr]
    \\
    &\quad {} 
    + \E_*^{0,i} \biggl[
      \int_0^\infty
      \mathrm{e}^{-\langle \bm{\dag}^\epsilon+\bm{\beta},\bm{t}\rangle + \iu\theta \xi_{t-}}
      \one_{\{J_{t-} = j\}}
      \, \diff{L_t}
    \biggr]
    n_j \biggl(
      \int_0^{\overbar{\zeta}} 
      \Bigl( \sum\nolimits_{k} \dag_k^\epsilon \one_{\{J_t = k\}} \Bigr)
      \mathrm{e}^{-\langle \bm{\dag}^\epsilon, \bm{t}\rangle}
      \, \diff{t}
    \biggr)
    \\
    &=
    \E_*^{0,i}
    \biggl[
      \int_0^\infty
      \mathrm{e}^{-\langle \bm{\dag}^\epsilon+\bm{\beta}, \bm{L}^{-1}_t\rangle + \iu\theta H^+_t}
      \one_{\{J^+_t = j\}}
      \,\diff{t}
    \biggr]
    \Bigl( a_j\dag^\epsilon_j
      + n_j\bigl( 1-\mathrm{e}^{-\langle \bm{\dag}^\epsilon,\overbar{\bm{\zeta}}\rangle}\bigr)
    \Bigr)
    \\
    &= \bm{\kappa}(\bm{\dag}^\epsilon+\bm{\beta}, -\iu\theta)^{-1}_{i,j}
    (\bm{\kappa}(\bm{\dag}^\epsilon,0)\one)_{j}
  \end{align*}
  where in the third line the sum is over excursion intervals $(g,d)$,
  and in the fourth line $\overbar{\zeta}$ is the lifetime of the excursion.
  In other words,
  \begin{equation} \label{e:max}
    \E\bigl[ \mathrm{e}^{\iu\theta \overbar{\xi}_{\zeta-}
    -\langle \bm{\beta},\overbar{G}(\zeta-)\rangle} ; \overbar{J}_\zeta\bigr]
    = \bm{\kappa}(\bm{\dag}^\epsilon+\bm{\beta}, -\iu\theta)^{-1}
    \bm{\Delta}_{\bm{\kappa}(\bm{\dag}^\epsilon,0)\one}.
  \end{equation}
  We now consider the case where the state $j$ above is one for which
  $\xi^{(j)}$ is irregular for $(0,\infty)$ and regular for $(-\infty,0]$.
  In this situation, the point $0$ may be either a holding point for
  $\xi^{(j)}$ reflected in its supremum (if $\xi^{(j)}$ is compound Poisson)
  or an irregular point (otherwise.) Either way, let us define
  $T_0 = 0$, $S_m = \inf\{t\ge T_{m-1} : \xi_t < \overbar{\xi}_t; J_t = j\}$,
  and $T_m = \inf\{t\ge S_m : \xi_t = \overbar{\xi}_t\}$,
  for $n\ge 1$. This implies that $\{(S_m,T_m) : m\ge 1\}$ is the set
  of excursions away from the maximum that start from state $j$, 
  and importantly, every $S_m$ and $T_m$ is a stopping time.
  In the case of an irregular point, $T_{m-1} = S_m$.
  Then, we can compute as follows, using the Markov property:
  \begin{align*}
    \E^{0,i}\bigl[ 
      \mathrm{e}^{\iu\theta \overbar{\xi}_{\zeta-} 
      - \langle \bm{\beta},\overbar{\bm{G}}(\zeta-) \rangle} ;
      \overbar{J}_\zeta = j
    \bigr]
    &=
    \E^{0,i}\Bigl[ 
      \sum\nolimits_{m\ge 1}
      \mathrm{e}^{\iu\theta \xi_{S_m}
      - \langle \bm{\beta},\bm{S_m} \rangle}
      \one_{\{J_{S_m} = j, S_m < \zeta \le T_m \}}
    \Bigr]
    \\
    &=
    \E_*^{0,i}\biggl[
      \sum\nolimits_{m\ge 1}
      \mathrm{e}^{\iu\theta \xi_{S_m} - \langle \bm{\beta},\bm{S_m}\rangle}
      \one_{\{J_{S_m} = j\}}
      \int_{S_m}^{T_m}
      \Bigl( \sum\nolimits_{k \in [n]} \dag^\epsilon_k \one_{\{J_t = k\}} \Bigr)
      \mathrm{e}^{-\langle \bm{\dag}^\epsilon, \bm{t}\rangle}
      \, \diff{t}
    \biggr]
    \\
    &=
    \E_*^{0,i}\Bigl[
      \sum\nolimits_{m\ge 1}
      \mathrm{e}^{\iu\theta \xi_{S_m} - \langle \bm{\beta}+\bm{\dag}^\epsilon, \bm{S_m} \rangle}
      \one_{\{J_{S_m} = j\}}
    \Bigr]
    \E_*^{0,j}
    \int_0^{T_1}
    \Bigl( \sum\nolimits_k \dag^\epsilon_k \one_{\{J_t = k\}} \Bigr)
    \mathrm{e}^{-\langle \bm{\dag}^\epsilon,\bm{t}\rangle}
    \, \diff{t}
    \\
    &=
    \E_*^{0,i}
    \biggl[
      \int_0^\infty
      \mathrm{e}^{-\langle \bm{\beta} + \bm{\dag}^\epsilon, \bm{L}^{-1}_t\rangle + \iu\theta H^+_t} ;
      J^+_t = j
    \biggr]
    \E_*^{0,j}\bigl[ 1-\mathrm{e}^{-\langle \bm{\dag}^\epsilon, \bm{T}_1\rangle} \bigr]
    \\
    &=
    \bm{\kappa}(\bm{\beta}+\bm{\dag}^\epsilon,-\iu\theta)^{-1}_{i,j}
    (\bm{\kappa}(\bm{\dag}^\epsilon,0)\one)_{j},
  \end{align*}
  which implies that \eqref{e:max} holds in all cases.

  Naturally, the same applies to the minimum of the process.
  Using \cite[Corollary~5.1]{ivanovs17}, we obtain
  \begin{equation}\label{e:beta-whf-inv}
    -(\bm{\Psi}(\theta) - \bm{\Delta}_{\bm{\beta}} - \bm{\Delta}_{\epsilon})^{-1}
    \bm{\Delta}_{\bm{\dag}^\epsilon}
    = \bm{\kappa}(\bm{\dag}^\epsilon+\bm{\beta}, -\iu\theta)^{-1}
    \bm{\Delta}_{\bm{\kappa}(\bm{\dag}^\epsilon,0)\one}
    \bm{\Delta}_{\underbars{\bm{c}}^\epsilon}^{-1}
    \bm{\Delta}_{\hat{\bm{\kappa}}(\bm{\dag}^\epsilon,0)\one}
    (\hat{\bm{\kappa}}(\bm{\dag}^\epsilon+\bm{\beta}, \iu\theta)^{-1})^\top
    \bm{\Delta}_{\bm{\pi}}
    \bm{\Delta}_{\bm{\dag}^\epsilon},
  \end{equation}
  where
  $\underbars{\bm{c}}^\epsilon 
  = \bm{\Delta}_{\bm{\kappa}(\bm{q}^\epsilon,0)\one} 
  (\bm{\kappa}(\bm{q}^\epsilon,0)^{-1})^\top \bm{\Delta}_{\bm{\pi}}\bm{\dag}^\epsilon$.
  Simplifying this yields
  \begin{equation}\label{e:beta-whf}
    -(\bm{\Psi}(\theta) - \bm{\Delta}_{\bm{\beta}} - \bm{\Delta}_{\epsilon})
    =
    \bm{\Delta}_{\bm{\pi}}^{-1}
    \hat{\bm{\kappa}}(\bm{\dag}^\epsilon + \bm{\beta}, \iu\theta)^\top
    \bm{\Delta}_{\hat{\bm{\kappa}}(\bm{\dag}^\epsilon,0)\one}^{-1}
    \bm{\Delta}_{\underbars{\bm{c}}^{\epsilon}}
    \bm{\Delta}_{\bm{\kappa}(\bm{\dag}^\epsilon,0)\one}^{-1}
    \bm{\kappa}(\bm{\dag}^\epsilon+\bm{\beta},-\iu\theta).
  \end{equation}
  Now,
  \begin{align*}
    b_i^\epsilon
    \coloneqq
    \bigl(
      \bm{\Delta}_{\hat{\bm{\kappa}}(\bm{\dag}^\epsilon,0)\one}^{-1}
      \bm{\Delta}_{\underbars{\bm{c}}^{\epsilon}}
      \bm{\Delta}_{\bm{\kappa}(\bm{\dag}^\epsilon,0)\one}^{-1}
    \bigr)_{i,i}
    &= \frac{\sum_{k\in[n]} \bm{\kappa}(\bm{\dag}^\epsilon, 0)^{-1}_{k,i}
    \pi(k) \dag^\epsilon_k}
    {(\hat{\bm{\kappa}}(\bm{\dag}^\epsilon,0)\one)_{i}},
  \end{align*}
  and we note that $b_i^\epsilon>0$ for all $i$ and all $\epsilon>0$.
  Assume for the moment that $\beta_i>0$ for all $i$.
  Element $(i,i)$ of \eqref{e:beta-whf-inv} at $\theta=0$ provides that
  \[
    -(\bm{\Psi}(0) - \bm{\Delta}_{\bm{\beta}} - \bm{\Delta}_{\epsilon})^{-1}_{i,i}
    =
    \sum\nolimits_{k\in[n]}
    \bm{\kappa}(\bm{\dag}^\epsilon+\bm{\beta},0)^{-1}_{i,k}
    \hat{\bm{\kappa}}(\bm{\dag}^\epsilon+\bm{\beta},0)^{-1}_{i,k} 
    \frac{\pi(i)}{b_k^{\epsilon}}.
  \]
  The left-hand side is an element of the resolvent matrix of the irreducible Markov
  process $J$ with additional killing, and therefore converges, as $\epsilon\to 0$, to
  a positive limit.
  On the right-hand side, we note that for every $i,k\in[n]$,
  $\lim_{\epsilon\to 0} \bm{\kappa}(\bm{\dag}^\epsilon+\bm{\beta},0)^{-1}_{i,k}
  = \bm{\kappa}(\bm{\dag}+\bm{\beta},0)^{-1}_{i,k}
  = \int \mathrm{e}^{-\langle \bm{\dag}+\bm{\beta}, \bm{x}\rangle}\, U_{i,k}(\diff{\bm{x}}) \ge 0$,
  where $U$ is the potential measure of $\bm{L}^{-1}$ under $\PP_*$,
  and likewise for $\hat{\bm{\kappa}}$.
  Specifically, when $k=i$,
  \[
    \lim_{\epsilon\to 0} \bm{\kappa}(\bm{\dag}^\epsilon + \bm{\beta},0)^{-1}_{i,i}
    = \int \mathrm{e}^{-\langle \bm{\dag}+\bm{\beta}, \bm{x}\rangle}
    \, U_{i,i}(\diff{\bm{x}}) > 0,
  \]
  and likewise
  $\lim_{\epsilon\to 0} \hat{\bm{\kappa}}(\bm{\dag}^\epsilon+\bm{\beta},0)^{-1}_{i,i} > 0$.
  (The reader is correct to be suspicious here: 
  despite the irreducibility of $J$, a problem
  seems to occur when
  $\xi^{(i)}$ is a strictly decreasing
  L\'evy process and transitional jumps into state $i$ are negative, since $(\xi,J)$
  can never achieve a maximum in state $i$, and
  this indeed implies
  that $U_{j,i} = 0$
  for all $j\ne i$. However, because of the way the local time is defined for states
  such as $i$ in which $0$ is irregular for the process reflected in its maximum,
  $J^+_t$ actually spends positive time in state $i$
  under $\PP_*^{0,i}$, and so $U_{i,i} \ne 0$.)
  From these considerations, we see that there exists 
  $b_i \coloneqq \lim_{\epsilon \to 0} b_i^\epsilon \in (0,\infty)$, for every $i$,
  and moreover that $b_i$ does not depend on $\bm{\beta}$, so we may drop
  our assumption that $\bm{\beta}$ has positive entries.

  Finally, we let $c_i = b_i/\pi_i$, and \eqref{e:beta-whf} can be rewritten as
  \[
    -(\bm{\Psi}(\theta) - \bm{\Delta}_{\bm{\beta}})
    =
    \bm{\Delta}_{\bm{\pi}}^{-1}
    \hat{\bm{\kappa}}(\bm{\dag} + \bm{\beta}, \iu\theta)^\top
    \bm{\Delta}_{\bm{\pi}}
    \bm{\Delta}_{\bm{c}}
    \bm{\kappa}(\bm{\dag}+\bm{\beta},-\iu\theta),
  \]
  which completes the proof.
\end{proof}

We note that the constant $\bm{c}$ depends on the killing rate $\bm{\dag}$
(as well as, of course, the law of $\xi$ under $\PP_*$).
For Lévy processes, the dependence on the killing rate is known
\cite[equation (A.3)]{PP-bernstein-gamma} but, lacking a Fristedt formula
for MAPs, we leave it in the implicit form appearing in the proof.

  \begin{proof}[Proof of Theorem~\ref{c:whf-norm}]
  Changing the normalisation of the local times amounts to multiplying the
  exponents $\bm{\kappa}$ and $\hat{\bm{\kappa}}$ on the left by diagonal
  matrices containing positive entries, so setting $\bm{\beta} = \bm{0}$ and
  choosing the normalisation appropriately in the preceding theorem
  leads to the equation
  \[
    -\bm{\Psi}(\theta)
    = \bm{\Delta}_{\bm{\pi}}^{-1}
    \hat{\bm{\kappa}}(\bm{\dag},\iu\theta)^\top
    \bm{\Delta}_{\bm{\pi}}
    \bm{\kappa}(\bm{\dag},-\iu\theta).
  \]
  Finally, we must identify the matrix exponents appearing
  here, which can be done as follows:
  \begin{align*}
    \mathrm{e}^{-t\bm{\kappa}(\bm{\dag},-\iu\theta)}
    = \E_*\bigl[\mathrm{e}^{\iu\theta H^+_t -\langle \bm{\dag}, \bm{L}^{-1}_t \rangle}; J^+_t \bigr]
    = \E\bigl[ \mathrm{e}^{\iu\theta \xi_{L^{-1}_t}} \one_{\{\zeta > L^{-1}_t\}}; J_{L^{-1}_t} \bigr]
    = \mathrm{e}^{t\bm{\Psi}^+(\theta)}.
  \end{align*}
  The exponent $\hat{\bm{\kappa}}$ can be identified similarly,
  and this completes the proof.
\end{proof}

\paragraph{Acknowledgements} 
LT gratefully acknowledges financial support of Carlsberg Foundation Young Researcher Fellowship grant CF20-0640 ``Exploring the potential of nonparametric modelling of complex systems via SPDEs''.

\printbibliography
\end{document}